\documentclass{amsart}
\usepackage[foot]{amsaddr}
\usepackage[margin=1in]{geometry}
\usepackage{graphicx}
\usepackage{epstopdf, epsfig}
\usepackage{amsmath,amsfonts, amssymb}
\usepackage{mathrsfs}
\usepackage{xcolor}
\newtheorem{lemma}{Lemma}

\newtheorem{remark}{Remark}
\allowdisplaybreaks

\newcommand\bxi{\boldsymbol{\xi}}
\newcommand\bx{\boldsymbol{x}}

\newcommand{\kn}{\mathrm{Kn}}
\newcommand{\stf}{\mathrm{stf}}
\newcommand{\vb}{\boldsymbol{v}}
\newcommand{\qbarb}{\boldsymbol{\bar{q}}}

\newcommand{\qb}{\boldsymbol{q}}
\newcommand{\sigmabarb}{\boldsymbol{\bar{\sigma}}}
\newcommand{\sigmab}{\boldsymbol{\sigma}}
\newcommand{\sigmabar}{\bar{\sigma}}
\newcommand\Rb{\boldsymbol{R}}
\newcommand{\vertiii}[1]{{\left\vert\kern-0.25ex\left\vert\kern-0.25ex\left\vert #1 
    \right\vert\kern-0.25ex\right\vert\kern-0.25ex\right\vert}}

\title[R13 equations with Onsager boundary conditions]{Time-dependent Regularized 13-Moment Equations with Onsager Boundary Conditions in the Linear Regime}

\author{Bo Lin}
\address[Bo Lin]{Department of Mathematics, National University of Singapore, Singapore 119076}
\email{linbo@u.nus.edu}

\author{Haoxuan Wang}
\address[Haoxuan Wang]{Department of Mathematics, National University of Singapore, Singapore 119076}
\email{e0321206@u.nus.edu}

\author{Siyao Yang}
\address[Siyao Yang]{Committee on Computational and Applied Mathematics, Department of Statistics, University of Chicago, Chicago, IL 60637 USA}
\email{siyaoyang@uchicago.edu}

\author{Zhenning Cai}
\address[Zhenning Cai]{Department of Mathematics, National University of Singapore, Singapore 119076}
\email{matcz@nus.edu.sg}


\begin{document}

\begin{abstract}
We develop the time-dependent regularized 13-moment equations for general elastic collision models under the linear regime. Detailed derivation shows the proposed equations have super-Burnett order for small Knudsen numbers, and the moment equations enjoy a symmetric structure. A new modification of Onsager boundary conditions is proposed to ensure stability as well as the removal of undesired boundary layers. Numerical examples of one-dimensional channel flows is conducted to verified our model.
\end{abstract}

\maketitle

\section{Introduction}
\label{sec:introduction}
The study of rarefied gas dynamics has evolved over more than a century, where kinetic theory is crucial for accurately describing its fluid mechanics via a microscopic explanation for macroscopic behavior. The fundamental equation of the kinetic theory, the Boltzmann equation, provides a detailed statistical description of the gas dynamics. The development of Boltzmann solvers has progressed 
significantly in \cite{BGK1954,BirdDSMC,Mieussens2000jcp,Dimarco_Pareschi_2014_acta,Gamba2017fastspectral,Dimarco2018jcp,Alekseenko2022jcp_data,lin2024sinum},
nevertheless, accurately solving the Boltzmann equation for practical problems remains a complex and resource-intensive task mainly due to its high dimensionality, nonlinear collisions and complex boundary conditions.

Instead of directly solving the Boltzmann equation, many researchers focus on the macroscopic moments of the distribution function and modify classical fluid equations, such as the Navier-Stokes equations, to describe the mechanics of moderately rarefied gases. Two {classical} 
techniques used to derive higher-order (in terms of Knudsen number) macroscopic fluid equations from the Boltzmann equation are the Grad's moment method in \cite{Grad13_1949} and the Chapman-Enskog expansion in \cite{chapman1970mathematical}. Despite their successes in capturing near-equilibrium effects, the Grad's 13-moment equations may experience a loss of hyperbolicity and unphysical subshocks. Moreover, the second-order Burnett and third-order super-Burnett equations, which are derived from the Chapman-Enskog expansion, suffer from instability.
Recent researches address these limitations by employing various regularization techniques (\cite{jin2001regularization,muller2003extended,StruchtrupR13_2003,bobylev2006instabilities,Ottinger2010thermo,Han2019learningclosure}) to enhance the accuracy and stability of the methods across a broader range of rarefied gas dynamics. {Besides the regularization by the above two classical techniques, researchers have also explored alternative ways to address their limitations and derive higher-order equations such as the Onsager-consistent approaches (\cite{Onsager_2019book,jadhav2023oburnett}).}

In this work, we will investigate the regularized version of Grad's 13-moment equations proposed in \cite{StruchtrupR13_2003}, which can be interpreted as a Chapman–Enskog like procedure applied to the Grad's 13 moments. This regularization, abbreviated as R13, has been developed for both Maxwell (\cite{Struchtrup2005macroscopic}) and non-Maxwell molecules (\cite{Struchtrup2013pof, Cai2020regularized}) and validated across diverse applications (\cite{Taheri2012pre_r13_app,rana2015numerical,r13_compare_2017}). The R13 equations are attractive since the equations include only second-order derivatives while attaining the super-Burnett order. Here the super-Burnett order means that the super-Burnett equations, which are the result of the Chapman-Enskog expansion up to the third order, can be derived from the R13 equations. Note that the original super-Burnett equations contain fourth-order derivatives, causing significant difficulties in the numerical discretization and formulation of boundary conditions. {Due to the inclusion of corrections from higher-order moments through terms derived from the Chapman-Enskog expansion}, the R13 equations can describe various rarefaction effects such as Knudsen boundary layers and heat fluxes from cold areas to hot areas (\cite{Rana2013}). Also, the R13 equations are equipped with reliable boundary conditions (\cite{Torrilhon2008, Struchtrup2017}) and have been shown to be robust in numerical simulations (\cite{Rana2013}). Recently in \cite{Torrilhon2017}, Onsager boundary conditions are proposed to ensure the stability of boundary value problems of linear R13 equations for Maxwell molecules. However, it is not straightforward to generalize these Onsager boundary conditions to non-Maxwell molecules, even in the linear case. The main reason is the loss of a symmetric structure required by the Onsager boundary conditions (see \cite{bunger2023KRM}). To remedy this, in \cite{yang2024siap}, the linear R13 equations for arbitrary elastic collision models were re-derived using a different method. The new derivation relies on the simpler form of the Chapman-Enskog expansion for steady-state equations compared with the time-dependent case, so that it cannot be applied directly to dynamic problems.

The purpose of this work is to seek possibilities to generalize the Onsager boundary conditions to time-dependent R13 equations in the linear regime. This also requires new derivations of R13 equations based on the linearized Boltzmann equation for the distribution function $f$:
\begin{equation}
\label{eq:Boltzmann}
\frac{\partial f}{\partial t} + \xi_k \frac{\partial f}{\partial x_k} = \frac{1}{\kn} \mathcal{L}f,
\end{equation}
where $\mathcal{L}$ is a self-adjoint negative semidefinite linear operator; $x_k$ and $\xi_k$ are position and velocity variables, respectively; $\kn$ is the Knudsen number defined as the ratio of the mean free path of the molecules to macroscopic length scale. Note that here we have used the summation convention: when the same index appears twice, it is assumed that the sum of applied for this index, and the range is from 1 to 3. Due to the linearization, the model is mainly applicable for fluids with low speed such as microflows, and numerical experiments show good agreement with results obtained from the DSMC method (see \cite{Wu2014}). In order that the Onsager boundary conditions can be applied, the linear moment equations should take the form
\begin{equation}
\label{eq:moment}
\mathbf{A}_0 \frac{\partial \boldsymbol{u}}{\partial t} + \mathbf{A}_k \frac{\partial \boldsymbol{u}}{\partial x_k} = \frac{1}{\kn} \mathbf{L} \boldsymbol{u},
\end{equation}
where $\mathbf{A}_0$ is the symmetric mass matrix, $\mathbf{A}_k$ is the symmetric discretization of operator $\xi_k$, and $\mathbf{L}$ is the discretization of operator $\mathcal{L}$ which is symmetric negative semidefinite. The matrix $\mathbf{A}_0$ is allowed to have zero eigenvalues to cover the parabolic scenarios such as the linearized R13 equations, which has been used to study phenomena including evaporation \cite{Claydon2017} and rarefied effects in flows around a sphere \cite{Beckmann2018}. It should be noted that the symmetric structure in the above matrices plays a key role in obtaining the second law of thermodynamics for our moment equations. The similar idea of maintaining symmetry to comply with thermodynamics has also been applied to the derivation of other moment equations (\cite{Agrawal2016pre, Agrawal2023aip}).

When considering thermodynamics of \eqref{eq:moment} within a bounded domain $\Omega$, the entropy production from its boundary should be taken into account. This can be seen by multiplying $\boldsymbol{u}^{\intercal}$ on \eqref{eq:moment} and integrating over $\Omega$, yielding
\begin{equation} \label{eq:L2stability}
\frac{\mathrm{d}}{\mathrm{d}t} \int_{\Omega} \frac{1}{2} \boldsymbol{u}^{\intercal} \mathbf{A}_0 \boldsymbol{u} \mathrm{d} \boldsymbol{x} + \int_{\partial \Omega} \frac{1}{2} \boldsymbol{u}^{\intercal} \mathbf{A}_k n_k \boldsymbol{u} \mathrm{d} \boldsymbol{s} = \frac{1}{\kn} \int_{\Omega} \boldsymbol{u}^{\intercal} \mathbf{L} \boldsymbol{u} \mathrm{d} \boldsymbol{x} \leqslant 0,
\end{equation}
where $\boldsymbol{n} = (n_1, n_2, n_3)^{\intercal}$ denotes the outward normal vector at boundary $\partial \Omega$. Hence, the
second law of thermodynamics can be achieved if the flux across the boundary is bounded, i.e., 
\begin{equation}\label{eq:stablebc}
    - \boldsymbol{u}^{\intercal} n_k \mathbf{A}_k \boldsymbol{u} \leqslant \boldsymbol{g}_{\mathrm{ext}}^{\intercal} \mathbf{M} \boldsymbol{g}_{\mathrm{ext}},
\end{equation}
for a given matrix $\mathbf{M}$ and an external source vector $\boldsymbol{g}_{\mathrm{ext}}=\boldsymbol{g}_{\mathrm{ext}}(t,\boldsymbol{x})$ from boundary conditions. With the above entropic stability, the uniqueness of the solution of \eqref{eq:moment} can be achieved.

The stable boundary conditions are studied in \cite{sarna2018stable} and the authors therein propose a specific form of boundary conditions that fulfills \eqref{eq:stablebc}, which is called Onsager boundary conditions. To specify these boundary conditions, we can introduce an orthogonal matrix $\mathbf{P}$ that decomposes moments $\boldsymbol{u}$ into two subsets $\boldsymbol{u}_{\text{odd}}$ and $\boldsymbol{u}_{\text{even}}$ as
\begin{equation}\label{eq:uoddueven}
    \mathbf{P} \boldsymbol{u} =
    \begin{pmatrix}
        \boldsymbol{u}_{\text{odd}} \\
        \boldsymbol{u}_{\text{even}}
    \end{pmatrix}, \qquad
    \mathbf{P} n_k \mathbf{A}_k \mathbf{P}^{\intercal} =
    \begin{pmatrix}
        0 & \mathbf{A}_{\text{oe}} \\
        \mathbf{A}_{\text{eo}} & 0
    \end{pmatrix}.
\end{equation}
The odd moments $\boldsymbol{u}_{\text{odd}}$ consist of quantities that change sign when the normal vector $\boldsymbol{n}$ is flipped, while the even moments $\boldsymbol{u}_{\text{even}}$ remain unchanged by this flipping. For example, the normal velocity, shear stress and normal heat flux belong to $\boldsymbol{u}_{\text{odd}}$, whereas the density, temperature and tangential velocity belong to $\boldsymbol{u}_{\text{even}}$. The anti-diagonal block structure of the matrix in \eqref{eq:uoddueven} is due to the fact that the normal flux of an odd moment corresponds to an even moment, and conversely, the normal flux of an even moment corresponds to an odd moment. The symmetricity of $\mathbf{A}_k$ implies the matrix in \eqref{eq:uoddueven} is symmetric and therefore $\mathbf{A}_{\text{eo}} = \mathbf{A}_{\text{oe}}^{\intercal}$. The Onsager boundary conditions in \cite{sarna2018stable} suggests that for a full row rank $\mathbf{A}_{\text{oe}}$, the following form of boundary conditions
\begin{equation}\label{eq:r13bc}
    \boldsymbol{u}_{\text{odd}} = \mathbf{Q} \mathbf{A}_{\text{oe}} (\boldsymbol{g}_{\mathrm{ext}} - \boldsymbol{u}_{\text{even}}),
\end{equation}
with matrix $\mathbf{Q}$ being negative semidefinite yield \eqref{eq:stablebc}, and therefore $L^2$ stability is obtained. However, as shown in \cite{yang2024siap}, for general collision models, the matrix $\mathbf{A}_{\mathrm{oe}}$ can be rank-deficient, which requires an extra step that reduces the number of boundary conditions to the rank of $\mathbf{A}_{\mathrm{oe}}$. Such a technique is also required in this work.

Two major difficulties were encountered when we tried to derive time-dependent linear R13 equations for general collision models. The first is to obtain the form \eqref{eq:moment} while maintaining the super-Burnett order. To achieve this, instead of the classical approach that uses the stress tensor and the heat flux in the equations, we change them to new variables that better match the order of magnitude method. Such an approach can lead to both the desired order of accuracy and the desired structure of equations, so that the technique introduced in \cite{bunger2023KRM} and tweaked in \cite{yang2024siap} can be employed to derive the Onsager boundary conditions. However, the solution exhibits unphysical boundary layers, giving us a second major difficulty in deriving the new model. To address this problem, we revise the original boundary conditions by enforcing equalities setting the coefficients of these boundary layers to zero. Both steady-state and unsteady examples are presented to show the solutions of these new equations.


The rest of this paper is organized as follows. {Our linearization and nondimensionalization of the Boltzmann equation and moment equations are introduced in Section \ref{sec:lin}.} Section \ref{sec:basic_functions} introduces the variables and then shows our main results including the explicit expressions of linear R13 equations with boundary conditions for general elastic collision models. The derivation of R13 equations is given in section \ref{sec:abstract_derivation_of_R13}, where we also demonstrate the super-Burnett order of the model. In section \ref{sec:boundary}, we formulate the boundary conditions, analyze the boundary layers, and remove the unwanted layers by fixing the boundary conditions. One-dimensional examples are shown in section \ref{sec:1Dexample} to validate our models. We conclude our paper with a brief summary in section \ref{sec:conclusion}.

{
\section{Linearization and nondimensionalization} \label{sec:lin}
All the derivations in this work will be established on dimensionless and linearized equations. In this section, we will state our approaches to both linearization and nondimensionalization. Note that each dimensionless variable in this section will be labeled by adding a hat ``$\,\hat{\cdot}\,$'' on top of the symbol denoting the variable, while such accents will be removed for convenience in other sections.

\subsection{Linearization and nondimensionalization of the Boltzmann equation}
In the original Boltzmann equation, the distribution function $f(\bx, \bxi, t)$ satisfies the following integro-differential equation:
\begin{equation} \label{eq:nonlinear_Boltzmann}
\frac{\partial f}{\partial t} + \xi_k \frac{\partial f}{\partial x_k} = Q[f,f] := \int_{\mathbb{R}^3} \int_{\mathbb{S}^2} B(\bxi - \bxi_1, \boldsymbol{\omega}) [f(\bxi_1') f(\bxi') - f(\bxi_1) f(\bxi)] \,\mathrm{d}\boldsymbol{\omega} \,\mathrm{d}\bxi_1,
\end{equation}
where $B(\cdot, \cdot)$ stands for the collision kernel, and 
\begin{displaymath}
\bxi' = \frac{\bxi + \bxi_1}{2} + \frac{|\bxi - \bxi_1|}{2} \boldsymbol{\omega}, \qquad
\bxi_1' = \frac{\bxi + \bxi_1}{2} - \frac{|\bxi - \bxi_1|}{2} \boldsymbol{\omega}.
\end{displaymath}
In this work, we are interested in the linear regime where the distribution function is close to a global Maxwellian with density $\rho_0$, velocity $\vb_0$, temperature $\theta_0$ and mass of a single molecule $m$:
\begin{displaymath}
\mathcal{M}_0(\bxi) = \frac{\rho_0}{m(2\pi \theta_0)^{3/2}} \exp \left( -\frac{|\bxi - \vb_0|^2}{2\theta_0} \right).
\end{displaymath}
We can then introduce a small quantity $\epsilon$ and write the distribution function $f$ as
\begin{equation} \label{eq:perturbation}
f(\bx,\bxi,t) = \mathcal{M}_0(\bxi) [1 + \epsilon \hat{f}(\hat{\bx},\hat{\bxi},\hat{t})],
\end{equation}
where $\hat{\bx}$, $\hat{\bxi}$ and $\hat{t}$ are dimensionless variables defined by
\begin{equation}\label{dimless_x_t}
\hat{\bx} = (\bx - \vb_0 t) / L, \qquad \hat{\bxi} = (\bxi - \vb_0) / \sqrt{\theta_0}, \qquad \hat{t} = t / (L / \sqrt{\theta_0}).
\end{equation}
Plugging \eqref{eq:perturbation} into the Boltzmann equation yields
\begin{equation} \label{eq:hat_f}
\frac{\partial \hat{f}}{\partial \hat{t}} + \hat{\xi}_k \frac{\partial \hat{f}}{\partial \hat{x}_k} = \frac{\rho_0 L}{m \sqrt{\theta_0}} \frac{\mathcal{Q}[\hat{\mathcal{M}}_0,\hat{f} \hat{\mathcal{M}}_0] + \mathcal{Q}[\hat{f}\hat{\mathcal{M}}_0, \hat{\mathcal{M}}_0] + \epsilon \mathcal{Q}[\hat{f}\hat{\mathcal{M}}_0, \hat{f} \hat{\mathcal{M}}_0]}{\hat{\mathcal{M}}_0},
\end{equation}
where $\hat{\mathcal{M}}_0$ is the dimensionless Maxwellian
\begin{displaymath}
\hat{\mathcal{M}}_0(\hat{\bxi}) = \frac{1}{(2\pi)^{3/2}} \exp \left( -\frac{|\hat{\bxi}|^2}{2} \right).
\end{displaymath}
The linearized equation can then be obtained by discarding the $O(\epsilon)$ term in \eqref{eq:hat_f} and defining
\begin{displaymath}
\hat{\mathcal{L}}[\hat{f}] := \frac{\mathcal{Q}[\hat{\mathcal{M}}_0,\hat{f} \hat{\mathcal{M}}_0] + \mathcal{Q}[\hat{f}\hat{\mathcal{M}}_0, \hat{\mathcal{M}}_0]}{B_0 \hat{\mathcal{M}}_0}, \qquad \kn = \frac{m B_0 \sqrt{\theta_0}}{\rho_0 L}.
\end{displaymath}
Here $B_0$ is a constant with the same dimension as the collision kernel $B(\cdot, \cdot)$ to guarantee that $\hat{\mathcal{L}}$ is dimensionless, and here we choose $B_0$ such that 
\begin{displaymath}
\langle \hat{\xi}_1 \hat{\xi}_2, \hat{\mathcal{L}}(\hat{\xi}_1 \hat{\xi}_2) \rangle = -1,
\end{displaymath}
where the inner product is defined by
\begin{displaymath}
\langle \hat{f}, \hat{g} \rangle = \int_{\mathbb{R}^3} \hat{f}(\hat{\bxi}) \hat{g}(\hat{\bxi}) \hat{\mathcal{M}}_0(\hat{\bxi}) \,\mathrm{d} \hat{\bxi}.
\end{displaymath}

Note that the linearized Boltzmann equation satisfies the linearized H-theorem. In the original Boltzmann equation \eqref{eq:nonlinear_Boltzmann}, the entropy density and the entropy flux are defined by
\begin{displaymath}
\eta = -k_B \int_{\mathbb{R}^3} f \log f \,\mathrm{d}\bxi, \qquad \psi_k = -k_B \int_{\mathbb{R}^3} \xi_k f \log f \,\mathrm{d}\bxi,
\end{displaymath}
which satisfy the inequality
\begin{displaymath}
\frac{\partial \eta}{\partial t} + \frac{\partial \psi_k}{\partial x_k} \geqslant 0,
\end{displaymath}
indicating the second law of thermodynamics. For the linearized Boltzmann equation, the corresponding definitions are
\begin{displaymath}
\hat{\eta} = -\int_{\mathbb{R}^3}  |\hat{f}|^2 \,\mathrm{d}\hat{\bxi}, \qquad \hat{\psi}_k = -\int_{\mathbb{R}^3} \xi_k |\hat{f}|^2 \,\mathrm{d}\hat{\bxi}.
\end{displaymath}
Since the linearized collision operator $\hat{\mathcal{L}}$ satisfies $\langle \hat{f}, \hat{\mathcal{L}}[\hat{f}] \rangle \leqslant 0$ for any $\hat{f}$, the H-theorem turns out to be the $L^2$ stability of the linearized Boltzmann equation:
\begin{displaymath}
\frac{\partial \hat{\eta}}{\partial \hat{t}} + \frac{\partial \hat{\psi}_k}{\partial \hat{x}_k} \geqslant 0,
\end{displaymath}
Later in our derivation of moment equations, this property will be preserved.

\subsection{Linearization and nondimensionalization of the conservation laws}
The laws of mass, momentum and energy conservation can be derived by taking moments of the Boltzmann equation, which yields
\begin{align*}
& \frac{\partial \rho}{\partial t} + \frac{\partial (\rho v_j)}{\partial x_j} = 0, \\
& \frac{\partial (\rho v_i)}{\partial t} + \frac{\partial}{\partial x_j} (\rho v_i v_j + \rho \theta \delta_{ij} + \sigma_{ij}) = 0, \\
& \frac{\partial}{\partial t} \left(\frac{1}{2} \rho v_i v_i + \frac{3}{2} \rho \theta \right) + \frac{\partial}{\partial x_j} \left( \frac{5}{2} \rho \theta  v_j + \frac{1}{2} \rho v_i v_i v_j + \sigma_{ij} v_i + q_j \right) = 0.
\end{align*}
Here $\sigma_{ij}$ denotes the stress tensor and $q_j$ denotes the heat flux. To linearize these equations, we again assume that the fluid state is close to the equilibrium specified by density $\rho_0$, velocity $v_{0,i}$ and temperature $\theta_0$, and introduce dimensionless variables $\hat{\rho}$, $\hat{\theta}$, $\hat{v}_i$, $\hat{\sigma}_{ij}$ and $\hat{q}_{i}$ by  
\begin{displaymath}
    \rho = \rho_0(1+\epsilon \hat{\rho}) , \quad   \theta = \theta_0(1+\epsilon\hat{\theta}), \quad v_i = v_{0,i} + \epsilon\sqrt{\theta_0} \hat{v}_i, \quad \sigma_{ij} = \epsilon\rho_0 \theta_0 \hat{\sigma}_{ij}, \quad q_i = \epsilon\rho_0 \theta_0^{3/2}\hat{q}_i.
\end{displaymath}
By substituting the above expressions together with \eqref{dimless_x_t} into the conservation laws, and then dropping all the higher-order terms, we reach the linearized and dimensionless conservation laws:
\begin{equation} \label{eq:conservation_laws}
\begin{aligned}
     &\frac{\partial \hat{\rho}}{\partial \hat{t}}+  \frac{\partial \hat{v}_j}{\partial \hat{x}_j}  = 0,  \\ 
     & \frac{\partial \hat{v}_i}{\partial \hat{t}}+ \frac{\partial \hat{\rho}}{\partial \hat{x}_i} +\frac{\partial \hat{\theta}}{\partial \hat{x}_i}+ \frac{\partial \hat{\sigma}_{ij}}{\partial \hat{x}_j}  = 0, \\ 
  & \frac{\partial \hat{\theta}}{\partial \hat{t}}+ \frac{2}{3} \frac{\partial \hat{v}_j}{\partial \hat{x}_j}  +\frac{2}{3}\frac{\partial \hat{q}_j}{\partial \hat{x}_j} = 0.
\end{aligned}
\end{equation}

Another approach to derive the same equations is to directly take moments of the linearized Boltzmann equation. The relationship between the dimensionless moments ($\hat{\rho}$, $\hat{\theta}$, etc.) and the dimensionless distribution function $\hat{f}$ is
\begin{gather*}
\hat{\rho} = \langle 1, \hat{f} \rangle, \quad \hat{v}_i = \langle \hat{\xi}_i, \hat{f} \rangle, \quad \hat{\theta} = \left\langle \frac{1}{3} |\hat{\bxi}|^2 - 1, \hat{f} \right\rangle, \\
\hat{\sigma}_{ij} = \left\langle \hat{\xi}_i \hat{\xi}_j - \frac{1}{3} |\hat{\bxi}|^2 \delta_{ij}, \hat{f} \right\rangle, \qquad \hat{q}_i = \left\langle \frac{1}{3} |\hat{\bxi}|^2 \hat{\xi}_i, \hat{f} \right\rangle.
\end{gather*}
Thus, using the fact that $\langle 1, \hat{\mathcal{L}} [\hat{f}] \rangle = \langle \hat{\xi}_i, \hat{\mathcal{L}} [\hat{f}] \rangle = \langle |\hat{\bxi}|^2, \hat{\mathcal{L}} [\hat{f}] \rangle = 0$, one can take the inner product of the linearized equation \eqref{eq:Boltzmann} and the polynomials $1$, $\hat{\xi}_i$ and $|\hat{\bxi}|^2$ to obtain the moment equations \eqref{eq:conservation_laws}.

The two methods to formulate the conservation laws are essentially equivalent, since taking moments and linearizing equations are two commutable operations. In our derivation of moment equations to be presented below, we will follow the second approach which takes moments of the linearized equations \eqref{eq:Boltzmann}. Due to the presence of $\hat{\sigma}_{ij}$ and $\hat{q}_i$ in \eqref{eq:conservation_laws}, our purpose is to find a suitable closure that provides decent approximation to the linearized Boltzmann equation when the Knudsen number $\kn$ is small. Hereafter, the accent ``\,$\hat{\cdot}$\,'' will no longer be added to dimensionless variables.
}

\section{Linear R13 equations}
\label{sec:basic_functions}
Here we will present the final form of the R13 equations with wall boundary conditions, which is the main conclusion of this paper. Since our selection of the 13 moments is slightly different from classical literature, we will introduce our choice of variables before list out the equations.

\subsection{Choice of variables}
The 13-moment equations were first proposed by \cite{Grad13_1949}, where the distribution function was expanded into an infinite series using Hermite polynomials. The coefficients of the polynomials were regarded as moments, and Grad chose to include the first 13 moments, namely the density ($\rho$), momentum ($\rho v_i$), pressure ($\rho \theta$), stress tensor ($\sigma_{ij}$) and heat flux ($q_i$), into his equations. In later developments of 13-moment systems, such a choice of the 13 moments becomes a standard and is followed by most works \cite{StruchtrupR13_2003,Myong1999nccr,yong2015cdf}. In particular, for the regularized 13-moment equations for Maxwell molecules, such a choice of moments matches exactly with the analysis using the order of magnitude method \cite[Chapter 8]{Struchtrup2005macroscopic}. It is seen that for linear equations, these moments, denoted by $\langle \boldsymbol{\phi}_{13}, f\rangle$ with $\boldsymbol{\phi}_{13}$ being a 13-dimensional vector of polynomials of $\boldsymbol{\xi}$, fully describes the distribution function $f$ up to order $O(\kn)$, which means for any polynomial $r(\boldsymbol{\xi})$ satisfying $\langle \boldsymbol{\phi}_{13}, r \rangle = 0$, the moment $\langle r, f\rangle$ has at least order $O(\kn^2)$ when performing asymptotic analysis. This property leads to the super-Burnett order of R13 equations for Maxwell molecules, meaning that the super-Burnett equations can be derived from the R13 equations by Chapman-Enskog expansion. However, it is also known that this no longer holds for other types of molecules \cite{Struchtrup20052ndorder}. Inspired by this fact, in our study of 13-moment equations for general gas molecules, we will adopt a different set of variables $\langle \tilde{\boldsymbol{\phi}}_{13}, f\rangle$, such that we can again obtain $\langle r, f\rangle \sim o(\kn)$ once $\langle \tilde{\boldsymbol{\phi}}_{13}, r\rangle = 0$. This will help achieve the super-Burnett order for the moment equations.

The details on the construction of $\tilde{\boldsymbol{\phi}}_{13}$ will be discussed in Section \ref{sec:abstract_derivation_of_R13}. Here we only provide some brief information necessary for the formulation of the moment equations. Our 13 moments also include conservative quantities such as the density ($\rho$), velocity ($v_i$) and temperature ($\theta$). Note that the velocity and the temperature are conservative only under linear settings. For the remaining 8 moments, they also include a trace-free 2-tensor and a vector like in the classical case. For consistency, we will denote the 2-tensor by $\bar{\sigma}_{ij}$ and the vector by $\bar{q}_i$. The relationship between these variables and the stress tensor ($\sigma_{ij}$) and the heat flux ($q_i$) will be given later in this section. The conservative variables $\rho$, $v_i$ and $\theta$ are $O(\kn^0)$ moments, while the higher-order moments $\bar{\sigma}_{ij}$ and $\bar{q}_i$ are $O(\kn)$ moments. Note that the conservative moments are always defined using the distribution function $f$ by
\begin{displaymath}
  (\rho, v_i, \theta) = \left\langle \left(1, \xi_i, \frac{1}{3} \xi_k \xi_k \right), f \right\rangle,
\end{displaymath}
where as the definitions of $\bar{\sigma}_{ij}$ and $\bar{q}_i$ depend on the collision model. When Maxwell molecules are considered, the definitions of $\bar{\sigma}_{ij}$ and $\bar{q}_i$ reduce to those of $\sigma_{ij}$ and $q_i$.

\subsection{Moment equations and boundary conditions}
In this section, we will list the equations and boundary conditions derived for the aforementioned variables. Like the generalized 13-moment equations derived in \cite{Struchtrup20052ndorder}, our equations will contain coefficients that depend on the type of gas molecules. In the equations below, these constants will be denoted by $k_i$ and $l_i$. The equations are
\begin{align}
    \label{newr13:eqa}
      &\frac{\partial \rho}{\partial t}+  \nabla \cdot \vb  = 0,  \\ 
    \label{newr13:eqb}
     &\frac{\partial \theta}{\partial t}+ \frac{2}{3} \nabla \cdot \vb  +\frac{2}{3}k_0\nabla \cdot \qbarb-k_1 \kn \Delta  \theta+k_2\kn \nabla \cdot (\nabla \cdot  \sigmabarb ) = 0 ,\\ 
    \label{newr13:eqc}
     &\frac{\partial \vb}{\partial t}+ \nabla \rho +\nabla \theta -k_3 \kn \nabla \cdot (\nabla \vb)_\stf  -k_4 \kn \nabla \cdot (\nabla \qbarb)_\stf+ k_5 \nabla\cdot\sigmabarb  = 0, \\ 
    &\frac{\partial \qbarb}{\partial t}+ \frac{5}{2} k_0 \nabla \theta - \frac{5}{2}k_4\kn\nabla \cdot (\nabla \vb)_\stf  -2 k_6 \kn \nabla(\nabla \cdot \qbarb) \notag \\ 
    & \quad {} -\frac{12}{5}k_7\kn\nabla \cdot (\nabla \qbarb)_\stf+k_8 \nabla \cdot \sigmabarb  = - \frac{1}{\kn} \frac{2}{3} l_1 \qbarb, \label{newr13:eqd1} \\
    &\frac{\partial\sigmabarb}{\partial t} + 3k_2\kn \left(\nabla^2 \theta \right)_\stf+ 2k_5 (\nabla \vb)_\stf +\frac{4}{5}k_8 (\nabla \qbarb)_\stf \notag \\
    & \quad {}- 2k_9\kn \nabla\cdot (\nabla\sigmabarb)_\stf -k_{10}\kn \left(\nabla(\nabla\cdot \sigmabarb) \right)_\stf= - \frac{1}{\kn} l_2 \sigmabarb. \label{newr13:eqe1}
\end{align}
Here $(\cdot)_{\stf}$ refers to the symmetric and trace-free part of a given tensor, which is defined entrywisely as $(\cdot)_{ij} \mapsto ((\cdot)_\stf)_{ij} = (\cdot)_{\langle ij\rangle}$ for a matrix and $(\cdot)_{ijk} \mapsto ((\cdot)_\stf)_{ijk} = (\cdot)_{\langle ijk\rangle}$ for a 3-tensor. The detailed definition can be found in \cite{Struchtrup2005macroscopic}. The coefficients $k_i$ and $l_i$ have been computed for gas molecules whose intermolecular potential satisfies inverse power laws. Values for these coefficients for some power indices are given in the Appendix \ref{app:coefficient}. Although these coefficients look arbitrary, our derivation in section \ref{sec:abstract_derivation_of_R13} will reveal that they are fully determined by the linearized and nondimensionalized Boltzmann collision operator, which depends only on the potential energy between two gas molecules. In inverse-power-law models, only two parameters exist: the intensity coefficient of the potential and the power $\eta$. After nondimensionalization, the intensity coefficient is integrated into the Knudsen number, and thus $\eta$ is the only parameter that dictates all the coefficients.

The first three equations \eqref{newr13:eqa}\eqref{newr13:eqb}\eqref{newr13:eqc} are conservation laws of mass, energy and momentum. By comparing \eqref{newr13:eqb} and \eqref{newr13:eqc} with the conservation laws expressed by the stress tensor and the heat flux {(see \eqref{eq:conservation_laws})}:
\begin{equation*}
     \frac{\partial \theta}{\partial t}+ \frac{2}{3} \nabla\cdot \vb  +\frac{2}{3}\nabla\cdot \qb= 0, \qquad
     \frac{\partial \vb}{\partial t}+\nabla \rho +\nabla \theta+ \nabla \cdot \sigmab  = 0,
\end{equation*}
one can observe the relationship between $\sigmabarb$, $\qbarb$ and $\sigmab$, $\qb$:
\begin{equation}\label{eq:qbartoq}
\sigmab = k_5\sigmabarb-k_4 \kn (\nabla \qbarb)_{\stf} -k_3 \kn(\nabla \vb)_{\stf}, \quad \qb = k_0\qbarb-\frac{3}{2}k_1\kn \nabla\theta+\frac{3}{2}k_2\kn\nabla\cdot \sigmabarb.
\end{equation}
Such a relationship allows us to calculate the stress tensor and heat flux once the equations \eqref{newr13:eqa}--\eqref{newr13:eqe1} are solved. We remark that the moment equations \eqref{newr13:eqa}--\eqref{newr13:eqe1} we derived is Galilean invariant. The detailed proof can be found in Appendix \ref{app:galilean}.

Another important ingredient of the moment equations is the wall boundary conditions, since the linear moment equations are often applied in channel flows, where the boundaries play crucial roles in rarefaction effects. The boundary conditions we present here are derived from the Maxwell boundary conditions for the Boltzmann equation, where the reflected gas flow is a linear combination of specular and diffusive reflections. The accommodation coefficient, denoted by $\chi$, is the parameter describing the proportion of diffusive reflection. For simplicity, we are going to use $\tilde{\chi} = 2\chi / (2-\chi)$ in our formulation of boundary conditions. 

We now focus on a specific point on the wall and let $\boldsymbol{n}$ be the outer unit normal vector at this point. Meanwhile, we let $\boldsymbol{\tau}_1$ and $\boldsymbol{\tau}_2$ be the two tangential unit vectors that are perpendicular to each other. When describing boundary conditions, we temporarily rotate our coordinate system such that the three axes become $\boldsymbol{n}$, $\boldsymbol{\tau}_1$ and $\boldsymbol{\tau}_2$. Correspondingly, the three components of the velocity will be denoted by $v_n$, $v_{\tau_1}$ and $v_{\tau_2}$ (similar for $\bar{q}$), and the components of $\bar{\sigma}$ are written by $\bar{\sigma}_{nn}, \bar{\sigma}_{n \tau_1}, \bar{\sigma}_{n \tau_2}, \bar{\sigma}_{\tau_1 \tau_1}, \bar{\sigma}_{\tau_1 \tau_2}, \bar{\sigma}_{\tau_2 \tau_2}$. Moreover, we assume that the temperature and the velocity of the wall are $\theta^W$ and $v_i^W$ at this point. Under these assumptions, the wall boundary conditions have the following form:
\begin{align}
    \label{eq:bcv} & v_n = 0, \\
    & \bar{q}_{n}  = \tilde{\chi} \left[ m_{11} (\theta -\theta^{W}) +m_{12} \bar{\sigma}_{nn} -  m_{13} \kn \frac{\partial \bar{q}_j}{\partial x_j}  -  m_{14} \kn \frac{\partial \bar{q}_{\langle n}}{\partial x_{n \rangle}}  - m_{15} \kn \frac{\partial v_{\langle n}}{\partial x_{n \rangle}} \right], \\ 
     & m_{26}\bar{q}_{n} + m_{27} \kn \frac{\partial \theta}{\partial x_n} - m_{28} \kn \frac{\partial \bar{\sigma}_{nj}}{\partial x_j}   \notag\\
     & \quad =\tilde{\chi} \left[-m_{21} (\theta -\theta^{W}) + m_{22} \bar{\sigma}_{nn} +  m_{23} \kn \frac{\partial \bar{q}_j}{\partial x_j} +  m_{24} \kn \frac{\partial \bar{q}_{\langle n}}{\partial x_{n \rangle}}  + m_{25} \kn \frac{\partial v_{\langle n}}{\partial x_{n \rangle}} \right],  \\ 
      & \bar{\sigma}_{\tau_i n} = \tilde{\chi} \left[m_{31} (v_{\tau_i} -v^{W}_{\tau_i}) +m_{32}\bar{q}_{\tau_i} - m_{33} \kn \frac{\partial \bar{\sigma}_{\tau_i j}}{\partial x_j}  - m_{34} \kn  \frac{\partial \bar{\sigma}_{\langle \tau_i n}}{\partial x_{n \rangle}}  + m_{35} \kn \frac{\partial \theta}{\partial x_{\tau_i}} \right], i=1,2,\\ 
        &  m_{46}\bar{\sigma}_{\tau_in} + m_{47}\kn\frac{\partial v_{\langle \tau_i}}{\partial x_{n\rangle}} + m_{48}\kn\frac{\partial \bar{q}_{\langle \tau_i}}{\partial x_{n\rangle}}  \notag\\
        & \quad = - \tilde{\chi} \left[-m_{41} (v_{\tau_i} -v^{W}_{\tau_i}) +m_{42}\bar{q}_{\tau_i} + m_{43} \kn \frac{\partial \bar{\sigma}_{\tau_i j}}{\partial x_j}  + m_{44} \kn  \frac{\partial \bar{\sigma}_{\langle \tau_i n}}{\partial x_{n \rangle}}  - m_{45} \kn \frac{\partial \theta}{\partial x_{\tau_i}} \right], i=1,2, \\ 
          & m_{56}\bar{\sigma}_{\tau_in} + m_{57}\kn\frac{\partial v_{\langle \tau_i}}{\partial x_{n\rangle}} + m_{58}\kn\frac{\partial \bar{q}_{\langle \tau_i}}{\partial x_{n\rangle}}    \notag \\
          &\quad = - \tilde{\chi} \left[-m_{51} (v_{\tau_i} -v^{W}_{\tau_i}) +m_{52}\bar{q}_{\tau_i} + m_{53} \kn  \frac{\partial \bar{\sigma}_{\tau_i j}}{\partial x_j}  + m_{54} \kn  \frac{\partial \bar{\sigma}_{\langle \tau_i n}}{\partial x_{n \rangle}}  - m_{55} \kn \frac{\partial \theta}{\partial x_{\tau_i}} \right], i=1,2,\\ 
  & m_{66} \bar{q}_n + \kn\left(m_{67} \frac{\partial \bar{\sigma}_{\langle nn }}{\partial x_{n \rangle}} 
 +m_{68}\frac{\partial \bar{\sigma}_{nj }}{\partial x_{j}} + m_{69}\frac{\partial \theta}{\partial x_n} \right) \notag\\
  & \quad = - \tilde{\chi} \left[-m_{61} (\theta -\theta^{W}) +m_{62} \bar{\sigma}_{nn} + m_{63} \kn \frac{\partial \bar{q}_j}{\partial x_j} +  m_{64} \kn \frac{\partial \bar{q}_{\langle n}}{\partial x_{n \rangle}} + m_{65} \kn \frac{\partial v_{\langle n}}{\partial x_{n \rangle}}\right],\\ 
 & \kn\left(\frac{\partial \bar{\sigma}_{\langle \tau_i \tau_i}}{\partial x_{n \rangle}} + \frac{1}{2} \frac{\partial \bar{\sigma}_{\langle n n}}{\partial x_{n \rangle}} \right) =  -\tilde{\chi} m_{71} \left( \bar{\sigma}_{\tau_i \tau_i}  + \frac{1}{2} \bar{\sigma}_{nn}  \right), i=1,2, \\
\label{eq:bcm}  & \kn \frac{\partial \bar{\sigma}_{\langle \tau_1 \tau_2}}{\partial x_{n\rangle}}  = -\tilde{\chi} m_{81} \bar{\sigma}_{\tau_1 \tau_2}.
\end{align}
In these equations, $m_{jk}$ are constants depending on the molecular interactions. As examples, we will consider inverse-power-law models in which the repulsive force between two gas molecules is proportional to the $\eta$th power of the distance between them. For such models, the values of the constants $k_i$ and $m_{ij}$ are tabulated in the Appendix \ref{app:coefficient} for some power indices.

It will become clear by its derivation that this set of boundary conditions satisfies the general form of $L^2$-stable boundary conditions \eqref{eq:r13bc}, which implies the second law of thermodynamics. The external contribution $\boldsymbol{g}_{\mathrm{ext}}$ is given by the terms with wall velocity $v_i^W$ and wall temperature $\theta^W$.

\subsection{Second law of thermodynamics}
The derivation of the new R13 equations \eqref{newr13:eqa}---\eqref{newr13:eqe1}, to be detailed in the next section, maintains the structure required by \eqref{eq:moment}. Therefore, according to \eqref{eq:L2stability}. The second law of thermodynamics is automatically satisfied. In fact, the entropy inequality can also be observed from \eqref{newr13:eqa}---\eqref{newr13:eqe1} directly by straightforward calculations. Assume that the spatial domain is $\mathbb{R}^3$ and all quantities decay to zero when $\bx$ tends infinity. Then, by calculating
\begin{displaymath}
\int_{\mathbb{R}^3} \left[\rho \cdot \eqref{newr13:eqa} + \frac{3}{2} \theta \cdot \eqref{newr13:eqb} + \vb \cdot \eqref{newr13:eqc} + \frac{2}{5} \qbarb \cdot \eqref{newr13:eqd1} + \frac{1}{2} \sigmabarb : \eqref{newr13:eqe1} \right] \mathrm{d}\boldsymbol{x},
\end{displaymath}
one obtains
\begin{equation}
\begin{split}
& \frac{\mathrm{d}}{\mathrm{d}t} \int_{\mathbb{R}^3} \frac{1}{2} \left(\rho^2 + \frac{3}{2} \theta^2 + |\vb|^2 + \frac{2}{5} |\qbarb|^2 + \frac{1}{2} \|\sigmabarb\|^2 \right) \mathrm{d}\bx \\
= & -\kn\int_{\mathbb{R}^3} \left( \frac{3}{2} k_1 |\nabla \theta|^2 + k_3 \|(\nabla \vb)_{\stf}\|^2 + \frac{4}{5} k_6 \|\nabla \qbarb\|^2 + \frac{24}{25} k_7 \|(\nabla \qbarb)_{\stf}\|^2 \right) \mathrm{d}\bx \\
& -\kn\int_{\mathbb{R}^3} \left( k_9 \vertiii{(\nabla \sigmabarb)_{\stf}}^2 + \frac{1}{2} k_{10} |\nabla \cdot \sigmabarb|^2\right) \mathrm{d}\bx - \frac{1}{\kn}\int_{\mathbb{R}^3} \left( \frac{4}{15} l_1 |\qbarb|^2 + \frac{1}{2} l_2 \|\sigmabarb\|^2 \right) \mathrm{d}\bx \\
& + 3\kn k_2\int_{\mathbb{R}^3} \nabla \theta \cdot \left( \nabla \cdot \sigmabarb \right) \mathrm{d}\bx - 2 \kn k_4 \int_{\mathbb{R}^3} \left( \nabla \vb \right)_{\stf} : \left( \nabla \qbarb \right)_{\stf} \mathrm{d}\bx,
\end{split}
\end{equation}
where we have used $|\cdot|$, $\|\cdot\|$ and $\vertiii{\cdot}$ to denote the Frobenius norms of vectors, 2-tensors and 3-tensors, respectively.
We will show later (at the end of section \ref{sec:R13_derivation}) that
\begin{equation}
\label{eq:k_ineq}
\frac{(3k_2)^2}{4} \leqslant \frac{3}{2}k_1 \cdot \frac{1}{2} k_{10}, \qquad k_4^2 \leqslant k_3 \cdot \frac{24}{25} k_7,
\end{equation}
which yields
\begin{gather*}
    \left|3k_2\int_{\mathbb{R}^3}\nabla\theta\cdot(\nabla \cdot \sigmabarb) \,\mathrm{d}\bx\right| \leqslant \int_{\mathbb{R}^3} \left( \frac{3}{2} k_1 |\nabla\theta|^2 + \frac{1}{2} k_{10} |\nabla \cdot \sigmabarb|^2\right)\mathrm{d}\bx, \\
    \left|2k_4\int_{\mathbb{R}^3}(\nabla \vb)_{\stf} : (\nabla \qbarb)_{\stf} \,\mathrm{d}\bx\right| \leqslant \int_{\mathbb{R}^3} \left( k_3 \|(\nabla\vb)_{\stf}\|^2 + \frac{24}{25} k_7 \|(\nabla \qbarb)_{\stf} \|^2\right)\mathrm{d}\bx.
\end{gather*}
Thus,
\begin{equation} \label{eq:H-theorem}
\begin{split}
& \frac{\mathrm{d}}{\mathrm{d}t} \int_{\mathbb{R}^3} \frac{1}{2} \left(\rho^2 + \frac{3}{2} \theta^2 + |\vb|^2 + \frac{2}{5} |\qbarb|^2 + \frac{1}{2} \|\sigmabarb\|^2 \right) \mathrm{d}\bx \\
\leqslant {} & {-\kn} \int_{\mathbb{R}^3}  \left( \frac{4}{5} k_6 \|\nabla \qbarb\|^2 +
 k_9 \vertiii{(\nabla \sigmabarb)_{\stf}}^2 \right) \mathrm{d}\bx
- \frac{1}{\kn}\int_{\mathbb{R}^3} \left( \frac{4}{15} l_1 |\qbarb|^2 + \frac{1}{2} l_2 \|\sigmabarb\|^2 \right) \mathrm{d}\bx
,
\end{split}
\end{equation}
Since the coefficients $k_6$, $k_9$, $l_1$ and $l_2$ are all nonnegative, the right-hand side of the equation above is nonpositive. The entropy density can thus be defined by
\begin{displaymath}
H = H_0 - \frac{1}{2} \left( \rho^2 + \frac{3}{2} \theta^2 + |\vb|^2 + \frac{2}{5} |\qbarb|^2 + \frac{1}{2} \|\sigmabarb\|^2 \right),
\end{displaymath}
where $H_0$ is an arbitrary constant. In the case of Maxwell molecules, this results echos the conclusion in \cite{Struchtrup2007}.

For bounded domains, with boundary conditions \eqref{eq:bcv}---\eqref{eq:bcm}, we can also show that the entropy may increase only due to the incoming fluxes from the boundary. The derivation is similar to the unbounded case, while boundary terms appear when performing integration by parts, for which the boundary conditions need to be applied to make further simplifications. For any bounded domain $\Omega$, the conclusion has the following form:
\begin{equation} \label{eq:H-theorem_b}
\begin{split}
-\frac{\mathrm{d}}{\mathrm{d}t} \int_{\Omega} H \,\mathrm{d}\bx \leqslant {} & \int_{\partial \Omega} \theta^W (C_1\theta + C_2\bar{q}_n + C_3\bar{\sigma}_{nn}) \,\mathrm{d}s \\
+ \int_{\partial \Omega} [v_{\tau_1}^W & (C_4 v_{\tau_1} + C_5 \bar{q}_{\tau_1} + C_6 \bar{\sigma}_{n\tau_1}) + v_{\tau_2}^W (C_4 v_{\tau_2} + C_5 \bar{q}_{\tau_2} + C_6 \bar{\sigma}_{n\tau_2})] \,\mathrm{d}s,
\end{split}
\end{equation}
where the constants $C_1$ to $C_6$ depend only on the coefficients in the equations and the boundary conditions.
In \eqref{eq:H-theorem_b}, the right-hand side only contains terms related to $\vb^W$ and $\theta^W$, meaning that the entropy can increase only due to the velocities and temperatures of the walls, which does not violate the second law of thermodynamics.

\section{Derivation of R13 equations}
\label{sec:abstract_derivation_of_R13}
In this section, we will detail our derivation of the time-dependent R13 equations. To begin with, we will first explain how the 13 variables used in our equations are chosen based on the collision model.

\subsection{Choice of variables}
As described in the previous section, our choice of the 13 variables is different from Grad's 13-moment equations. To explain how the new variables are defined, we start from the following definition of trace-free moments, which has been utilized in the derivation of steady-state R13 equations \cite{yang2024siap}:
\begin{equation}
\label{eq:def_w}
w_{i_1 \cdots i_l}^n = \langle \psi_{i_1 \cdots i_l}^n, f \rangle,
\end{equation}
where $\psi_{i_1 \cdots i_l}^n$ is a polynomial defined by
\begin{displaymath}
\psi_{i_1 \cdots i_l}^n(\boldsymbol{\xi}) = \bar{L}_n^{(l+1/2)} \left( \frac{|\boldsymbol{\xi}|^2}{2} \right) \xi_{\langle i_1} \cdots \xi_{i_l \rangle},
\end{displaymath}
and $\bar{L}_n^{(l+1/2)}$ is the normalized Laguerre polynomial:
\begin{displaymath}
\bar{L}_n^{(l+1/2)}(x) = \sqrt{\frac{\sqrt{\pi}}{2^{l+1} n! \Gamma (n+l+3/2)}} x^{-(l+1/2)} \left( \frac{\mathrm{d}}{\mathrm{d}x} - 1 \right)^n x^{n+l+1/2}.
\end{displaymath}
These moments can be viewed as generalizations of Grad's 13 moments, and they are related to Grad's moments by
\begin{displaymath}
\rho = w^0, \quad v_i = \sqrt{3} w_i^0, \quad \theta = -\sqrt{\frac{2}{3}} w^1, \quad \sigma_{ij} = \sqrt{15} w_{ij}^0, \quad q_i = -\sqrt{\frac{15}{2}} w_i^1.
\end{displaymath}
By taking moments of the Boltzmann equation \eqref{eq:Boltzmann}, one can derive evolution equations for the moments $w_{i_1 \cdots i_l}^n$:
\begin{equation} \label{eq:mnt_eq}
\begin{split}
& \frac{\partial w_{i_1\cdots i_l}^n }{\partial t} + \left( \sqrt{2(n+l)+3} \frac{\partial w_{i_1\cdots i_l j}^n}{\partial x_j} - \sqrt{2n} \frac{\partial w_{i_1\cdots i_l j}^{n-1}}{\partial x_j} \right) + \\
& \qquad \frac{l}{2l+1} \left( \sqrt{2(n+l)+1} \frac{\partial w_{\langle i_1\cdots i_{l-1}}^n}{\partial x_{i_l \rangle}} - \sqrt{2(n+1)} \frac{\partial w_{\langle i_1\cdots i_{l-1}}^{n+1}}{\partial x_{i_l\rangle}} \right)=
   \frac{1}{\kn} \sum_{n'=0}^{+\infty} a_{lnn'} w_{i_1\cdots i_l}^{n'}.
\end{split}
\end{equation}
Here the coefficients $a_{lnn'}$ on the right-hand side are given by
\begin{displaymath}
a_{lnn'} = \frac{(2l+1)!!}{l!} \langle \psi_{i_1 \cdots i_l}^n, \mathcal{L} \psi_{i_1 \cdots i_l}^{n'} \rangle,
\end{displaymath}
where we do not take summation over the indices $i_1, \ldots, i_l$, and the choice of these indices does not affect the value of $a_{lnn'}$ due to the rotational invariance of the linearized collision operator. By the self-adjointness of $\mathcal{L}$, one can observe that $a_{lnn'} = a_{ln'n}$. Due to the conservation of mass, momentum and energy, it holds that
\begin{equation} \label{eq:a_zero}
a_{00n} = a_{0n0} = a_{01n} = a_{0n1} = a_{10n} = a_{1n0} = 0.
\end{equation}
For inverse-power-law models, the computation of these coefficients are provided in \cite{Cai2015}.

To select appropriate moments in our equations, we apply the Chapman-Enskog expansion to the moment equations, which requires assuming that $\kn$ is a small parameter and applying the following asymptotic expansion:
\begin{equation} \label{eq:w_asymp}
w_{i_1\cdots i_l}^n =  w_{i_1\cdots i_l}^{n|0} + \kn w_{i_1\cdots i_l}^{n|1} + \kn^2 w_{i_1\cdots i_l}^{n|2} + \kn^3 w_{i_1\cdots i_l}^{n|3} + \cdots.
\end{equation}
The conservative variables $w^0$, $w_i^0$ and $w^1$ are regarded as $O(1)$ variables. Straightforward analysis leads to the following results:
\begin{gather}
\label{eq:order0}
w^{n|0} = w^{n|1} = 0, \quad n \geqslant 2;  \qquad w_i^{n|0} = 0, \quad n \geqslant 1; \\
w_{i_1\cdots i_l}^{n|0} = w_{i_1\cdots i_l}^{n|1} = 0, \quad l = 3; \qquad w_{i_1\cdots i_l}^{n|0} = w_{i_1\cdots i_l}^{n|1} = w_{i_1\cdots i_l}^{n|2} = 0, \quad l \geqslant 4; \\
\label{eq:order1}
w_i^{n|1} = \beta_1^{(1),n} \frac{\partial w^1}{\partial x_i}, \quad n \geqslant 1; \qquad w_{ij}^{n|1} = \beta_2^{(1),n} \frac{\partial w_{\langle i}^0}{\partial x_{j\rangle}}; \\
\label{eq:order2}
w^{n|2} = \gamma_0^{(2),n} \frac{\partial w_j^{1|1}}{\partial x_j}, \quad n \geqslant 2; \qquad
w_i^{n|2} =  \gamma_{t1}^{(1),n} \frac{\partial w^{1|1}_i}{\partial t} + \gamma_{s1}^{(1),n} \frac{\partial w^{0|1}_{ij}}{\partial x_j}, \quad n \geqslant 2; \\
\label{eq:order3}
w_{ij}^{n|2}  =  \gamma_{t2}^{(1),n} \frac{\partial w^{0|1}_{ij}}{\partial t} + \gamma_{s2}^{(1),n} \frac{\partial w^{1|1}_{\langle i}}{\partial x_{j \rangle}}, \quad n \geqslant 1;
\qquad w_{ijk}^{n|2} = \gamma_3^{(2),n} \frac{\partial w_{\langle ij}^{0|1}}{\partial x_{k \rangle}}.
\end{gather}
The coefficients $\beta$ and $\gamma$ are all constants depending on the collision term.
Note that \eqref{eq:order1} implies the Navier-Stokes law and the Fourier law, and these equations lead to the following linear relationship between the moments:
\begin{equation}
\label{eq:1st_order_moments}
w_i^{n|1} = \frac{\beta_1^{(1),n}}{\beta_1^{(1),1}} w_i^{1|1}, \quad n \geqslant 1; \qquad w_{ij}^{n|1} = \frac{\beta_2^{(1),n}}{\beta_2^{(1),0}} w_{ij}^{n|0}.
\end{equation}
By the definition of $w_{i_1\cdots i_l}^n$ (see \eqref{eq:def_w}), the above equations indicate
\begin{equation}
\label{eq:2nd_order1}
\left\langle \psi_i^n - \frac{\beta_1^{(1),n}}{\beta_1^{(1),1}} \psi_i^1, f \right\rangle \sim O(\kn^2), \quad n \geqslant 1; \qquad\left\langle \psi_{ij}^n - \frac{\beta_2^{(1),n}}{\beta_2^{(1),0}} \psi_{ij}^0, f \right\rangle \sim O(\kn^2).
\end{equation}
Meanwhile, \eqref{eq:order0} yields
\begin{equation}
\label{eq:2nd_order2}
\langle \psi^n, f \rangle \sim O(\kn^2), \quad n \geqslant 2; \qquad \langle \psi_{i_1\cdots i_l}^n, f \rangle \sim O(\kn^2), \quad l \geqslant 3.
\end{equation}
Recall that we want to choose the 13 variables of the our equations to be $\langle \tilde{\boldsymbol{\phi}}_{13}, f \rangle$ with $\tilde{\boldsymbol{\phi}}_{13}$ satisfying $\langle r, f\rangle \sim o(\kn)$ whenever $\langle \tilde{\boldsymbol{\phi}}_{13}, r\rangle = 0$. A straightforward approach is to select $\tilde{\boldsymbol{\phi}}_{13}$ to be a basis of $(\mathbb{S}^{(2)})^{\perp}$, with $\mathbb{S}^{(2)}$ being the linear span of all polynomials appearing in \eqref{eq:2nd_order1}\eqref{eq:2nd_order2}:
\begin{displaymath}
\begin{split}
\mathbb{S}^{(2)} =  \operatorname{span} \left\{ \psi_i^n - \frac{\beta_1^{(1),n}}{\beta_1^{(1),1}} \psi_i^1 \,\Bigg\vert\, n \geqslant 1\right\} \oplus\operatorname{span}\left\{ \psi_{ij}^n - \frac{\beta_2^{(1),n}}{\beta_2^{(1),0}} \psi_{ij}^0 \,\Bigg\vert\, n \geqslant 0\right\} \quad \\
\oplus \operatorname{span}\{ \psi^n \mid n \geqslant 2\} \oplus\operatorname{span}\{ \psi_{i_1\cdots i_l}^n \mid l \geqslant 3\}.
\end{split}
\end{displaymath}
This leads to the following choice of $\tilde{\boldsymbol{\phi}}_{13}$:
\begin{displaymath}
\psi^0, \quad \psi^1, \quad \psi_i^0, \quad \phi_i^1, \quad \phi_{ij}^0,
\end{displaymath}
where $\psi^0$, $\psi^1$, and $\psi_i^0$ correspond to conservative moments, and $\phi_i^1$ and $\phi_{ij}^0$, corresponding to non-equilibrium variables, are defined by
\begin{equation}
\label{eq:phi}
\phi_i^{1} = \sum_{n=1}^{+\infty} c_1^{1,n} \psi_i^n, \qquad
\phi_{ij}^{0} = \sum_{n=0}^{+\infty} c_2^{0,n} \psi_{ij}^n,
\end{equation}
with the coefficients $c_1^{1,n}$ and $c_2^{0,n}$ satisfying
\begin{equation} \label{eq:c11n_c200}
c_1^{1,n} = \frac{\beta_1^{(1),n}}{\beta_1^{(1),1}} c_1^{1,1}, \qquad c_2^{0,n} = \frac{\beta_2^{(1),n}}{\beta_2^{(1),0}} c_2^{0,0}.
\end{equation}
Thus, the 13 moments in our equations include:
\begin{displaymath}
\rho, \quad \theta, \quad v_i, \quad \bar{q}_i = \langle \phi_i^1, f \rangle, \quad \bar{\sigma}_{ij} = \langle \phi_{ij}^0, f \rangle.
\end{displaymath}
In \eqref{eq:phi}, the constants $c_1^{1,1}$ and $c_2^{0,0}$ can be chosen as arbitrary nonzero numbers, and here we choose them such that when the collision model reduces Maxwell molecules ($\eta = 5$), $\bar{q}_i$ and $\bar{\sigma}_{ij}$ reduce to $q_i$ and $\sigma_{ij}$. More precisely, we choose negative $c_1^{1,1}$ and positive $c_2^{0,0}$ such that
\begin{equation} \label{eq:c_normalization}
    \sum_{n=1}^{+\infty} \left( c_1^{1,n} \right)^2 = \frac{15}{2}, \qquad \sum_{n=0}^{+\infty} \left( c_2^{0,n} \right)^2 = 15.
\end{equation}
Note that the definitions of $\bar{q}_i$ and $\bar{\sigma}_{ij}$ depend on the collision model since the coefficients $c_1^{1,n}$ and $c_2^{0,n}$ are determined by $\beta_1^{(1),n}$ and $\beta_2^{(1),0}$, and these constants will vary when the collision model changes.

\subsection{Second-order variables}
To derive a model with the super-Burnett order, second-order moments are needed as auxiliary variables. The choice of these variables also follows the requirement of orthogonality, meaning that we want to find a vector function $\tilde{\psi}(\boldsymbol{\xi})$ such that
\begin{enumerate}
\item $\langle \tilde{\phi}_{13}, \tilde{\psi}\rangle$ = 0.
\item For any function $r(\boldsymbol{\xi})$ such that $\langle r,\tilde{\psi}\rangle = 0$, it holds that $\langle r, f \rangle \sim o(\kn^2)$.
\item The number of components of $\tilde{\psi}$ should be as low as possible.
\end{enumerate}
To find $\tilde{\psi}$, we need the following linear relationship between the second-order terms of moments:
\begin{equation} \label{eq:linear_rel}
\begin{gathered}
w^{n|2} = d^n_{02} w^{2|2}, \quad n \geqslant 2; \qquad w^{n|2}_i = d^n_{12} w^{2|2}_i +  d^n_{13} w^{3|2}_i, \quad n \geqslant 2; \\
w^{n|2}_{ij} = d^n_{21} w^{1|2}_{ij} +  d^n_{22} w^{2|2}_{ij}, \quad n \geqslant 1; \qquad w^{n|2}_{ijk} = d^n_{30} w^{0|2}_{ijk}, \quad n \geqslant 0,
\end{gathered}
\end{equation}
where
\begin{gather*}
d^n_{02} = \frac{\gamma_0^{(2),n}}{\gamma_0^{(2),2}}, \qquad d^n_{30} = \frac{\gamma_3^{(2),n}}{\gamma_3^{(2),0}}, \\
d^n_{12} =   \frac{\gamma^{(1),n}_{t1} \gamma^{(1),3}_{s1}-\gamma^{(1),n}_{s1} \gamma^{(1),3}_{t1}}{\gamma^{(1),3}_{s1} \gamma^{(1),2}_{t1} - \gamma^{(1),2}_{s1} \gamma^{(1),3}_{t1}}, \qquad d^n_{13} =   \frac{\gamma^{(1),n}_{s1} \gamma^{(1),2}_{t1}-\gamma^{(1),n}_{t1} \gamma^{(1),2}_{s1}}{\gamma^{(1),3}_{s1} \gamma^{(1),2}_{t1} - \gamma^{(1),2}_{s1} \gamma^{(1),3}_{t1}}, \\
d^n_{21} =   \frac{\gamma^{(1),n}_{t2} \gamma^{(1),2}_{s2}-\gamma^{(1),n}_{s2} \gamma^{(1),2}_{t2}}{\gamma^{(1),2}_{s2} \gamma^{(1),1}_{t2} - \gamma^{(1),1}_{s2} \gamma^{(1),2}_{t2}},\qquad  d^n_{22} =   \frac{\gamma^{(1),n}_{s2} \gamma^{(1),1}_{t2}-\gamma^{(1),n}_{t2} \gamma^{(1),1}_{s2}}{\gamma^{(1),2}_{s2} \gamma^{(1),1}_{t2} - \gamma^{(1),1}_{s2} \gamma^{(1),2}_{t2}}.
\end{gather*}
These equations can be derived from \eqref{eq:order2}\eqref{eq:order3} by cancelling the differential terms. The equations in \eqref{eq:linear_rel} reveal that low-order parts of the moments can be cancelled by linear combinations:
\begin{equation} \label{eq:oKn2}
\begin{gathered}
w^n - d_{02}^n w^2 \sim o(\kn^2), \quad n \geqslant 2; \\
(w_i^n - \kn w_i^{n|1}) - d_{12}^n (w_i^2 - \kn w_i^{2|1}) - d_{13}^n (w_i^3 - \kn w_i^{3|1}) \sim o(\kn^2), \quad n \geqslant 2; \\
(w_{ij}^n - \kn w_{ij}^{n|1}) - d_{21}^n (w_{ij}^1 - \kn w_{ij}^{1|1}) - d_{22}^n (w_{ij}^2 - \kn w_{ij}^{2|1}) \sim o(\kn^2), \quad n \geqslant 1; \\
w_{ijk}^n - d_{30} w_{ijk}^0 \sim o(\kn^2), \quad n \geqslant 0.
\end{gathered}
\end{equation}
Note that the correctness of \eqref{eq:order2} and \eqref{eq:order3} (and thus \eqref{eq:linear_rel} and \eqref{eq:oKn2}) relies only on the fact \eqref{eq:1st_order_moments}, and does not require the specific choice of $w_i^{n|1}$ and $w_{ij}^{n|1}$ as described in \eqref{eq:order1}. Therefore, in \eqref{eq:oKn2}, we can follow \eqref{eq:2nd_order1} to select
\begin{displaymath}
\kn w_i^{n|1} = \left\langle \frac{\beta_1^{(1),n}}{\beta_1^{(1),1}} \psi_i^1, f \right\rangle, \quad n \geqslant 1; \qquad \kn w_{ij}^{n|1} = \left\langle \frac{\beta_2^{(1),n}}{\beta_2^{(1),0}} \psi_{ij}^0, f \right\rangle, \quad n \geqslant 0,
\end{displaymath}
so that we can find a linear space $\mathbb{S}^{(3)}$ such that $\langle r, f \rangle \sim o(\kn^2)$ for all $r \in \mathbb{S}^{(3)}$. The definition of $\mathbb{S}^{(3)}$ is
\begin{displaymath}
\begin{split}
\mathbb{S}^{(3)} &=  \operatorname{span} \left\{ \psi^n - d_{02}^n \psi^2 \mid n \geqslant 3\right\} \\
&\oplus \operatorname{span}\left\{ \left(\psi_i^n - \frac{\beta_1^{(1),n}}{\beta_1^{(1),1}} \psi_i^1 \right) - d_{12}^n\left( \psi_i^2 - \frac{\beta_1^{(1),2}}{\beta_1^{(1),1}} \psi_i^1 \right) - d_{13}^n \left( \psi_i^3 - \frac{\beta_1^{(1),3}}{\beta_1^{(1),1}} \psi_i^1\right) \,\Bigg\vert\, n \geqslant 4\right\} \\ 
&\oplus \operatorname{span}\left\{ \left(\psi_{ij}^n - \frac{\beta_2^{(1),n}}{\beta_2^{(1),0}} \psi_{ij}^0 \right) - d_{21}^n\left( \psi_{ij}^1 - \frac{\beta_2^{(1),1}}{\beta_2^{(1),0}} \psi_{ij}^0 \right) - d_{22}^n \left( \psi_{ij}^2 - \frac{\beta_2^{(1),2}}{\beta_2^{(1),0}} \psi_{ij}^0\right) \,\Bigg\vert\, n \geqslant 3\right\} \\ 
& \oplus\operatorname{span}\left\{ \psi_{ijk}^n - d_{30} \psi_{ijk}^0 \mid n \geqslant 1\right\} \oplus \operatorname{span}\left\{ \psi_{i_1 \cdots i_l}^n \mid l \geqslant 4\right\}.
\end{split}
\end{displaymath}
The second-order variables are chosen to be a basis of $\mathbb{S}^{(2)} \cap (\mathbb{S}^{(3)})^{\perp}$.

The function space $\mathbb{S}^{(2)} \cap (\mathbb{S}^{(3)})^{\perp}$ has the form
\begin{equation} \label{eq:V2}
\mathbb{S}^{(2)} \cap (\mathbb{S}^{(3)})^{\perp} = \operatorname{span} \{ \phi^{2}, \phi_{i}^{2}, \phi_{i}^{3}, \phi_{ij}^{1}, \phi_{ij}^{2},\phi_{ijk}^{0} \},
\end{equation}
where
\begin{equation} \label{eq:2nd_phi}
\begin{gathered}
\phi^{2} = \sum_{n=2}^{+\infty} c_0^{2,n} \psi^n, \quad
\phi_{i}^{2} = \sum_{n=2}^{+\infty} c_{1}^{2,n} \left(\psi_{i}^{n} - \frac{\beta_1^{(1),n}}{\beta_1^{(1),1}} \psi^1_i\right), \quad
\phi_{i}^{3} = \sum_{n=2}^{+\infty} c_{1}^{3,n} \left(\psi_{i}^{n} - \frac{\beta_1^{(1),n}}{\beta_1^{(1),1}} \psi^1_i\right) \\
\phi_{ij}^{1} = \sum_{n=1}^{+\infty} c_{2}^{1,n} \left(\psi_{ij}^{n} - \frac{\beta_2^{(1),n}}{\beta_2^{(1),0}} \psi^0_{ij}\right), \quad \phi_{ij}^{2} = \sum_{n=1}^{+\infty} c_{2}^{2,n} \left(\psi_{ij}^{n} - \frac{\beta_2^{(1),n}}{\beta_2^{(1),0}} \psi^0_{ij}\right), \quad
\phi_{ijk}^{0} = \sum_{n=0}^{+\infty} c_3^{0,n} \psi_{ijk}^n.
\end{gathered}
\end{equation}
The coefficients $c_k^{\ell,n}$ can be represented as functions of $\beta_k^{(\ell), n}$, and the details can be found in the appendix \ref{sec:c}. Here we only emphasize that the choice has been made such that $\langle \phi_i^2, \phi_i^3\rangle = \langle \phi_{ij}^1, \phi_{ij}^2 \rangle = 0$, and in the case of Maxwell molecules, $\phi_i^{2,3} = \psi_i^{2,3}$ and $\phi_{ij}^{1,2} = \psi_{ij}^{1,2}$. The result \eqref{eq:V2} indicates that the minimum number of second-order variables is 24 ($= 1+3+3+5+5+7$). These variables will be denoted by
\begin{equation}
\label{eq:2nd-order}
\begin{gathered}
u^2 := \langle \phi^2, f \rangle, \qquad
u_i^2 := \langle \phi_i^2, f \rangle, \qquad
u_i^3 := \langle \phi_i^3, f \rangle, \\
u_{ij}^1 := \langle \phi_{ij}^1, f \rangle, \qquad
u_{ij}^2 := \langle \phi_{ij}^2, f \rangle, \qquad
u_{ijk}^0 := \langle \phi_{ijk}^0, f \rangle.
\end{gathered}
\end{equation}

\begin{remark} \rm
A similar procedure has been applied in \cite{yang2024siap} to obtain second-order variables in the case of steady-state equations. When time derivatives are absent, we have simpler relationships $w_i^{n|2} = d_{12}^n w_i^{2|2}$ and $w_{ij}^{n|2} = d_{21}^n w_{ij}^{1|2}$ in place of 
\eqref{eq:linear_rel}, resulting in only 16 second-order variables. In \cite{Struchtrup2013pof}, when time-dependent R13 equations were derived for hard-sphere molecules, only $29$ ($= 13+ 16$) equations are considered. However, the resulting equations may not hold the symmetry as in \eqref{eq:moment}. In this paper, we are going to include all the $37$ ($= 13 + 24$) equations in our derivation to achieve the $L^2$ stability.
\end{remark}

\subsection{Derivation of R13 equations} \label{sec:R13_derivation}
By the analysis above, the distribution function $f$ can be separated into four parts according to the order of magnitude:
\begin{displaymath}
f = f^{(0)} + f^{(1)} + f^{(2)} + f^{(\mathrm{r})},
\end{displaymath}
where $f^{(k)}$ is the projection of $f$ onto the function space $\mathbb{V}^{(k)}$, defined by
\begin{gather*}
\mathbb{V}^{(0)} = \operatorname{span}\{\psi^0, \psi^1, \psi_i^0\}, \qquad
\mathbb{V}^{(1)} = (\mathbb{V}^{(0)})^{\perp} \cap (\mathbb{S}^{(2)})^{\perp}, \\
\mathbb{V}^{(2)} = \mathbb{S}^{(2)} \cap (\mathbb{S}^{(3)})^{\perp}, \qquad \mathbb{V}^{(\mathrm{r})} = \mathbb{S}^{(3)}.
\end{gather*}
It is clear that $f^{(0)}$, $f^{(1)}$ and $f^{(2)}$ represent, respectively, the zeroth-, first- and second-order parts of $f$, and $f^{(\mathrm{r})}$ represents the remaining part of $f$ that has order $o(\kn^2)$. For simplicity, we use $\mathcal{P}^{(k)}$ to denote the projection operator onto $\mathbb{V}^{(k)}$. A key observation in our derivation is that the super-Burnett equations can be derived from the following equations without using $f^{(\mathrm{r})}$:
\begin{equation}
\label{eq:R13_abstract}
\begin{split}
\frac{\partial}{\partial t} \begin{pmatrix} f^{(0)} \\ f^{(1)} \\ 0 \end{pmatrix} +
\begin{pmatrix} \partial_{x_i} \mathcal{P}^{(0)} \xi_i f^{(0)} + 
\partial_{x_i} \mathcal{P}^{(0)} \xi_i f^{(1)} + 
\partial_{x_i} \mathcal{P}^{(0)} \xi_i f^{(2)} \\
\partial_{x_i} \mathcal{P}^{(1)} \xi_i f^{(0)} + 
\partial_{x_i} \mathcal{P}^{(1)} \xi_i f^{(1)} + 
\partial_{x_i} \mathcal{P}^{(1)} \xi_i f^{(2)} \\
\partial_{x_i} \mathcal{P}^{(2)} \xi_i f^{(0)} + 
\partial_{x_i} \mathcal{P}^{(2)} \xi_i f^{(1)}
\end{pmatrix} \qquad \\
=
\frac{1}{\kn} \begin{pmatrix}
0 \\
\mathcal{P}^{(1)} \mathcal{L} f^{(1)} + 
\mathcal{P}^{(1)} \mathcal{L} f^{(2)} \\
\mathcal{P}^{(2)} \mathcal{L} f^{(1)} + 
\mathcal{P}^{(2)} \mathcal{L} f^{(2)}
\end{pmatrix}.
\end{split}
\end{equation}
This form is similar to Grad's moment equations with truncation up to the second order, but the temporal and spatial derivatives of $f^{(2)}$ in the last equation are removed, allowing us to express $f^{(2)}$ using $f^{(0)}$ and $f^{(1)}$ according to the last equation:
\begin{displaymath}
f^{(2)} = \left[\mathcal{P}^{(2)}\mathcal{L}|_{\mathbb{V}^{(2)}} \right]^{-1} \left( \partial_{x_i} \mathcal{P}^{(2)} \xi_i f^{(0)} + 
\partial_{x_i} \mathcal{P}^{(2)} \xi_i f^{(1)} - \mathcal{P}^{(2)} \mathcal{L} f^{(1)} \right).
\end{displaymath}
Thus, the second-order part $f^{(2)}$ can be eliminated from the equations, and the resulting equations contain only $f^{(0)}$ and $f^{(1)}$, which can be represented as 13-moment equations.

The reason why \eqref{eq:R13_abstract} has the super-Burnett order will be detailed in the next subsection. Here we will apply this result to derive the R13 equations. The derivation requires explicit formulations of $f^{(0)}$, $f^{(1)}$ and $f^{(2)}$:
\begin{equation}
\label{eq:f012}
\begin{aligned}
f^{(0)} &= \rho \psi^{0} - \sqrt{\frac{3}{2}} \theta \psi^{1} + \sqrt{3} v_i \psi_i^{0}, \\
 f^{(1)} &= \frac{2}{5} \bar{q}_i \phi_i^{1} + \frac{1}{2} \bar{\sigma}_{ij} \phi_{ij}^{0}, \\
f^{(2)} &= u^{2} \phi^{2} + 3u_i^{2} \phi_i^{2} + 3u_i^{3} \phi_i^{3}+ \frac{15}{2}u_{ij}^{1} \phi_{ij}^{1} + \frac{15}{2}u_{ij}^{2} \phi_{ij}^{2}  + \frac{35}{2}u_{ijk}^{0} \phi_{ijk}^{0}.
\end{aligned}
\end{equation}
Then, the equations \eqref{eq:R13_abstract} can be converted to moment equations by the Galerkin method. For example, the equation of $\bar{\sigma}_{ij}$ can be obtained by
\begin{displaymath}
\langle \phi_{ij}^0, f^{(1)} \rangle + \frac{\partial}{\partial x_k} \langle \phi_{ij}^0, \xi_k (f^{(0)} + f^{(1)} + f^{(2)}) \rangle = \langle \phi_{ij}^0, \mathcal{L} (f^{(1)} + f^{(2)}) \rangle,
\end{displaymath}
and the equation of $u_i^2$ is
\begin{displaymath}
\frac{\partial}{\partial x_k} \langle \phi_i^2, \xi_k(f^{(0)} + f^{(1)}) \rangle = \langle \phi_i^2, \mathcal{L}(f^{(1)} + f^{(2)}) \rangle.
\end{displaymath}
All these inner products can be computed explicitly, and the complete 37-moment equations are
\begin{align}
    \label{eqa}
      &\frac{\partial \rho}{\partial t}+ \frac{\partial v_j}{\partial x_j}  = 0,  \\ 
    \label{eqb}
     &\frac{\partial \theta}{\partial t} + \frac{2}{3}  \frac{\partial v_j}{\partial x_j}  + \frac{2}{3}  c_1^{1,1}\frac{\partial \bar{q}_j}{\partial x_j} - \sqrt{\frac{10}{3}}\left( c^{2,1}_1 \frac{\partial u^{2}_j}{\partial x_j} + c^{3,1}_1 \frac{\partial u^{3}_j}{\partial x_j} \right)   = 0 ,\\ 
    \label{eqc}
     &\frac{\partial v_i}{\partial t} + \frac{\partial \rho}{\partial x_i} + \frac{\partial \theta}{\partial x_i} + c_2^{0,0} \frac{\partial \bar{\sigma}_{ij}}{\partial x_j} + \sqrt{15} \left( c^{1,0}_2 \frac{\partial u^{1}_{ij}}{\partial x_j} +    c^{2,0}_2 \frac{\partial u^{2}_{ij}}{\partial x_j} \right)  = 0, \\ 
      & \frac{\partial \bar{q}_i}{\partial t} + \frac{5}{2}c_1^{1,1} \frac{\partial \theta}{\partial x_i} - \frac{\sqrt{2}}{6} A_{45} \frac{\partial \bar{\sigma}_{ij}}{\partial x_j} - \sqrt{\frac{5}{6}} \left( A_{49}\frac{\partial u^{1}_{ij}}{\partial x_j} + A_{4,10}\frac{\partial u^{2}_{ij}}{\partial x_j} + A_{46} \frac{\partial u^{2}}{\partial x_i} \right) \notag \\
      & \hspace{120pt}= \frac{1}{\kn}( \frac{1}{3} \mathscr{L}^{(11)}_1 \bar{q}_{i} - \sqrt{\frac{5}{6}} \mathscr{L}^{(12)}_1 u^{2}_{i} - \sqrt{\frac{5}{6}} \mathscr{L}^{(13)}_1 u^{3}_{i}), \label{eqd1} \\ 
      & \frac{\partial \bar{\sigma}_{ij}}{\partial t} + 2 c_2^{0,0} \frac{\partial v_{\langle i}}{\partial x_{j \rangle}}  -\frac{2\sqrt{2}}{15} A_{45}\frac{\partial \bar{q}_{\langle i}}{\partial x_{j \rangle}} + \frac{2}{\sqrt{15}} \left( A_{57}\frac{\partial u^{2}_{\langle i}}{\partial x_{j \rangle}} + A_{58}\frac{\partial u^{3}_{\langle i}}{\partial x_{j \rangle}} + A_{5,11} \frac{\partial u^{0}_{ijk}}{\partial x_k} \right) \notag \\
      & \hspace{100pt}= \frac{1}{\kn}(\frac{2}{15} \mathscr{L}^{(00)}_2 \bar{\sigma}_{ij}+ \frac{2}{\sqrt{15}} \mathscr{L}^{(01)}_2 u^{1}_{ij}+ \frac{2}{\sqrt{15}} \mathscr{L}^{(02)}_2 u^{2}_{ij}),  \label{eqe1} \\ 
      \label{eqf}
      & -\sqrt{\frac{2}{15}} A_{46} \frac{\partial \bar{q}_j}{\partial x_j}   =  \frac{1}{\kn}\mathscr{L}^{(22)}_0 u^{2},\\ 
      \label{eqd2}
          &    -\sqrt{\frac{15}{2}} c^{2,1}_1 \frac{\partial \theta}{\partial x_i} + \sqrt{\frac{1}{15}} A_{57} \frac{\partial \bar{\sigma}_{ij}}{\partial x_j}  = \frac{1}{\kn}(-\sqrt{\frac{2}{15}} \mathscr{L}^{(21)}_1 \bar{q}_{i}+\mathscr{L}^{(22)}_1 u^{2}_{i}+\mathscr{L}^{(23)}_1 u^{3}_{i}),\\ 
                \label{eqd3}
           &   -\sqrt{\frac{15}{2}} c^{3,1}_1 \frac{\partial \theta}{\partial x_i} + \sqrt{\frac{1}{15}} A_{58} \frac{\partial \bar{\sigma}_{ij}}{\partial x_j} = \frac{1}{\kn}(-\sqrt{\frac{2}{15}}\mathscr{L}^{(31)}_1 \bar{q}_{i}+\mathscr{L}^{(32)}_1 u^{2}_{i}+\mathscr{L}^{(33)}_1 u^{3}_{i}),\\ 
        \label{eqe2}
        &  \sqrt{15} c^{1,0}_2 \frac{\partial v_{\langle i}}{\partial x_{j \rangle}} -\sqrt{\frac{2}{15}} A_{49}\frac{\partial \bar{q}_{\langle i}}{\partial x_{j \rangle}}   = \frac{1}{\kn}(\sqrt{\frac{1}{15}} \mathscr{L}^{(10)}_2 \bar{\sigma}_{ij}+\mathscr{L}^{(11)}_2 u^{1}_{ij}+\mathscr{L}^{(12)}_2 u^{2}_{ij}),\\ 
        \label{eqe3}
        &  \sqrt{15} c^{2,0}_2 \frac{\partial v_{\langle i}}{\partial x_{j \rangle}} -\sqrt{\frac{2}{15}} A_{4,10}\frac{\partial \bar{q}_{\langle i}}{\partial x_{j \rangle}}   = \frac{1}{\kn}(\sqrt{\frac{1}{15}} \mathscr{L}^{(20)}_2 \bar{\sigma}_{ij}+\mathscr{L}^{(21)}_2 u^{1}_{ij}+\mathscr{L}^{(22)}_2 u^{2}_{ij}),\\
       \label{eqg}
       & \sqrt{\frac{1}{15}} A_{5,11} \frac{\partial \bar{\sigma}_{\langle ij}}{\partial x_{k \rangle}}  = \frac{1}{\kn} \mathscr{L}^{(00)}_3 u^{0}_{ijk}. 
 \end{align}
The coefficients $A_{ij}$ are series of $c_k^{l,n}$, with their expressions provided in the appendix \ref{sec:appendixAcoeff}. 

It is now clear that the second-order variables defined in \eqref{eq:2nd-order} can be solved according to the equations \eqref{eqf}--\eqref{eqg}. Plugging the results into \eqref{eqb}--\eqref{eqe1} yields our final equations \eqref{newr13:eqa}--\eqref{newr13:eqe1}. Here we would like to write down the explicit expressions for a few coefficients in \eqref{newr13:eqa}--\eqref{newr13:eqe1}:
\begin{gather*}
k_1 = -5\begin{pmatrix} c_1^{2,1} & c_1^{3,1}
\end{pmatrix} \begin{pmatrix}
\mathscr{L}_1^{(22)} & \mathscr{L}_1^{(23)} \\
\mathscr{L}_1^{(32)} & \mathscr{L}_1^{(33)}
\end{pmatrix}^{-1}\begin{pmatrix} c_1^{2,1} \\ c_1^{3,1}
\end{pmatrix}, \\
k_2 = -\frac{\sqrt{2}}{3} \begin{pmatrix} c_1^{2,1} & c_1^{3,1}
\end{pmatrix} \begin{pmatrix}
\mathscr{L}_1^{(22)} & \mathscr{L}_1^{(23)} \\
\mathscr{L}_1^{(32)} & \mathscr{L}_1^{(33)}
\end{pmatrix}^{-1}\begin{pmatrix} A_{57} \\ A_{58}
\end{pmatrix}, \\
k_3 = -15 \begin{pmatrix} c_2^{1,0} & c_2^{2,0}
\end{pmatrix} \begin{pmatrix}
\mathscr{L}_2^{(11)} & \mathscr{L}_2^{(12)} \\
\mathscr{L}_2^{(21)} & \mathscr{L}_2^{(22)}
\end{pmatrix}^{-1}\begin{pmatrix} c_2^{1,0} \\ c_2^{2,0}
\end{pmatrix}, \\
k_4 = \sqrt{2} \begin{pmatrix} A_{49} & A_{4,10}
\end{pmatrix} \begin{pmatrix}
\mathscr{L}_2^{(11)} & \mathscr{L}_2^{(12)} \\
\mathscr{L}_2^{(21)} & \mathscr{L}_2^{(22)}
\end{pmatrix}^{-1}\begin{pmatrix} c_2^{1,0} \\ c_2^{2,0}
\end{pmatrix}, \\
k_7 = -\frac{5}{36} \begin{pmatrix} A_{49} & A_{4,10}
\end{pmatrix} \begin{pmatrix}
\mathscr{L}_2^{(11)} & \mathscr{L}_2^{(12)} \\
\mathscr{L}_2^{(21)} & \mathscr{L}_2^{(22)}
\end{pmatrix}^{-1}\begin{pmatrix} A_{49} \\ A_{4,10}
\end{pmatrix}, \\
k_{10} = -\frac{2}{15}
\begin{pmatrix} A_{57} & A_{58}
\end{pmatrix} \begin{pmatrix}
\mathscr{L}_1^{(22)} & \mathscr{L}_1^{(23)} \\
\mathscr{L}_1^{(32)} & \mathscr{L}_1^{(33)}
\end{pmatrix}^{-1}\begin{pmatrix} A_{57} \\ A_{58}
\end{pmatrix}.
\end{gather*}
Since the matrices 
\begin{displaymath} \begin{pmatrix}
\mathscr{L}_1^{(22)} & \mathscr{L}_1^{(23)} \\
\mathscr{L}_1^{(32)} & \mathscr{L}_1^{(33)}
\end{pmatrix} \quad \text{and} \quad
\begin{pmatrix}
\mathscr{L}_2^{(11)} & \mathscr{L}_2^{(12)} \\
\mathscr{L}_2^{(21)} & \mathscr{L}_2^{(22)}
\end{pmatrix}
\end{displaymath}
are symmetric and negative definite, we can obtain \eqref{eq:k_ineq} by the Cauchy-Schwarz inequality, which completes the proof of the H-theorem \eqref{eq:H-theorem}.

\subsection{Verification of the super-Burnett order}
To show that the equations \eqref{eq:R13_abstract} have super-Burnett order, we rewrite the Boltzmann equation in the form
\begin{displaymath}
\frac{\partial}{\partial t} \begin{pmatrix} f^{(0)} \\ f^{(1)} \\ f^{(2)} \\ f^{(\mathrm{r})} \end{pmatrix} + \begin{pmatrix}
  \mathcal{A}^{(00)} &  \mathcal{A}^{(01)} &  \mathcal{A}^{(02)} &  \mathcal{A}^{(0\mathrm{r})} \\
  \mathcal{A}^{(10)} &  \mathcal{A}^{(11)} &  \mathcal{A}^{(12)} &  \mathcal{A}^{(1\mathrm{r})} \\
  \mathcal{A}^{(20)} &  \mathcal{A}^{(21)} &  \mathcal{A}^{(22)} &  \mathcal{A}^{(2\mathrm{r})} \\
  \mathcal{A}^{(\mathrm{r}0)} &  \mathcal{A}^{(\mathrm{r}1)} &  \mathcal{A}^{(\mathrm{r}2)} &  \mathcal{A}^{(\mathrm{rr})}
\end{pmatrix} \begin{pmatrix} f^{(0)} \\ f^{(1)} \\ f^{(2)} \\ f^{(\mathrm{r})} \end{pmatrix} =
\frac{1}{\kn} \begin{pmatrix}
  0 & 0 & 0 & 0 \\
  0 & \mathcal{L}^{(11)} &  \mathcal{L}^{(12)} &  \mathcal{L}^{(1\mathrm{r})} \\
  0 & \mathcal{L}^{(21)} &  \mathcal{L}^{(22)} &  \mathcal{L}^{(2\mathrm{r})} \\
  0 & \mathcal{L}^{(\mathrm{r}1)} &  \mathcal{L}^{(\mathrm{r}2)} &  \mathcal{L}^{(\mathrm{rr})}
\end{pmatrix} \begin{pmatrix} f^{(0)} \\ f^{(1)} \\ f^{(2)} \\ f^{(\mathrm{r})} \end{pmatrix},
\end{displaymath}
where
\begin{displaymath}
\mathcal{A}^{(kl)} f^{(l)} = \frac{\partial}{\partial x_i} \mathcal{A}_i^{(kl)} f^{(l)}, \qquad \mathcal{A}_i^{(kl)} = \mathcal{P}^{(k)} \xi_i \mathcal{P}^{(l)}, \qquad
\mathcal{L}^{(kl)} f^{(l)} = \mathcal{P}^{(k)} \mathcal{L} f^{(l)},
\end{displaymath}
and they satisfy
\begin{equation} \label{eq:adjoint}
\langle g^{(k)}, \mathcal{A}_i^{(kl)} g^{(l)} \rangle = \langle \mathcal{A}_i^{(lk)} g^{(k)}, g^{(l)} \rangle, \qquad
\langle g^{(k)}, \mathcal{L}^{(kl)} g^{(l)} \rangle = \langle \mathcal{L}^{(lk)} g^{(k)}, g^{(l)} \rangle
\end{equation}
for all $g^{(k)} \in \mathbb{V}^{(k)}$ and $g^{(l)} \in \mathbb{V}^{(l)}$.

It is well-known that the super-Burnett equations can be derived via three \emph{Maxwellian iterations} applied to the Boltzmann equation, which begin with the initial value
\begin{displaymath}
  f_0^{(1)} = f_0^{(2)} = f_0^{(r)} = 0,
\end{displaymath}
and proceeds from the $j$th iteration to the $(j+1)$th iteration according to
\begin{equation} \label{eq:max_iter}
\frac{\partial}{\partial t} \begin{pmatrix} f_j^{(1)} \\ f_j^{(2)} \\ f_j^{(\mathrm{r})} \end{pmatrix} + \begin{pmatrix}
  \mathcal{A}^{(10)} &  \mathcal{A}^{(11)} &  \mathcal{A}^{(12)} &  \mathcal{A}^{(1\mathrm{r})} \\
  \mathcal{A}^{(20)} &  \mathcal{A}^{(21)} &  \mathcal{A}^{(22)} &  \mathcal{A}^{(2\mathrm{r})} \\
  \mathcal{A}^{(\mathrm{r}0)} &  \mathcal{A}^{(\mathrm{r}1)} &  \mathcal{A}^{(\mathrm{r}2)} &  \mathcal{A}^{(\mathrm{rr})}
\end{pmatrix} \begin{pmatrix} f^{(0)} \\ f_j^{(1)} \\ f_j^{(2)} \\ f_j^{(\mathrm{r})} \end{pmatrix} =
\frac{1}{\kn} \begin{pmatrix}
  \mathcal{L}^{(11)} &  \mathcal{L}^{(12)} &  \mathcal{L}^{(1\mathrm{r})} \\
  \mathcal{L}^{(21)} &  \mathcal{L}^{(22)} &  \mathcal{L}^{(2\mathrm{r})} \\
  \mathcal{L}^{(\mathrm{r}1)} &  \mathcal{L}^{(\mathrm{r}2)} &  \mathcal{L}^{(\mathrm{rr})}
\end{pmatrix} \begin{pmatrix} f_{j+1}^{(1)} \\[2pt] f_{j+1}^{(2)} \\[2pt] f_{j+1}^{(\mathrm{r})} \end{pmatrix}.
\end{equation}
The super-Burnett equations can be written in the form
\begin{equation} \label{eq:superBurnett}
\frac{\partial f^{(0)}}{\partial t} + \mathcal{A}^{(00)} f^{(0)} +  \mathcal{A}^{(01)} f_3^{(1)} +  \mathcal{A}^{(02)} f_3^{(2)} +  \mathcal{A}^{(02)} f_3^{(3)} = 0.
\end{equation}
To show that the equations \eqref{eq:R13_abstract} have the super-Burnett order, we just need to demonstrate that the derivation of super-Burnett equations does not involve any operators that do not appear in \eqref{eq:R13_abstract}. Below, we will carry out the derivation step by step in the following subsections.

\subsubsection{First Maxwellian iteration}
Setting $j = 0$ in \eqref{eq:max_iter} and using the initial data $f_0^{(1)} = f_0^{(2)} = f_0^{(\mathrm{r})} = 0$, we obtain
\begin{displaymath}
\begin{pmatrix}
  \mathcal{A}^{(10)} f^{(0)} \\
  \mathcal{A}^{(20)} f^{(0)} \\
  \mathcal{A}^{(\mathrm{r}0)} f^{(0)}
\end{pmatrix} =
\frac{1}{\kn} \begin{pmatrix}
  \mathcal{L}^{(11)} &  \mathcal{L}^{(12)} &  \mathcal{L}^{(1\mathrm{r})} \\
  \mathcal{L}^{(21)} &  \mathcal{L}^{(22)} &  \mathcal{L}^{(2\mathrm{r})} \\
  \mathcal{L}^{(\mathrm{r}1)} &  \mathcal{L}^{(\mathrm{r}2)} &  \mathcal{L}^{(\mathrm{rr})}
\end{pmatrix} \begin{pmatrix} f_1^{(1)} \\[2pt] f_1^{(2)} \\[2pt] f_1^{(\mathrm{r})} \end{pmatrix}.
\end{displaymath}
By Gaussian elimination, 
\begin{equation}
\begin{pmatrix}
  \mathcal{A}^{(10)} f^{(0)} \\
  \left[\mathcal{A}^{(20)} - \mathcal{L}^{(21)} (\mathcal{L}^{(11)})^{-1} \mathcal{A}^{(10)}\right] f^{(0)} \\
  \left[\mathcal{A}^{(\mathrm{r}0)} - \mathcal{L}^{(\mathrm{r}1)} (\mathcal{L}^{(11)})^{-1} \mathcal{A}^{(10)}\right] f^{(0)} \\
\end{pmatrix} =
\frac{1}{\kn} \begin{pmatrix}
  \mathcal{L}^{(11)} &  \mathcal{L}^{(12)} &  \mathcal{L}^{(1\mathrm{r})} \\
  & \mathcal{L}_*^{(22)} & \mathcal{L}_*^{(2\mathrm{r})} \\
  & \mathcal{L}_*^{(\mathrm{r}2)} &  \mathcal{L}_*^{(\mathrm{rr})}
\end{pmatrix} \begin{pmatrix} f_1^{(1)} \\[2pt] f_1^{(2)} \\[2pt] f_1^{(\mathrm{r})} \end{pmatrix},
\end{equation}
where
\begin{gather*}
\mathcal{L}_*^{(22)} = \mathcal{L}^{(22)} - \mathcal{L}^{(21)} (\mathcal{L}^{(11)})^{-1} \mathcal{L}^{(12)}, \qquad
\mathcal{L}_*^{(2\mathrm{r})} = \mathcal{L}^{(2\mathrm{r})} - \mathcal{L}^{(21)} (\mathcal{L}^{(11)})^{-1} \mathcal{L}^{(1\mathrm{r})}, \\
\mathcal{L}_*^{(\mathrm{r}2)} = \mathcal{L}^{(\mathrm{r}2)} - \mathcal{L}^{(\mathrm{r}1)} (\mathcal{L}^{(11)})^{-1} \mathcal{L}^{(12)}, \qquad
\mathcal{L}_*^{(\mathrm{rr})} = \mathcal{L}^{(\mathrm{rr})} - \mathcal{L}^{(\mathrm{r}1)} (\mathcal{L}^{(11)})^{-1} \mathcal{L}^{(1\mathrm{r})}.
\end{gather*}
Note that throughout the iterations, we always have $f_j^{(2)} \sim o(\kn)$ and $f_j^{(r)} \sim o(\kn^2)$, which requires
\begin{equation} \label{eq:op_relationship}
  \mathcal{A}^{(20)} - \mathcal{L}^{(21)} (\mathcal{L}^{(11)})^{-1} \mathcal{A}^{(10)} =
  \mathcal{A}^{(\mathrm{r}0)} - \mathcal{L}^{(\mathrm{r}1)} (\mathcal{L}^{(11)})^{-1} \mathcal{A}^{(10)} = 0.
\end{equation}
Thus, the result of the first Maxwellian iteration can be written as
\begin{displaymath}
f_1^{(1)} = \kn (\mathcal{L}^{(11)})^{-1} \mathcal{A}^{(10)} f^{(0)}, \qquad f_1^{(2)} = 0, \qquad f_1^{(\mathrm{r})} = 0.
\end{displaymath}
For simplicity, we define $\mathcal{S}_1^{(10)} = (\mathcal{L}^{(11)})^{-1} \mathcal{A}^{(10)}$, so that $f_1^{(1)} = \kn \mathcal{S}_1^{(10)} f^{(0)}$.

\subsubsection{Second Maxwellian iteration}
Setting $j = 0$ in \eqref{eq:max_iter} and using the result of the first Maxwellian iteration, we get
\begin{displaymath}
\begin{split}
\frac{\partial}{\partial t} \begin{pmatrix}
  \kn \mathcal{S}_1^{(10)} f^{(0)} \\
  0 \\ 0
\end{pmatrix} +
\begin{pmatrix}
  \left( \mathcal{A}^{(10)} + \kn \mathcal{A}^{(11)} \mathcal{S}_1^{(10)} \right) f^{(0)} \\
  \left( \mathcal{A}^{(20)} + \kn \mathcal{A}^{(21)} \mathcal{S}_1^{(10)} \right) f^{(0)} \\
  \left( \mathcal{A}^{(\mathrm{r}0)} + \kn \mathcal{A}^{(\mathrm{r}1)} \mathcal{S}_1^{(10)} \right) f^{(0)}
\end{pmatrix}
= \frac{1}{\kn} \begin{pmatrix}
  \mathcal{L}^{(11)} &  \mathcal{L}^{(12)} &  \mathcal{L}^{(1\mathrm{r})} \\
  \mathcal{L}^{(21)} &  \mathcal{L}^{(22)} &  \mathcal{L}^{(2\mathrm{r})} \\
  \mathcal{L}^{(\mathrm{r}1)} &  \mathcal{L}^{(\mathrm{r}2)} &  \mathcal{L}^{(\mathrm{rr})}
\end{pmatrix} \begin{pmatrix} f_2^{(1)} \\[2pt] f_2^{(2)} \\[2pt] f_2^{(\mathrm{r})} \end{pmatrix}.
\end{split}
\end{displaymath}
Next, we perform the following operations:
\begin{enumerate}
\item Replace the time derivative with the spatial derivative using
\begin{displaymath}
\frac{\partial f^{(0)}}{\partial t} + \mathcal{A}^{(00)} f^{(0)} = O(\kn),
\end{displaymath}
where the $O(\kn)$ part can be discarded and it does not affect the order of accuracy.
\item Apply Gaussian elimination and \eqref{eq:op_relationship} to obtain
\begin{equation} \label{eq:ge}
\begin{split}
\begin{pmatrix}
\left[\mathcal{A}^{(10)} + \kn \mathcal{A}^{(11)} \mathcal{S}_1^{(10)} - \kn \mathcal{S}_1^{(10)} \mathcal{A}^{(00)}\right] f^{(0)} \\
\kn \left[ \mathcal{A}^{(21)} \mathcal{S}_1^{(10)} - \mathcal{L}^{(21)} (\mathcal{L}^{(11)})^{-1} \left( \mathcal{A}^{(11)} \mathcal{S}_1^{(10)} - \mathcal{S}_1^{(10)} \mathcal{A}^{(00)} \right) \right] f^{(0)} \\
\kn \left[ \mathcal{A}^{(\mathrm{r}1)} \mathcal{S}_1^{(10)} - \mathcal{L}^{(\mathrm{r}1)} (\mathcal{L}^{(11)})^{-1} \left( \mathcal{A}^{(11)} \mathcal{S}_1^{(10)} - \mathcal{S}_1^{(10)} \mathcal{A}^{(00)} \right) \right] f^{(0)} \\
\end{pmatrix} \qquad \\
= \frac{1}{\kn} \begin{pmatrix}
  \mathcal{L}^{(11)} &  \mathcal{L}^{(12)} &  \mathcal{L}^{(1\mathrm{r})} \\
  & \mathcal{L}^{(22)}_* &  \mathcal{L}^{(2\mathrm{r})}_* \\
  &  \mathcal{L}^{(\mathrm{r}2)}_* &  \mathcal{L}^{(\mathrm{rr})}_*
\end{pmatrix} \begin{pmatrix} f_2^{(1)} \\[2pt] f_2^{(2)} \\[2pt] f_2^{(\mathrm{r})} \end{pmatrix}.
\end{split}
\end{equation}
\item Apply Gaussian elimination again to eliminate $\mathcal{L}_*^{(\mathrm{r}2)}$. Then the last equation in the system becomes
\begin{displaymath}
\kn \mathcal{B}^{(\mathrm{r0})} f^{(0)} = \frac{1}{\kn} \left( \mathcal{L}_*^{(\mathrm{rr})} - \mathcal{L}_*^{(\mathrm{r2})} (\mathcal{L}_*^{(22)})^{-1} \mathcal{L}_*^{(\mathrm{2r})} \right) f_2^{(\mathrm{r})},
\end{displaymath}
with
\begin{displaymath}
\begin{split}
\mathcal{B}^{(\mathrm{r}0)} &= \left( \mathcal{A}^{(\mathrm{r1})} - \mathcal{L}_*^{(\mathrm{r2})} (\mathcal{L}_*^{(22)})^{-1} \mathcal{A}^{(\mathrm{21})} \right) \mathcal{S}_1^{(10)} \\
& \quad - \left( \mathcal{L}^{(\mathrm{r1})} - \mathcal{L}_*^{(\mathrm{r2})} (\mathcal{L}_*^{(22)})^{-1} \mathcal{L}^{(\mathrm{21})} \right) (\mathcal{L}^{(11)})^{-1} \left( \mathcal{A}^{(11)} \mathcal{S}_1^{(10)} - \mathcal{S}_1^{(10)} \mathcal{A}^{(00)} \right).
\end{split}
\end{displaymath}
Due to the fact that $f_2^{(\mathrm{r})} \sim o(\kn^2)$, the operator $\mathcal{B}^{(\mathrm{r}0)}$ has to be zero, and thus $f_2^{(\mathrm{r})} = 0$, and $f_2^{(1)}$ and $f_2^{(2)}$ can be solved from \eqref{eq:ge}. The results are
\begin{align*}
f_2^{(1)} 
&= \kn \mathcal{S}_1^{(10)} f^{(0)} - \kn^2 (\mathcal{L}^{(11)})^{-1} \mathcal{L}^{(12)}(\mathcal{L}_*^{(22)})^{-1} \mathcal{A}^{(11)} \mathcal{S}_1^{(10)} f^{(0)} \\
& \quad + \kn^2 (\mathcal{L}^{(11)})^{-1} \left( \mathcal{I} - \mathcal{L}^{(12)} (\mathcal{L}_*^{(22)})^{-1} \mathcal{L}^{(21)} (\mathcal{L}^{(11)})^{-1} \right) \left( \mathcal{A}^{(11)} \mathcal{S}_1^{(10)} - \mathcal{S}_1^{(10)} \mathcal{A}^{(00)} \right) f^{(0)}, \\
f_2^{(2)} &= \kn^2 (\mathcal{L}_*^{(22)})^{-1}\left[ \mathcal{A}^{(21)} \mathcal{S}_1^{(10)} - \mathcal{L}^{(21)} (\mathcal{L}^{(11)})^{-1} \left( \mathcal{A}^{(11)} \mathcal{S}_1^{(10)} - \mathcal{S}_1^{(10)} \mathcal{A}^{(00)} \right) \right] f^{(0)}.
\end{align*}
\end{enumerate}
For conciseness, below we will write these equations in the following simpler form:
\begin{displaymath}
f_2^{(1)} = \kn \mathcal{S}_1^{(10)} f^{(0)} + \kn^2 \mathcal{S}_2^{(10)} f^{(0)}, \qquad
f_2^{(2)} = \kn^2 \mathcal{S}_2^{(20)} f^{(0)}.
\end{displaymath}

\subsubsection{Third Maxwellian iteration}
Similarly, the third Maxwellian iteration requires to find $f_3^{(1)}$, $f_3^{(2)}$ and $f_3^{(3)}$ from \eqref{eq:max_iter} with $j = 2$. However, for the purpose of deriving super-Burnett equations, only the first equation is needed. Here we rewrite this equation below:
\begin{multline*}
\frac{\partial}{\partial t} \left(\kn \mathcal{S}_1^{(10)} f^{(0)} + \kn^2 \mathcal{S}_2^{(20)} f^{(0)} \right) \\
+ \mathcal{A}^{(10)} f^{(0)} + \kn \mathcal{A}^{(11)} \mathcal{S}_1^{(10)}f^{(0)} + \kn^2 \left( \mathcal{A}^{(11)} \mathcal{S}_2^{(10)} + \mathcal{A}^{(12)} \mathcal{S}_2^{(20)} \right) f^{(0)} \\
= \frac{1}{\kn} (\mathcal{L}^{(11)} f_3^{(1)} + \mathcal{L}^{(12)} f_3^{(2)} + \mathcal{L}^{(1\mathrm{r})} f_3^{(\mathrm{r})}).
\end{multline*}
Further simplification again requires the time derivatives to be replaced without affecting the order of magnitude. This can be done by using
\begin{displaymath}
\begin{split}
& \frac{\partial}{\partial t} \left(\kn \mathcal{S}_1^{(10)} f^{(0)} + \kn^2 \mathcal{S}_2^{(20)} f^{(0)} \right) = \\
& \qquad -\kn \mathcal{S}_1^{(10)} \mathcal{A}^{(00)} f^{(0)} - \kn^2 (\mathcal{S}_1^{(10)} \mathcal{A}^{(01)} \mathcal{S}_1^{(10)} + \mathcal{S}_2^{(20)} \mathcal{A}^{(00)}) f^{(0)} + O(\kn^3),
\end{split}
\end{displaymath}
which leads to
\begin{equation} \label{eq:3rd_iter}
\begin{split}
& \mathcal{L}^{(11)} f_3^{(1)} + \mathcal{L}^{(12)} f_3^{(2)} + \mathcal{L}^{(1\mathrm{r})} f_3^{(\mathrm{r})} \\
={} & \kn \mathcal{A}^{(10)} f^{(0)} + \kn^2 \left(\mathcal{A}^{(11)} \mathcal{S}_1^{(10)} - \mathcal{S}_1^{(10)} \mathcal{A}^{(00)} \right) f^{(0)} \\
& {}+ \kn^3 \left( \mathcal{A}^{(11)} \mathcal{S}_2^{(10)} + \mathcal{A}^{(12)} \mathcal{S}_2^{(20)} - \mathcal{S}_2^{(20)} \mathcal{A}^{(00)} - \mathcal{S}_1^{(10)} \mathcal{A}^{(01)} \mathcal{S}_1^{(10)} \right) f^{(0)}.
\end{split}
\end{equation}

\subsubsection{Super-Burnett equations}
By the adjointness \eqref{eq:adjoint} and the equalities \eqref{eq:op_relationship}, we have
\begin{displaymath}
\mathcal{A}^{(02)} = \mathcal{A}^{(01)} (\mathcal{L}^{(11)})^{-1} \mathcal{L}^{(12)}, \qquad
\mathcal{A}^{(0\mathrm{r})} = \mathcal{A}^{(01)} (\mathcal{L}^{(11)})^{-1} \mathcal{L}^{(1\mathrm{r})}.
\end{displaymath}
Thus, the super-Burnett equations \eqref{eq:superBurnett} can be reformulated as
\begin{displaymath}
\frac{\partial f^{(0)}}{\partial t} + \mathcal{A}^{(00)} f^{(0)} +  \mathcal{A}^{(01)} (\mathcal{L}^{(11)})^{-1} \left(\mathcal{L}^{(11)}f_3^{(1)} + \mathcal{L}^{(12)} f_3^{(2)} + \mathcal{L}^{(1\mathrm{r})} f_3^{(\mathrm{r})} \right) = 0.
\end{displaymath}
The final super-Burnett equations can be found by plugging \eqref{eq:3rd_iter} into the above equation.

Such a derivation of super-Burnett equations shows that the final equations depend only on the operators
\begin{displaymath}
\mathcal{A}^{(00)}, \mathcal{A}^{(01)}, \mathcal{A}^{(02)}, \mathcal{A}^{(10)}, \mathcal{A}^{(11)}, \mathcal{A}^{(12)}, \mathcal{A}^{(20)}, \mathcal{A}^{(21)}, \mathcal{L}^{(11)}, \mathcal{L}^{(12)}, \mathcal{L}^{(21)}, \mathcal{L}^{(22)},
\end{displaymath}
and the operators
\begin{equation} \label{eq:ops}
\mathcal{A}^{(0\mathrm{r})}, \mathcal{A}^{(1\mathrm{r})}, \mathcal{A}^{(22)}, \mathcal{A}^{(2\mathrm{r})}, \mathcal{A}^{(\mathrm{r}0)}, \mathcal{A}^{(\mathrm{r}1)}, \mathcal{A}^{(\mathrm{r}2)}, \mathcal{A}^{(\mathrm{rr})}, \mathcal{L}^{(1\mathrm{r})}, \mathcal{L}^{(2\mathrm{r})}, \mathcal{L}^{(\mathrm{r}1)}, \mathcal{L}^{(\mathrm{r}2)}, \mathcal{L}^{(\mathrm{rr})}
\end{equation}
do not appear in the final equations. Therefore, we can set these operators to be zero, and the resulting equations still have the super-Burnett order. Meanwhile, the time derivative $\partial_t f_2$ does not participate in the derivation, which can also be set to zero without affecting the super-Burnett order. 

\section{Onsager boundary conditions}
\label{sec:boundary}
In microflows, boundary conditions play a key role in describing the rarefaction effect. In this work, we focus mainly on the wall boundary condition, and we restrict ourselves to Maxwell's accommodation model in \cite{Maxwell1879bc}, where it is assumed that all gas molecules hitting the wall are reflected either specularly or diffusively, according to the \emph{accommodation coefficient} of the wall, which specifies the probability of diffusive reflection. In this section, we will follow the notations used in \eqref{eq:bcv} to \eqref{eq:bcm} and use $n$, $\tau_1$ and $\tau_2$ instead of $1$, $2$ and $3$ as indices of vectors or tensors. Thus, Maxwell's wall boundary conditions for the Boltzmann equation \eqref{eq:Boltzmann} can be formulated as
\begin{equation} \label{eq:Maxwell_bc}
f(\xi_n, \xi_{\tau_1}, \xi_{\tau_2}) = \chi f_W(\xi_n, \xi_{\tau_1}, \xi_{\tau_2}) + (1-\chi) f(-\xi_n, \xi_{\tau_1}, \xi_{\tau_2}), \qquad \xi_n < 0.
\end{equation}
Here $\chi$ is the accommodation coefficient, and we have omitted the spatial and temporal variables since the boundary condition is valid for any boundary point $\boldsymbol{x}$ and time $t$, and there are no spatial or temporal derivatives involved. The function $f_W$ describes the diffusive reflection, which depends on the wall velocity $\boldsymbol{v}_W$ and the wall temperature $\theta_W$. Under the linearized setting, it has the form
\begin{displaymath}
f_W(\boldsymbol{\xi}) = \rho_W + v_{W,\tau_1} \xi_{\tau_1} + v_{W,\tau_2} \xi_{\tau_2} + \frac{1}{2} \theta_W (|\boldsymbol{\xi}|^2 - 3),
\end{displaymath}
where 
\begin{equation}
\label{eq:rho_W}
\rho_W = \sqrt{2\pi} \langle (\xi_n)_+, f \rangle - \frac{\theta_W}{2},
\end{equation}
which is chosen such that $\langle \xi_n, f\rangle = 0$, meaning that the normal component of the velocity is zero.

In \cite{yang2024siap}, the boundary condition \eqref{eq:Maxwell_bc} has been equivalently reformulated to
\begin{equation}
\label{eq:bc}
\mathcal{P}_{\mathrm{odd}} f = \frac{2\chi}{2-\chi} \mathcal{P}_{\mathrm{odd}} \mathcal{C}(f_W - \mathcal{P}_{\mathrm{even}} f),
\end{equation}
where
\begin{align*}
\mathcal{P}_{\mathrm{odd}} g(\xi_n, \xi_{\tau_1}, \xi_{\tau_2}) &= 
  \frac{g(\xi_n, \xi_{\tau_1}, \xi_{\tau_2}) - g(-\xi_n, \xi_{\tau_1}, \xi_{\tau_2})}{2}, \\
\mathcal{P}_{\mathrm{even}} g(\xi_n, \xi_{\tau_1}, \xi_{\tau_2}) &= 
  \frac{g(\xi_n, \xi_{\tau_1}, \xi_{\tau_2}) + g(-\xi_n, \xi_{\tau_1}, \xi_{\tau_2})}{2}, \\
\mathcal{C} g(\xi_n, \xi_{\tau_1}, \xi_{\tau_2}) &= \left\{ \begin{array}{@{}ll}
  g(\xi_n, \xi_{\tau_1}, \xi_{\tau_2}) & \text{if } \xi_n < 0, \\[3pt]
  0 & \text{if } \xi_n > 0.
\end{array} \right.
\end{align*}
Such formulation of the boundary condition better fits Grad's theory of boundary conditions for moment equations \cite{Grad13_1949}, where it is required that all boundary conditions should only be imposed on odd moments. A direct application of Galerkin's method on \eqref{eq:bc} leads to Grad's boundary conditions. However, these boundary conditions may not satisfy the stability requirement \eqref{eq:stablebc}. Our boundary conditions \eqref{eq:bcv}--\eqref{eq:bcm} are based on the theory of Onsager boundary conditions developed in \cite{sarna2018stable,bunger2023KRM,yang2024siap}, which can be demonstrated to satisfy the general form \eqref{eq:r13bc}. In particular, we will follow the formulation in \cite{yang2024siap} that rewrites \eqref{eq:bc} as
\begin{equation}
\label{eq:bc_stable}
\mathcal{P}_{\mathrm{odd}} f = \frac{2\chi}{2-\chi} \mathcal{S} \xi_n (f_W - \mathcal{P}_{\mathrm{even}} f),
\end{equation}
where $\mathcal{S} = \mathcal{P}_{\mathrm{odd}} \mathcal{C} \xi_n^{-1}$. Thus, the operator $\frac{2\chi}{2-\chi} \mathcal{S}$  corresponds to the negative semidefinite matrix $\mathbf{Q}$, and $\xi_n$ corresponds to the matrix $\mathbf{A}_{\mathrm{oe}}$. More details on the discretization will be given in the following subsections.

\begin{remark}
Note that existing Onsager boundary conditions such as the ones derived in \cite{sarna2018stable} and \cite{yang2024siap} cannot be directly applied to our systems. The primary reason is that our system generally requires more boundary conditions than previous R13 equations, due to a larger number of moments considered in our derivation. For instance, the boundary conditions for R13 equations with Maxwell molecules consist of 10 equations at each boundary point \cite{sarna2018stable}. However, in our system, the equations of $\theta$ and $\vb$ become parabolic, requiring 3 more boundary conditions to determine the solution. This motivates us to derive the boundary conditions from scratch and guarantee reliable matching between the equations and boundary conditions.
\end{remark}

\subsection{Onsager boundary conditions}\label{subsec:Onsager}
The abstract form of the R13 equations  \eqref{eq:R13_abstract} shows that all operators in the Boltzmann equation are approximated by operators in $\mathbb{V} = \mathbb{V}^{(0)} \oplus \mathbb{V}^{(1)} \oplus \mathbb{V}^{(2)}$. If we rewrite the Boltzmann equation \eqref{eq:Boltzmann} as
\begin{equation}
\frac{\partial}{\partial t} (\mathcal{I} f) + \frac{\partial}{\partial x_k} (\xi_k f) = \frac{1}{\kn} \mathcal{L}f,
\end{equation}
then the identity operator $\mathcal{I}$ is approximated by $\mathcal{P}^{(01)} := \mathcal{P}^{(0)} + \mathcal{P}^{(1)}$, the velocity operator $\xi_k$ is approximated by
\begin{displaymath}
  \mathcal{A}_k = \mathcal{P}^{(01)} \xi_k \mathcal{P}^{(01)} + \mathcal{P}^{(2)} \xi_k \mathcal{P}^{(01)} + \mathcal{P}^{(01)} \xi_k \mathcal{P}^{(2)},
\end{displaymath}
and the collision operator $\mathcal{L}$ is approximated by $(\mathcal{P}^{(1)} + \mathcal{P}^{(2)}) \mathcal{L}(\mathcal{P}^{(1)} + \mathcal{P}^{(2)})$. Similarly, when deriving boundary conditions for moment equations, operators in \eqref{eq:bc_stable} are also approximated by using operators on $\mathbb{V}$.

We first study the operator $\mathcal{P}_{\mathrm{even}}$, which extracts the even part of a function with respect to $\xi_n$. The discretization of $\mathcal{P}_{\mathrm{even}}$ is simply its restriction on $\mathbb{V}$. Note that all the basis functions of $\mathbb{V}$ are given in \eqref{eq:f012}, from which it is not difficult to observe that
\begin{displaymath}
\begin{split}
& \mathbb{V}_{\mathrm{even}} := \mathcal{P}_{\mathrm{even}} \mathbb{V}
= \operatorname{span} \{ \psi^0, \psi^1, \psi_{\tau_1}^0, \psi_{\tau_2}^0; \phi_{\tau_1}^1, \phi_{\tau_2}^1, \phi_{\tau_1 \tau_1}^0, \phi_{\tau_2 \tau_2}^0, \phi_{\tau_1 \tau_2}^0; \\
& \quad \phi^2, \phi_{\tau_1}^2, \phi_{\tau_2}^2, \phi_{\tau_1}^3, \phi_{\tau_2}^3, \phi_{\tau_1 \tau_1}^1, \phi_{\tau_2 \tau_2}^1, \phi_{\tau_1 \tau_2}^1, \phi_{\tau_1 \tau_1}^2, \phi_{\tau_2 \tau_2}^2, \phi_{\tau_1 \tau_2}^2, \phi_{\tau_1 \tau_1 \tau_1}^0, \phi_{\tau_1 \tau_1 \tau_2}^0,\phi_{\tau_1 \tau_2 \tau_2}^0,\phi_{\tau_2 \tau_2 \tau_2}^0\}.
\end{split}
\end{displaymath}
It is clear that $f_W \in \mathbb{V}_{\mathrm{even}}$. Thus, the velocity operator $\xi_n$ should be discretized as the restriction of $\mathcal{A}_n$ on
$\mathbb{V}_{\mathrm{even}}$. We denote its range by $\mathbb{V}_{\mathrm{odd}}
:= \mathcal{A}_n\mathbb{V}_{\mathrm{even}}$. By straightforward calculation,
\begin{equation}
\label{eq:Vodd}
\begin{split}
\mathbb{V}_{\mathrm{odd}} = \operatorname{span}\Big\{ & \psi_n^0; \phi^1_n,\phi_{n\tau_1}^{0}, \phi_{n\tau_2}^{0}; \phi_{n\tau_1}^{1}, \phi_{n\tau_2}^{1},\phi_{n\tau_1}^{2}, \phi_{n\tau_2}^{2}, \phi_{n\tau_1 \tau_2}^{0},  \phi_{n\tau_1 \tau_1}^{0} + \frac{1}{2} \phi_{nnn}^{0}, \\
&\phi_n^{2} + \frac{c^{3,1}_1}{c^{2,1}_1} \phi_n^{3},\mu_1 \phi_{nnn}^{0} + \mu_2 \left( \phi_n^{3} - \frac{c^{3,1}_1}{c^{2,1}_1} \phi_n^{2} \right)\Big\},
\end{split}
\end{equation}
where
\begin{equation}\label{mu}
\mu_1 = 3 A_{5,11} \left[ 1 + \left( \frac{c_1^{3,1}}{c_1^{2,1}} \right)^2 \right], \qquad
\mu_2 = 2 \left( A_{58} - \frac{c_1^{3,1}}{c_1^{2,1}} A_{57} \right).
\end{equation}
Note that the basis functions listed in \eqref{eq:Vodd} are orthogonal basis functions, and the orthogonality will simplify our derivations of boundary conditions. The discretization the operator $\mathcal{S}$ is then a straightforward application of Galerkin's method with the subspace $\mathbb{V}_{\mathrm{odd}}$.

We would like to comment that $\mathbb{V}_{\mathrm{odd}} \oplus \mathbb{V}_{\mathrm{even}}$ is only a proper subspace of $\mathbb{V}$, or equivalently, $\mathbb{V}_{\mathrm{odd}} \neq \mathcal{P}_{\mathrm{odd}} \mathbb{V}$. The reason is that the operator $\mathcal{A}_n$ has a null space being a subspace of $\mathcal{P}_{\mathrm{odd}} \mathbb{V}$. This can be seen by focusing on the $x_n$ derivatives in \eqref{eqd2}\eqref{eqd3}\eqref{eqg}, and find that by the linear combination
\begin{displaymath}
A_{5,11} \left(c_1^{3,1} \times \eqref{eqd2} - c_1^{2,1} \times \eqref{eqd3} \right) - \frac{5}{3} (c_1^{3,1} A_{57} - c_1^{2,1} A_{58}) \times \eqref{eqg},
\end{displaymath}
the $x_n$ derivative will be cancelled out. This indicates
\begin{displaymath}
\mathcal{A}_n \left(
  3A_{5,11} \left(c_1^{3,1} \phi_n^2 - c_1^{2,1} \phi_n^3 \right) - \frac{175}{6} (c_1^{3,1} A_{57} - c_1^{2,1} A_{58}) \phi_{nnn}^0
\right) = 0,
\end{displaymath}
which reveals the null space of $\mathcal{A}_n$. The function space $\mathbb{V}_{\mathrm{odd}}$ defined in \eqref{eq:Vodd} is orthogonal to the null space of $\mathcal{A}_n$.

Practically, all the boundary conditions are written as
\begin{equation} \label{eq:bc_odd}
\langle \phi_{\mathrm{odd}}^i, f_{\mathrm{odd}} \rangle =
  \frac{2\chi}{2-\chi} \sum_j \frac{\langle \phi_{\mathrm{odd}}^i, \mathcal{C} \xi_n^{-1} \phi_{\mathrm{odd}}^j \rangle \langle \phi_{\mathrm{odd}}^j, \mathcal{A}_n(f_W - f_{\mathrm{even}}) \rangle}{\langle \phi_{\mathrm{odd}}^j, \phi_{\mathrm{odd}}^j \rangle}, \qquad j = 1,\cdots, 12.
\end{equation}
where $\phi_{\mathrm{odd}}^j$ refers to the $j$th basis function of $\mathbb{V}_{\mathrm{odd}}$ (see \eqref{eq:Vodd}), and
$f_{\mathrm{odd}}$ and $f_{\mathrm{even}}$ are, respectively, the projection of $f$ onto the function space $\mathbb{V}_{\mathrm{odd}}$ and $\mathbb{V}_{\mathrm{even}}$. Note that this formula requires the orthogonality of the basis $\{\phi_{\mathrm{odd}}^j\}$. Calculation of the integrals in \eqref{eq:bc_odd} requires specifying the collision model. Here we only provide the general form:
\begin{equation}
\label{eq:bc_raw}
\begin{pmatrix}
  v_n \\ \bar{q}_n \\ u_n^2 + \frac{c_1^{3,1}}{c_1^{2,1}} u_n^3 \\
  \bar{\sigma}_{\tau_1 n} \\ \bar{\sigma}_{\tau_2 n} \\
  u_{\tau_1 n}^1 \\ u_{\tau_2 n}^1 \\
  u_{\tau_1 n}^2 \\ u_{\tau_2 n}^2 \\
  u_{nnn}^0 + \frac{\mu_2}{\mu_1} \left( u_n^3 - \frac{c_1^{3,1}}{c_1^{2,1}} u_n^2 \right) \\
  u_{\tau_1 \tau_1 n}^0 + \frac{1}{2} u_{nnn}^0 \\ u_{\tau_1 \tau_2 n}^0
\end{pmatrix} = 
\frac{2\chi}{2-\chi} Z \begin{pmatrix}
  V_n \\ \bar{Q}_n \\ U_n^2 + \frac{c_1^{3,1}}{c_1^{2,1}} U_n^3 \\
  \bar{\Sigma}_{\tau_1 n} \\ \bar{\Sigma}_{\tau_2 n} \\
  U_{\tau_1 n}^1 \\ U_{\tau_2 n}^1 \\
  U_{\tau_1 n}^2 \\ U_{\tau_2 n}^2 \\
  U_{nnn}^0 + \frac{\mu_2}{\mu_1} \left( U_n^3 - \frac{c_1^{3,1}}{c_1^{2,1}} U_n^2 \right) \\
  U_{\tau_1 \tau_1 n}^0 + \frac{1}{2} U_{nnn}^0 \\ U_{\tau_1 \tau_2 n}^0
\end{pmatrix},
\end{equation}
where the capitalized variables are given by
\begin{equation} \label{eq:Aoe_moments}
 \begin{aligned}
    V_n = &\ (\rho^W - \rho) + (\theta^W - \theta) - c_2^{0,0} \bar{\sigma}_{nn} - \sqrt{15} (c_2^{1,0} u_{nn}^1 + c_2^{2,0} u_{nn}^2), \\
    \bar{Q}_{n} = &\ \frac{5}{2} c^{1,1}_1(\theta^W - \theta) + \frac{\sqrt{2}}{6}A_{45}\bar{\sigma}_{nn}+\sqrt{\frac{5}{6}} (A_{46}u^{2}+A_{49}u^{1}_{nn} + A_{4,10}u^{2}_{nn}),\\
    U^{2}_{n} = &\  \frac{\sqrt{15}}{3}\left(\frac{\sqrt{2}}{2} c^{2,1}_1 (\theta - \theta^W) - \frac{1}{15} A_{57}\bar{\sigma}_{nn} \right),\\
    U^{3}_{n} = &\  \frac{\sqrt{15}}{3}\left(\frac{\sqrt{2}}{2} c^{3,1}_1 (\theta - \theta^W) - \frac{1}{15} A_{58}\bar{\sigma}_{nn} \right),\\
    \bar{\Sigma}_{\tau_in} = &\ c^{0,0}_2 (v^W_{\tau_i}-v_{\tau_i}) + \frac{\sqrt{2}}{15} A_{45} \bar{q}_{\tau_i} - \frac{1}{\sqrt{15}} (A_{57}u^{2}_{\tau_i}  + A_{58}u^{3}_{\tau_i} + 2 A_{5,11} u^{0}_{\tau_inn}), \quad i=1,2, \\
    U^{1}_{\tau_in} = &\  \frac{\sqrt{15}}{15} \left( c^{1,0}_2 (v^W_{\tau_i}-v_{\tau_i})  + \frac{\sqrt{2}}{15} A_{49}\bar{q}_{\tau_i}  \right),\quad i=1,2,\\
    U^{2}_{\tau_in} = &\  \frac{\sqrt{15}}{15} \left( c^{2,0}_2 (v^W_{\tau_i}-v_{\tau_i})  + \frac{\sqrt{2}}{15} A_{4,10}\bar{q}_{\tau_i}  \right),\quad i=1,2,\\
    U^{0}_{\tau_i\tau_jn} = &\ \frac{2\sqrt{15}}{1575} A_{5,11}\left( \frac{2}{5} \delta_{ij} \bar{\sigma}_{nn}  -  \bar{\sigma}_{\tau_i\tau_j}\right), \quad i,j=1,2, \\
    U^{0}_{nnn} = &\ {-\frac{2\sqrt{15}}{875}}A_{5,11}\bar{\sigma}_{nn}.
    \end{aligned}
\end{equation}
and the structure of the matrix $Z$ is (``$*$'' means a nonzero entry)
\begin{displaymath}
Z = \left( \begin{array}{@{}cccccccccccc@{}}
  * & * & * & 0 & 0 & 0 & 0 & 0 & 0 & * & * & 0 \\
  * & * & * & 0 & 0 & 0 & 0 & 0 & 0 & * & * & 0 \\
  * & * & * & 0 & 0 & 0 & 0 & 0 & 0 & * & * & 0 \\
  0 & 0 & 0 & * & 0 & * & 0 & * & 0 & 0 & 0 & 0 \\
  0 & 0 & 0 & 0 & * & 0 & * & 0 & * & 0 & 0 & 0 \\
  0 & 0 & 0 & * & 0 & * & 0 & * & 0 & 0 & 0 & 0 \\
  0 & 0 & 0 & 0 & * & 0 & * & 0 & * & 0 & 0 & 0 \\
  0 & 0 & 0 & * & 0 & * & 0 & * & 0 & 0 & 0 & 0 \\
  0 & 0 & 0 & 0 & * & 0 & * & 0 & * & 0 & 0 & 0 \\
  * & * & * & 0 & 0 & 0 & 0 & 0 & 0 & * & * & 0 \\
  * & * & * & 0 & 0 & 0 & 0 & 0 & 0 & * & * & 0 \\
  0 & 0 & 0 & 0 & 0 & 0 & 0 & 0 & 0 & 0 & 0 & *
\end{array} \right) = \begin{pmatrix}
  z_0 & \zeta_0^{\intercal} \\ \zeta_1 & Z_1
\end{pmatrix}.
\end{displaymath}
The matrix $Z$ is the discrete form of the operator $\mathcal{S} = \mathcal{P}_{\mathrm{odd}} \mathcal{C} \xi_n^{-1}$, and is therefore similar to a symmetric and positive definite matrix. To complete the derivation, we still need to take into account the condition $v_n = \langle \xi_n, f \rangle = 0$. Since $V_n$ depends on $\rho_W$ (see \eqref{eq:rho_W}), which has to be determined by this condition. Instead of solving $\rho_W$, an equivalent way is to solve $V_n$ from the first equation of \eqref{eq:bc_raw}, and then plug the result into other equations. Such an equivalence is due to the fact that $\rho^W$ appears only in $V_n$ in all the quantities defined in \eqref{eq:Aoe_moments}. This leads to boundary conditions that have the same form as \eqref{eq:bc_raw}, but the matrix $Z$ is replaced by
\begin{displaymath}
  \begin{pmatrix} 0 & 0 \\ 0 & Z_1 - z_0^{-1} \zeta_1 \zeta_0^{\intercal} \end{pmatrix}.
\end{displaymath}
Since the Schur complement of a symmetric and positive definite matrix is still symmetric and positive definite, we can guarantee the the above matrix can still be viewed as a positive semidefinite operator.

Such boundary conditions satisfy all the properties required by the $L^2$ stability. However, spurious boundary layers are observed in the solution, which will be detailed in the next subsection.

\subsection{One-dimensional problems and spurious boundary layers} \label{sec:boundary_layer}
To show the boundary layers in the solution, we consider the steady-state one-dimensional flow in which
\begin{displaymath}
    v_2 = v_3  = \bar{q}_2 = \bar{q}_3 = \bar{\sigma}_{13} = \bar{\sigma}_{23}= 0, \qquad \bar{\sigma}_{22} = \bar{\sigma}_{33}.
\end{displaymath}
These conditions come from the assumption that the distribution function satisfies $f(\boldsymbol{x}, \xi_1, \xi_2, \xi_3, t) = \tilde{f}(\boldsymbol{x}, \xi_1, \sqrt{\xi_2^2 + \xi_3^2}, t)$ for every $\boldsymbol{x}$ and $t$, meaning that $f$ is axisymmetric about the axis $\xi_2 = \xi_3 = 0$. Under this assumption, only five variables remains and the equations have the form
\begin{align}
& \frac{\partial v}{\partial x} = 0, \label{eq:1deqrho} \\
& \frac{2}{3} \frac{\partial v}{\partial x}  +\frac{2}{3}k_0\frac{\partial \bar q}{\partial x}-k_1 \kn \frac{\partial^2 \theta}{\partial x^2}+k_2\kn\frac{\partial^2\bar{\sigma}}{\partial x^2} = 0, \label{eq:1deqv}\\ 
& \frac{\partial \rho}{\partial x} +\frac{\partial \theta}{\partial x} - \frac{2}{3} k_3 \kn \frac{\partial^2 v}{\partial x^2} - \frac{2}{3} k_4 \kn \frac{\partial^2\bar{q}}{\partial x^2}+ k_5 \frac{\partial \bar{\sigma}}{\partial x}  = 0, \label{eq:1deqtheta} \\ 
& \frac{5}{2} k_0\frac{\partial \theta}{\partial x} -\frac{5}{3} k_4\kn\frac{\partial^2v}{\partial x^2} -2 k_6 \kn\frac{\partial^2 \bar{q}}{\partial x^2}-\frac{8}{5}k_7\kn\frac{\partial^2 \bar{q}}{\partial x^2}+k_8\frac{\partial\bar{\sigma}}{\partial x}  = - \frac{2}{3} \frac{1}{\kn}l_1\bar{q}, \label{eq:1deqq} \\
& 2k_2\kn\frac{\partial^2\theta}{\partial x^2}+ \frac{4}{3} k_5 \frac{\partial v}{\partial x} +\frac{8}{15}k_8\frac{\partial \bar{q}}{\partial x} - \frac{6}{5} k_9\kn\frac{\partial^2 \bar{\sigma}}{\partial x^2} -\frac{2}{3} k_{10}\kn\frac{\partial^2 \bar{\sigma}}{\partial x^2}= - \frac{1}{\kn}l_2\bar{\sigma}. \label{eq:1deqsigma}
\end{align}
Here time derivatives have been removed so that the solution of the equations correspond to the steady state of the fluid.
In the one-dimensional case, the boundary conditions are reduced to
\begin{align}
    & v = 0, \label{eq:1dbcv} \\
    & \bar{q}  = s_n \tilde{\chi} \left[ \hat{m}_{11} (\theta -\theta^{W}) +\hat{m}_{12} \bar{\sigma} -  \hat{m}_{13} \kn \frac{\partial \bar{q}}{\partial x} -  \frac{2}{3} \hat{m}_{14} \kn \frac{\partial \bar{q}}{\partial x}  - \frac{2}{3} \hat{m}_{15} \kn \frac{\partial v}{\partial x} \right], \label{eq:1dbcq} \\ 
    & \hat{m}_{26}\bar{q} + \hat{m}_{27} \kn \frac{\partial \theta}{\partial x} - \hat{m}_{28} \kn \frac{\partial \bar{\sigma}}{\partial x}   \notag\\
     & \quad =s_n \tilde{\chi} \left[-\hat{m}_{21} (\theta -\theta^{W}) + \hat{m}_{22} \bar{\sigma} +  \hat{m}_{23} \kn \frac{\partial \bar{q}}{\partial x} + \frac{2}{3} \hat{m}_{24} \kn \frac{\partial \bar{q}}{\partial x}  + \frac{2}{3} \hat{m}_{25} \kn \frac{\partial v}{\partial x} \right], \label{eq:1dbctheta} \\ 
  & {-\hat{m}_{66}} \bar{q} - \kn\left(\frac{3}{5} \hat{m}_{67} \frac{\partial \bar{\sigma}}{\partial x} +\hat{m}_{68}\frac{\partial \bar{\sigma}}{\partial x} + \hat{m}_{69}\frac{\partial \theta}{\partial x} \right) \notag\\
  & \quad = s_n \tilde{\chi} \left[-\hat{m}_{61} (\theta -\theta^{W}) +\hat{m}_{62} \bar{\sigma} + \hat{m}_{63} \kn \frac{\partial \bar{q}}{\partial x} + \frac{2}{3} \hat{m}_{64} \kn \frac{\partial \bar{q}}{\partial x} + \frac{2}{3} \hat{m}_{65} \kn \frac{\partial v}{\partial x}\right], \label{eq:1dbcsigma}
\end{align}
where $s_n=1$ at the right boundary and $s_n=-1$ at the left one. The coefficient $\hat{m}_{ij}$ is computed from the Onsager boundary conditions \eqref{eq:bc_raw}, which is not equal to the $m_{ij}$ in our final boundary conditions as shown in \eqref{eq:bcv}--\eqref{eq:bcm}.

The one-dimensional equations \eqref{eq:1deqrho}--\eqref{eq:1dbcsigma} can be solved analytically. First, it is straightforward to get $v(x)=0$ from \eqref{eq:1deqrho} and \eqref{eq:1dbcv}. Second, by noticing variable $\rho$ only appears in \eqref{eq:1deqtheta}, we can solve other equations and then plug other variables in \eqref{eq:1deqtheta}. Third, \eqref{eq:1deqv} 
reveals that $\frac{\partial \bar{q}}{\partial x} = \frac{3}{2 k_0} \kn ( k_1 \frac{\partial^2 \theta}{\partial x^2} - k_2 \frac{\partial^2 \bar{\sigma}}{\partial x^2} )$, which indicates 
\begin{equation}\label{eq:1dsolq}
\bar{q}(x) = \frac{3}{2 k_0} \kn \left( k_1 \frac{\partial \theta}{\partial x} - k_2 \frac{\partial \bar{\sigma}}{\partial x} \right) + C_{\bar{q}}    
\end{equation}
with the constant $C_{\bar{q}}$ to be determined. Plugging $\bar{q}(x)$ into \eqref{eq:1deqq} and integrating the equation, we find that
\begin{equation}\label{eq:2eqq}
    \kn^2 \left( -k_6-\frac{4}{5} k_7 \right) \frac{3}{k_0} \left( k_1 \frac{\partial^2 \theta}{\partial x^2} - k_2 \frac{\partial^2 \bar{\sigma}}{\partial x^2} \right) + \frac{5}{2} k_0 \theta + k_8 \bar{\sigma} = - \frac{l_1}{k_0} \left( k_1 \theta - k_2 \bar{\sigma}  \right) + C_{\bar{q}}^1 x + C_{\bar{q}}^2, 
\end{equation}
where $C_{\bar{q}}^1 = -\frac{2 l_1}{3 \kn} C_{\bar{q}}$ and $C_{\bar{q}}^2$ are two constants to be determined. Similarly, plugging \eqref{eq:1dsolq} into \eqref{eq:1deqsigma} yields
\begin{equation}\label{eq:2eqsigma}
2 k_2\kn\frac{\partial^2\theta}{\partial x^2} + \frac{4 k_8 }{5 k_0} \kn \left( k_1 \frac{\partial^2 \theta}{\partial x^2} - k_2 \frac{\partial^2 \bar{\sigma}}{\partial x^2} \right) - \left( \frac{6}{5} k_9 + \frac{2}{3} k_{10} \right) \kn\frac{\partial^2 \bar{\sigma}}{\partial x^2}= - \frac{1}{\kn}l_2\bar{\sigma}.
\end{equation}
We can then define
\begin{displaymath}
    \xi = \kn \frac{\partial \theta}{\partial x}, \qquad \xi' = \kn \frac{\partial \bar{\sigma}}{\partial x},
\end{displaymath}
and reformulate \eqref{eq:2eqq} and \eqref{eq:2eqsigma} as
\begin{align*}
    \kn
    \begin{pmatrix}
        0 & 0 & 1 & 0 \\
        0 & 0 & 0 & 1 \\
        - \left( k_6 + \frac{4}{5} k_7 \right) \frac{3 k_1}{k_0} & \left( k_6 + \frac{4}{5} k_7 \right) \frac{3 k_2}{k_0} & 0 & 0 \\
        2 k_2 + \frac{4 k_8 k_1}{5 k_0} & - \left( \frac{6}{5} k_9 + \frac{2}{3} k_{10} + \frac{4 k_8 k_2}{5 k_0} \right) & 0 & 0 \\
    \end{pmatrix}
    \frac{\partial}{\partial x}
    \begin{pmatrix}
        \xi \\
        \xi' \\
        \theta \\
        \bar{\sigma}
    \end{pmatrix} \\
    +
    \begin{pmatrix}
        -1 & 0 & 0 & 0 \\    
        0 & -1 & 0 & 0 \\
        0 & 0 & \frac{5}{2}k_0 + \frac{l_1 k_1}{k_0} & k_8 - \frac{l_1 k_2}{k_0} \\    
        0 & 0 & 0 & l_2 \\    
    \end{pmatrix}
    \begin{pmatrix}
        \xi \\
        \xi' \\
        \theta \\
        \bar{\sigma}
    \end{pmatrix} 
    = 
    \begin{pmatrix}
        0 \\
        0 \\
        C_{\bar{q}}^1 x + C_{\bar{q}}^2 \\
        0
    \end{pmatrix}.
\end{align*}
By multiplying the inverse of the matrix in front of the first-order derivatives, the above equation has the following structure
\begin{equation}\label{eq:1orderode}
    \kn \frac{\partial}{\partial x}
    \begin{pmatrix}
        \xi \\
        \xi' \\
        \theta \\
        \bar{\sigma}
    \end{pmatrix}
    +
    \begin{pmatrix}
        0 & 0 & b_{11} & b_{12} \\    
        0 & 0 & b_{21} & b_{22} \\    
        -1 & 0 & 0 & 0 \\    
        0 & -1 & 0 & 0     
    \end{pmatrix}
    \begin{pmatrix}
        \xi \\
        \xi' \\
        \theta \\
        \bar{\sigma}
    \end{pmatrix} 
    = 
    \begin{pmatrix}
        c_1 ( C_{\bar{q}}^1 x + C_{\bar{q}}^2 ) \\
        c_2 ( C_{\bar{q}}^1 x + C_{\bar{q}}^2 ) \\
        0 \\
        0
    \end{pmatrix}.
\end{equation}
In general, given the following ODE system
\begin{displaymath}
    \kn \frac{\partial g}{\partial x} + A g = r(x),
\end{displaymath}
its solution is 
\begin{displaymath}
    g(x) = e^{-A x / \kn} g(0) + \frac{1}{\kn} \int_0^x e^{A(s-x)/\kn} r(s) \mathrm{d}s.
\end{displaymath}
In our case, the matrix in \eqref{eq:1orderode} has four distinct eigenvalues $\pm \lambda_1$ and $\pm \lambda_2$ with
\begin{equation}\label{eq:nonmaxeig}
    \lambda_j = \frac{\sqrt{(-1)^j\sqrt{b_{11}^2-2 b_{22} b_{11}+b_{22}^2+4 b_{12} b_{21}}-b_{11}-b_{22}}}{\sqrt{2}}, \qquad j=1,2.
\end{equation}
The solution of \eqref{eq:1orderode} therefore has the form
\begin{equation}\label{eq:1dsolform}
    \kappa_1 \sinh \left( \frac{\lambda_1 x}{\kn} \right) + \kappa_2 \cosh \left( \frac{\lambda_1 x}{\kn} \right) + \kappa_3 \sinh \left( \frac{\lambda_2 x}{\kn} \right) + \kappa_4 \cosh \left( \frac{\lambda_2 x}{\kn} \right) + \kappa_5 x + \kappa_6,
\end{equation}
where the hyperbolic sines and cosines denote the Knudsen layers. Since $\lambda_1 \neq \lambda_2$, two boundary layers will be present in the general solution. In particular, when the magnitude of eigenvalues is large or the Knudsen number is small, the solution would exhibit boundary layers.

For Maxwell molecules ($\eta = 5$), the above procedure of solving one-dimensional equations can be further simplified from the third step \eqref{eq:1dsolq}. By noticing $k_1 = k_2 = 0$, one can get $\frac{\partial \bar{q}}{\partial x} = 0$ and therefore \eqref{eq:1dsolq} becomes $\bar{q}(x) = C_{\bar{q}}$. The equation \eqref{eq:2eqsigma} becomes a second-order ODE for $\bar{\sigma}$, so that the general solution of $\bar{\sigma}(x)$ contains only one boundary layer. Furthermore, \eqref{eq:2eqq} is in absence of second-order derivatives, and thus
\begin{displaymath}
\theta(x) = \frac{2}{5k_0} \left(C_{\bar{q}}^1 x + C_{\bar{q}}^2 - k_8 \bar{\sigma}(x) \right),
\end{displaymath}
showing that the boundary layer of $\theta$ has the same thickness as that of $\bar{\sigma}$. In general, when the type of molecules gets closer to Maxwell molecules, the second boundary layer will get thinner, which explains why there is only one boundary layer left for Maxwell molecules.

The thinness of the second boundary layer also indicates the largeness of the corresponding eigenvalue $\lambda_2$. To illustrate the magnitude of eigenvalues, we compute them according to \eqref{eq:nonmaxeig} for various inverse-power-law models and summarize the results in table \ref{tab:1deig}. For Maxwell molecules, only $\lambda_1$ is provided. The table clearly shows that all models share one similar eigenvalue $\lambda_1$ and each of the non-Maxwell cases has an additional larger eigenvalue $\lambda_2$. This eigenvalue $\lambda_2$ approaches infinity when $\eta$ tends to $5$.

The one-dimensional solutions on the domain $(-0.5, 0.5)$, with Knudsen number $\kn=0.2$ and accommodation coefficients $\chi=1$, are depicted in figure \ref{fig:layers}. The figure verifies that the larger eigenvalue $\lambda_2$ in non-Maxwell cases induces additional boundary layers compared to the Maxwell solutions. These boundary layers are non-physical according to previous results for the same problem as seen in \cite{Hu2020, yang2024siap}.
In the next subsection, we will propose a revision to the Onsager boundary conditions derived in section \ref{subsec:Onsager} to remove the boundary layers, leading to our final boundary conditions \eqref{eq:bcv}--\eqref{eq:bcm}.


\begin{table}
\centering \footnotesize
\begin{displaymath}
\begin{array}{| c | c | c |}
\hline
\hline
\eta &  \lambda_{1}   &  \lambda_{2} \\
5 & 9.1287 \times 10^{-1} & - \\
7 & 9.2648\times 10^{-1} & 1.2696\times 10^{1} \\
10 & 9.2989\times 10^{-1} & 7.6747 \\ 
17 & 9.2904\times 10^{-1} & 5.6909 \\
\infty & 9.2248\times 10^{-1} & 4.2387 \\
\hline
\hline
\end{array}
\end{displaymath}
\caption{Eigenvalues in the general solution \eqref{eq:1dsolform} for some inverse-power-law models with power index $\eta$.\label{tab:1deig}} 
\end{table}

\begin{figure}
    \centering
    \includegraphics[width=0.4\textwidth]{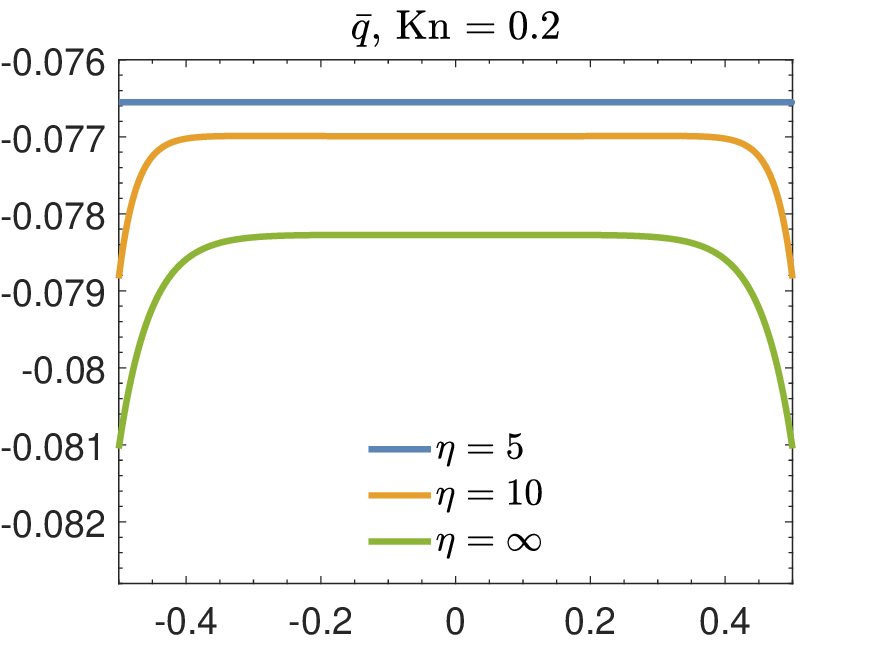}
    \hspace{0.05\textwidth}
    \includegraphics[width=0.4\textwidth]{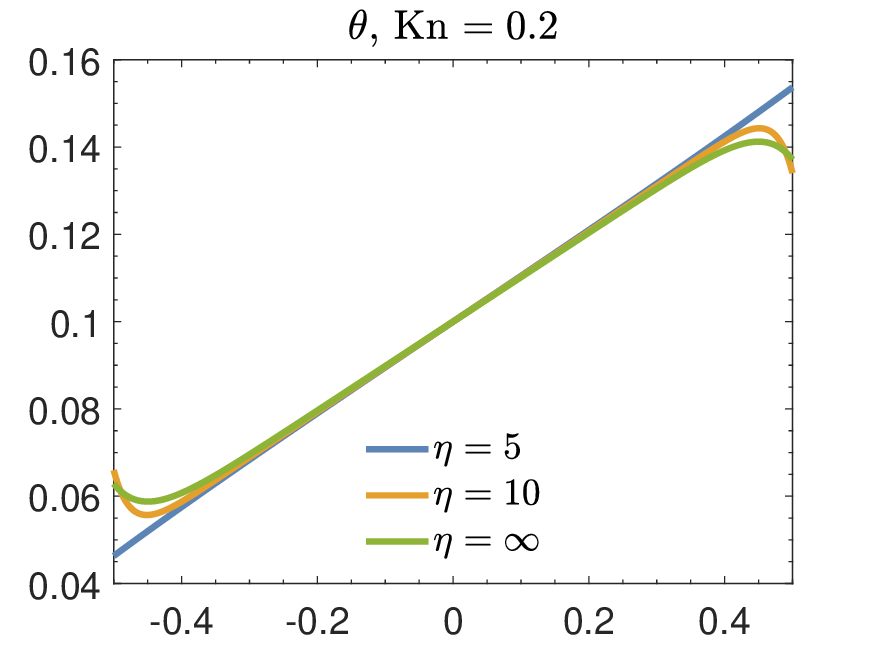}
    \caption{$\bar{q}$ and $\theta$ of the one-dimensional problem \eqref{eq:1deqrho}--\eqref{eq:1dbcsigma}. $\theta^W=0$ in the left boundary and $\theta^{W}=0.2$ in the right one.}
    \label{fig:layers}
\end{figure}

\subsection{Removal of undesired boundary layers}\label{sec:removelayer}
Our discussion will be based on the following general form of one-dimensional symmetric hyperbolic equations:
\begin{equation}\label{eq:generic}
\begin{pmatrix}
  0 & \mathbf{A}_{\mathrm{oe}} \\
  \mathbf{A}_{\mathrm{eo}} & 0
\end{pmatrix} \frac{\partial}{\partial x} \begin{pmatrix}
  \boldsymbol{u}_{\mathrm{o}} \\ \boldsymbol{u}_{\mathrm{e}}
\end{pmatrix} = \begin{pmatrix}
  \mathbf{L}_{\mathrm{oo}} & 0 \\
  0 & \mathbf{L}_{\mathrm{ee}}
\end{pmatrix} \begin{pmatrix}
  \boldsymbol{u}_{\mathrm{o}} \\ \boldsymbol{u}_{\mathrm{e}}
\end{pmatrix},
\end{equation}
where $\mathbf{A}_{\mathrm{eo}} = \mathbf{A}_{\mathrm{oe}}^{\intercal}$,  and $\mathbf{L}_{\mathrm{oo}}$ and $\mathbf{L}_{\mathrm{ee}}$ are symmetric and negative semidefinite matrices. As noted in section \ref{subsec:Onsager}, the operator $\mathcal{A}_n$ has a non-trivial null space being a subspace of $\mathcal{P}_{\mathrm{odd}} \mathbb{V}$. Accordingly, in order to mimic this property, we assume that the matrix $\mathbf{A}_{\mathrm{oe}}$ is rank-deficient,
and its singular value decomposition is
\begin{displaymath}
\mathbf{A}_{\mathrm{oe}} = 
\begin{pmatrix}
  \mathbf{U}_1 & \mathbf{U}_2
\end{pmatrix}
\begin{pmatrix}
  \boldsymbol{\Lambda}_r & 0 \\
  0 & 0
\end{pmatrix}
\begin{pmatrix}
  \mathbf{V}_1^{\intercal} \\ \mathbf{V}_2^{\intercal}
\end{pmatrix},
\end{displaymath}
where $r$ denotes the rank of $\mathbf{A}_{\mathrm{oe}}$ and
\begin{displaymath}
\boldsymbol{\Lambda}_r = \begin{pmatrix}
        \lambda_1 & 0 & \cdots & 0 \\
        0 & \lambda_2 & \cdots & 0 \\
        \vdots & \vdots & \ddots & \vdots \\
        0 & 0 & \cdots & \lambda_r
    \end{pmatrix},
    \qquad \lambda_1 \geqslant \lambda_2 \geqslant ... \geqslant \lambda_r > 0.
\end{displaymath}
Inspired by the preliminary section in \cite{Jiang2013}, we define
\begin{displaymath}
\boldsymbol{v}_{\mathrm{o}} = \mathbf{U}_1^{\intercal} \boldsymbol{u}_{\mathrm{o}}, \quad
\boldsymbol{v}_{\mathrm{e}} = \mathbf{V}_1^{\intercal} \boldsymbol{u}_{\mathrm{e}}.
\end{displaymath}
By straightforward calculation, it can be shown that $\boldsymbol{v}_{\mathrm{o}}$
and $\boldsymbol{v}_{\mathrm{e}}$ satisfy the following differential equations:
\begin{equation}\label{eq:vrhsmat}
\frac{\partial}{\partial x}
\begin{pmatrix}
    \boldsymbol{v}_{\mathrm{o}} \\
    \boldsymbol{v}_{\mathrm{e}}
\end{pmatrix} =
\begin{pmatrix}
  0 & \boldsymbol{\Lambda}_r^{-1} \tilde{\mathbf{L}}_{\mathrm{ee}} \\
  \boldsymbol{\Lambda}_r^{-1} \tilde{\mathbf{L}}_{\mathrm{oo}} &  0
\end{pmatrix}
\begin{pmatrix}
    \boldsymbol{v}_{\mathrm{o}} \\
    \boldsymbol{v}_{\mathrm{e}}
\end{pmatrix},
\end{equation}
where
\begin{align*}
  \tilde{\mathbf{L}}_{\mathrm{oo}} &= \mathbf{U}_1^{\intercal} \mathbf{L}_{\mathrm{oo}}\mathbf{U}_1 - \mathbf{U}_1^{\intercal} \mathbf{L}_{\mathrm{oo}}\mathbf{U}_2 (\mathbf{U}_2^{\intercal} \mathbf{L}_{\mathrm{oo}}\mathbf{U}_2)^{-1}\mathbf{U}_2^{\intercal} \mathbf{L}_{\mathrm{oo}}\mathbf{U}_1, \\
  \tilde{\mathbf{L}}_{\mathrm{ee}} &= \mathbf{V}_1^{\intercal} \mathbf{L}_{\mathrm{ee}}\mathbf{V}_1 - \mathbf{V}_1^{\intercal} \mathbf{L}_{\mathrm{ee}}\mathbf{V}_2 (\mathbf{V}_2^{\intercal} \mathbf{L}_{\mathrm{ee}}\mathbf{V}_2)^{-1}\mathbf{V}_2^{\intercal} \mathbf{L}_{\mathrm{ee}}\mathbf{V}_1.
\end{align*}
It is obvious that $\tilde{\mathbf{L}}_{\mathrm{oo}}$ and $\tilde{\mathbf{L}}_{\mathrm{ee}}$ are Schur complements of negative semidefinite matrices, and therefore are also negative semidefinite matrices. Note that here we require the invertibility of $\mathbf{U}_2^{\intercal} \mathbf{L}_{\mathrm{oo}}\mathbf{U}_2$ and $\mathbf{V}_2^{\intercal} \mathbf{L}_{\mathrm{ee}}\mathbf{V}_2$, which is satisfied in the R13 equations derived in previous sections.

The following lemma reveals the eigen-pairs of the matrix in \eqref{eq:vrhsmat}.
\begin{lemma}\label{lemma:eig}
    If $\{ \lambda, (\boldsymbol{r}_{\mathrm{o}}^{\intercal},\boldsymbol{r}_{\mathrm{e}}^{\intercal})^{\intercal} \}$ is an eigen-pair of the matrix in \eqref{eq:vrhsmat}, then $\{ -\lambda, (-\boldsymbol{r}_{\mathrm{o}}^{\intercal},\boldsymbol{r}_{\mathrm{e}}^{\intercal})^{\intercal} \}$ is also an eigen-pair of it. 
\end{lemma}

\begin{proof}
    From
    \begin{equation*}
        \begin{pmatrix}
        0 & \boldsymbol{\Lambda}_r^{-1} \tilde{\mathbf{L}}_{\mathrm{ee}} \\
        \boldsymbol{\Lambda}_r^{-1}\tilde{\mathbf{L}}_{\mathrm{oo}} & 0
        \end{pmatrix} 
        \begin{pmatrix}
            \boldsymbol{r}_{\mathrm{o}} \\
            \boldsymbol{r}_{\mathrm{e}}
        \end{pmatrix}
        = \lambda
        \begin{pmatrix}
            \boldsymbol{r}_{\mathrm{o}} \\
            \boldsymbol{r}_{\mathrm{e}}
        \end{pmatrix},
    \end{equation*}
    one sees that $\boldsymbol{\Lambda}_r^{-1} \tilde{\mathbf{L}}_{\mathrm{ee}} \boldsymbol{r}_{\mathrm{e}} = \lambda \boldsymbol{r}_{\mathrm{o}}$ and $\boldsymbol{\Lambda}_r^{-1} \tilde{\mathbf{L}}_{\mathrm{oo}} \boldsymbol{r}_{\mathrm{o}} = \lambda \boldsymbol{r}_{\mathrm{e}}$. Then
    \begin{equation*}
        \begin{pmatrix}
        0 & \boldsymbol{\Lambda}_1^{-1} \tilde{\mathbf{L}}_{\mathrm{ee}} \\
        \boldsymbol{\Lambda}_1^{-1}\tilde{\mathbf{L}}_{\mathrm{oo}} & 0
        \end{pmatrix} 
        \begin{pmatrix}
            -\boldsymbol{r}_{\mathrm{o}} \\
            \boldsymbol{r}_{\mathrm{e}}
        \end{pmatrix}
        =
        \begin{pmatrix}
            \boldsymbol{\Lambda}_r^{-1} \tilde{\mathbf{L}}_{\mathrm{ee}} \boldsymbol{r}_{\mathrm{e}} \\
            -\boldsymbol{\Lambda}_r^{-1} \tilde{\mathbf{L}}_{\mathrm{oo}} \boldsymbol{r}_{\mathrm{o}}
        \end{pmatrix} 
        = 
        \begin{pmatrix}
            \lambda \boldsymbol{r}_{\mathrm{o}} \\
            -\lambda \boldsymbol{r}_{\mathrm{e}}
        \end{pmatrix} 
        = - \lambda
        \begin{pmatrix}
            -\boldsymbol{r}_{\mathrm{o}} \\
            \boldsymbol{r}_{\mathrm{e}}
        \end{pmatrix},
    \end{equation*}
    which completes the proof.
\end{proof}
By the above lemma, the eigenvalue decomposition of the matrix in \eqref{eq:vrhsmat} is
\begin{equation}\label{eq:Meig}
    \begin{pmatrix}
        0 & \boldsymbol{\Lambda}_r^{-1} \tilde{\mathbf{L}}_{\mathrm{ee}} \\
        \boldsymbol{\Lambda}_r^{-1} \tilde{\mathbf{L}}_{\mathrm{oo}} & 0
    \end{pmatrix} = \mathbf{R}_r
    \begin{pmatrix}
        \tilde{\lambda}_1 \\
        & -\tilde{\lambda}_1 \\
        & & \ddots \\
        & & & \tilde{\lambda}_r \\
        & & & & -\tilde{\lambda}_r
    \end{pmatrix}
    \mathbf{R}_r^{-1},
\end{equation}
where $\tilde{\lambda}_1 \geqslant \tilde{\lambda}_2 \geqslant ... \geqslant \tilde{\lambda}_s > 0 = \tilde{\lambda}_{s+1} = \cdots = \tilde{\lambda}_r$ for some $s \leqslant r$, and
\begin{displaymath}
\mathbf{R}_r = \begin{pmatrix}
    \mathbf{R}_s & \mathbf{R}_0
\end{pmatrix}
\end{displaymath}
with
\begin{equation}
    \mathbf{R}_s =
    \begin{pmatrix}
        \boldsymbol{r}_{\mathrm{o}1} & - \boldsymbol{r}_{\mathrm{o}1} & \ldots & \boldsymbol{r}_{\mathrm{o}s} & -\boldsymbol{r}_{\mathrm{o}s} \\
        \boldsymbol{r}_{\mathrm{e}1} & \boldsymbol{r}_{\mathrm{e}1} & \ldots & \boldsymbol{r}_{\mathrm{e}s} & \boldsymbol{r}_{\mathrm{e}s} 
    \end{pmatrix}
\end{equation}
For simplicity, we define $\mathbf{L}_r^{\intercal} = \mathbf{R}_r^{-1}$ so that $\mathbf{L}_r$ has the form $\begin{pmatrix} \mathbf{L}_s & \mathbf{L}_0 \end{pmatrix}$ with
\begin{equation}
    \mathbf{L}_s =
    \begin{pmatrix}
        \boldsymbol{l}_{\mathrm{o}1} & - \boldsymbol{l}_{\mathrm{o}1} & \ldots & \boldsymbol{l}_{\mathrm{o}s} & -\boldsymbol{l}_{\mathrm{o}s} \\
        \boldsymbol{l}_{\mathrm{e}1} & \boldsymbol{l}_{\mathrm{e}1} & \ldots & \boldsymbol{l}_{\mathrm{e}s} & \boldsymbol{l}_{\mathrm{e}s} 
    \end{pmatrix}
\end{equation}
If we assume the simulation domain is $(0,1)$ without loss of generality, the solution of such a linear system has the form
\begin{equation}\label{eq:r37sol}
\begin{split}
    \boldsymbol{v}(x) = & \begin{pmatrix}
        \boldsymbol{v}_{\mathrm{o}}(x) \\
        \boldsymbol{v}_{\mathrm{e}}(x)
    \end{pmatrix}
    = \mathbf{R}_r 
    \begin{pmatrix}
        e^{\tilde{\lambda}_1 x} \\
        & e^{-\tilde{\lambda}_1 x} \\
        & & \ddots \\
        & & & e^{\tilde{\lambda}_r x} \\ 
        & & & & e^{-\tilde{\lambda}_r x}
    \end{pmatrix}
    \mathbf{L}_r^{\intercal} \boldsymbol{v}(0) \\
    & = 
    \sum_{j=1}^{s} \left( e^{\tilde{\lambda}_j x}
    \begin{pmatrix}
        \boldsymbol{r}_{\mathrm{o}j} \\
        \boldsymbol{r}_{\mathrm{e}j} 
    \end{pmatrix} (\boldsymbol{l}_{\mathrm{o}j}^{\intercal} , \boldsymbol{l}_{\mathrm{e}j}^{\intercal} )
    + 
    e^{-\tilde{\lambda}_j x}
    \begin{pmatrix}
        -\boldsymbol{r}_{\mathrm{o}j} \\
        \boldsymbol{r}_{\mathrm{e}j} 
    \end{pmatrix} (-\boldsymbol{l}_{\mathrm{o}j}^{\intercal} , \boldsymbol{l}_{\mathrm{e}j}^{\intercal} )
    \right) \boldsymbol{v}(0) + \mathbf{R}_0 \mathbf{L}_0^{\intercal} \boldsymbol{v}(0) \\
    & = 
    \sum_{j=1}^{s} \left(
    2\cosh(\tilde{\lambda}_j x)
    \begin{pmatrix}
         \boldsymbol{r}_{\mathrm{o}j} \boldsymbol{l}_{\mathrm{o}j}^{\intercal} \boldsymbol{v}_{\mathrm{o}}(0) \\
         \boldsymbol{r}_{\mathrm{e}j} \boldsymbol{l}_{\mathrm{e}j}^{\intercal} \boldsymbol{v}_{\mathrm{e}}(0) 
    \end{pmatrix} +
    2\sinh(\tilde{\lambda}_j x)
    \begin{pmatrix}
         \boldsymbol{r}_{\mathrm{o}j} \boldsymbol{l}_{\mathrm{e}j}^{\intercal} \boldsymbol{v}_{\mathrm{e}}(0) \\
         \boldsymbol{r}_{\mathrm{e}j} \boldsymbol{l}_{\mathrm{o}j}^{\intercal} \boldsymbol{v}_{\mathrm{o}}(0) 
    \end{pmatrix}
    \right) + \mathbf{R}_0 \mathbf{L}_0^{\intercal} \boldsymbol{v}(0).
\end{split}
\end{equation}
The rank $r$ for Maxwell molecules is 3 less than the rank of the non-Maxwell cases. In fact, for non-Maxwell gases, the values of $\lambda_r$, $\lambda_{r-1}$ and $\lambda_{r-2}$ are very close to $0$, and thus $\tilde{\lambda}_{1,2,3} \gg 1$, which exhibits additional boundary layers. Note that when the system is reduced to the one-dimensional case as in section \ref{sec:boundary_layer}, only one large eigenvalue is left.

To remove the boundary layers in the non-Maxwell cases, we will revise the boundary condition to eliminate the corresponding terms in \eqref{eq:r37sol}.
For simplicity, we only introduce the removal of the boundary layer associated with $\tilde{\lambda}_1$. The other two layers are removed similarly. The idea is to enforce the following equalities in our boundary conditions:
\begin{equation}\label{eq:removelayer}
\boldsymbol{l}_{\mathrm{o}1}^{\intercal} \boldsymbol{v}_{\mathrm{o}}(0) = \boldsymbol{l}_{\mathrm{o}1}^{\intercal} \boldsymbol{v}_{\mathrm{o}}(1) =0,    
\end{equation}
which is sufficient to get rid of the terms involving $\cosh(\tilde{\lambda}_1 x)$ and $\sinh(\tilde{\lambda}_1 x)$ in \eqref{eq:r37sol}. The reason will be explained as follows.

We can follow \eqref{eq:r37sol} to write the solution using boundary conditions at $x=1$ as
\begin{equation}\label{eq:r37solat1}
\begin{split}
    &\boldsymbol{v}(x) = \\ 
    & \quad \sum_{j=1}^{r} \left(
    2\cosh(\tilde{\lambda}_j (x-1))
    \begin{pmatrix}
         \boldsymbol{r}_{\mathrm{o}j} \boldsymbol{l}_{\mathrm{o}j}^{\intercal} \boldsymbol{v}_{\mathrm{o}}(1) \\
         \boldsymbol{r}_{\mathrm{e}j} \boldsymbol{l}_{\mathrm{e}j}^{\intercal} \boldsymbol{v}_{\mathrm{e}}(1)
    \end{pmatrix} +
    2\sinh(\tilde{\lambda}_j (x-1))
    \begin{pmatrix}
         \boldsymbol{r}_{\mathrm{o}j} \boldsymbol{l}_{\mathrm{e}j}^{\intercal} \boldsymbol{v}_{\mathrm{e}}(1) \\
         \boldsymbol{r}_{\mathrm{e}j} \boldsymbol{l}_{\mathrm{o}j}^{\intercal} \boldsymbol{v}_{\mathrm{o}}(1) 
    \end{pmatrix}
    \right)+ \mathbf{R}_0 \mathbf{L}_0^{\intercal} \boldsymbol{v}(1).
\end{split}
\end{equation}
Plugging the following equalities into \eqref{eq:r37solat1}:
\begin{equation*}
\begin{split}
    \cosh( \tilde{\lambda}_1(x-1)) = \cosh(\tilde{\lambda}_1 x) \cosh(-\tilde{\lambda}_1) + \sinh(\tilde{\lambda}_1 x) \sinh(-\tilde{\lambda}_1), \\
    \sinh( \tilde{\lambda}_1(x-1)) = \sinh(\tilde{\lambda}_1 x) \cosh(-\tilde{\lambda}_1) + \cosh(\tilde{\lambda}_1 x) \sinh(-\tilde{\lambda}_1),
\end{split}
\end{equation*}
and matching the coefficients of $\cosh(\tilde{\lambda}_1 x)$ and $\sinh(\tilde{\lambda}_1 x)$ in \eqref{eq:r37sol}, one can show that
\begin{equation}
\begin{split}
    \begin{pmatrix}
        \boldsymbol{r}_{\mathrm{o}1} ( \boldsymbol{l}_{\mathrm{o}1}^{\intercal} \boldsymbol{v}_{\mathrm{o}}(1) \cosh(-\tilde{\lambda}_1) + \boldsymbol{l}_{\mathrm{e}1}^{\intercal} \boldsymbol{v}_{\mathrm{e}}(1) \sinh(-\tilde{\lambda}_1) ) \\
        \boldsymbol{r}_{\mathrm{e}1} ( \boldsymbol{l}_{\mathrm{e}1}^{\intercal} \boldsymbol{v}_{\mathrm{e}}(1) \cosh(-\tilde{\lambda}_1) + \boldsymbol{l}_{\mathrm{o}1}^{\intercal} \boldsymbol{v}_{\mathrm{o}}(1) \sinh(-\tilde{\lambda}_1) )       
    \end{pmatrix} = 
    \begin{pmatrix}
         \boldsymbol{r}_{\mathrm{o}1} \boldsymbol{l}_{\mathrm{o}1}^{\intercal} \boldsymbol{v}_{\mathrm{o}}(0) \\
         \boldsymbol{r}_{\mathrm{e}1} \boldsymbol{l}_{\mathrm{e}1}^{\intercal} \boldsymbol{v}_{\mathrm{e}}(0) 
    \end{pmatrix}, \\
    \begin{pmatrix}
        \boldsymbol{r}_{\mathrm{o}1} ( \boldsymbol{l}_{\mathrm{o}1}^{\intercal} \boldsymbol{v}_{\mathrm{o}}(1) \sinh(-\tilde{\lambda}_1) + \boldsymbol{l}_{\mathrm{e}1}^{\intercal} \boldsymbol{v}_{\mathrm{e}}(1) \cosh(-\tilde{\lambda}_1) ) \\
        \boldsymbol{r}_{\mathrm{e}1} ( \boldsymbol{l}_{\mathrm{e}1}^{\intercal} \boldsymbol{v}_{\mathrm{e}}(1) \sinh(-\tilde{\lambda}_1) + \boldsymbol{l}_{\mathrm{o}1}^{\intercal} \boldsymbol{v}_{\mathrm{o}}(1) \cosh(-\tilde{\lambda}_1) )       
    \end{pmatrix} = 
    \begin{pmatrix}
         \boldsymbol{r}_{\mathrm{o}1} \boldsymbol{l}_{\mathrm{e}1}^{\intercal} \boldsymbol{v}_{\mathrm{e}}(0) \\
         \boldsymbol{r}_{\mathrm{e}1} \boldsymbol{l}_{\mathrm{o}1}^{\intercal} \boldsymbol{v}_{\mathrm{o}}(0) 
    \end{pmatrix} .
\end{split}
\end{equation}
We can now apply $\boldsymbol{l}_{\mathrm{o}1}^{\intercal} \boldsymbol{v}_{\mathrm{o}}(0) = \boldsymbol{l}_{\mathrm{o}1}^{\intercal} \boldsymbol{v}_{\mathrm{o}}(1) =0$ to the above equation, and it is not difficult to derive $\boldsymbol{l}_{\mathrm{e}1}^{\intercal} \boldsymbol{v}_{\mathrm{e}}(0) = \boldsymbol{l}_{\mathrm{e}1}^{\intercal} \boldsymbol{v}_{\mathrm{e}}(1) = 0$. Therefore, all the terms involving $\cosh(\tilde{\lambda}_1 x)$ and $\sinh(\tilde{\lambda}_1 x)$ in \eqref{eq:r37sol} or \eqref{eq:r37solat1} disappear so that the corresponding boundary layers are not present in the solution.

The above derivation shows that \eqref{eq:removelayer} is a sufficient condition to suppress boundary layers. 
Moreover, it is straightforward to generalize such a condition to multidimensional cases by imposing $\boldsymbol{l}_{\mathrm{o}1}^{\intercal} \boldsymbol{v}_{\mathrm{o}}=0$ for $\boldsymbol{v}_{\mathrm{o}}$ at all boundary points. Recall that the Onsager boundary conditions take a form of \eqref{eq:r13bc}, which can be reformulated as
\begin{displaymath}
    \boldsymbol{v}_{\mathrm{o}} = \tilde{\mathbf{Q}} \boldsymbol{\Lambda}(\tilde{\boldsymbol{g}}_{\mathrm{ext}} - \boldsymbol{v}_{\mathrm{e}}).
\end{displaymath}
To impose $\boldsymbol{l}_{\mathrm{o}1}^{\intercal} \boldsymbol{v}_{\mathrm{o}}=0$, we choose to modify the symmetric matrix $\tilde{\mathbf{Q}}$ such that
\begin{equation}\label{eq:stilde}
    \boldsymbol{l}_{\mathrm{o}1}^{\intercal} \tilde{\mathbf{Q}} = 0.
\end{equation}
However, \eqref{eq:stilde} is an underdetermined linear system. To make it a uniquely solvable, we add the following constraints: 
\begin{enumerate}
\item The revised $\tilde{\mathbf{Q}}$ must be symmetric.
\item An element of $\tilde{\mathbf{Q}}$ should remain zero if it is zero for all non-Maxwell gas molecules.
\item \label{item:mbciii} An element $Q_{ij} = \frac{2\chi}{2-\chi} \frac{\langle \phi_{\mathrm{odd}}^i, \mathcal{C} \xi_n^{-1} \phi_{\mathrm{odd}}^j \rangle}{\langle \phi_{\mathrm{odd}}^j, \phi_{\mathrm{odd}}^j \rangle}$ (refer to \eqref{eq:bc_odd}) of $\tilde{\mathbf{Q}}$ should remain unchanged if both $\phi_{\mathrm{odd}}^i$ and $\phi_{\mathrm{odd}}^j$ belong to $\mathbb{V}_{\mathrm{odd}} \cap (\mathbb{V}_{\mathrm{modified}})^{\perp}$, where
\begin{displaymath}
    \mathbb{V}_{\mathrm{modified}} = \operatorname{span}\Big\{ \phi_{n\tau_1}^{2}, \phi_{n\tau_2}^{2}, \phi_n^{2} + \frac{c^{3,1}_1}{c^{2,1}_1} \phi_n^{3}\Big\}
\end{displaymath}
is a subspace of the $\mathbb{V}_{\mathrm{odd}}$ in \eqref{eq:Vodd}.
\end{enumerate}
The rationale for proposing the constraint \eqref{item:mbciii} is that $\mathbb{V}_{\mathrm{modified}} = \emptyset$ in the case of Maxwell gas, which does not exhibit any non-physical boundary layers. The additional basis functions in $\mathbb{V}_{\mathrm{modified}}$ introduce these layers in the non-Maxwell cases. Therefore, we restricts us to modify those elements deduced from $\mathbb{V}_{\mathrm{modified}}$.

After imposing these conditions, the matrix $\tilde{\mathbf{Q}}$ is uniquely determined. Currently, we do not have theoretical guarantee of the semidefiniteness of the revised $\tilde{\mathbf{Q}}$, but this turns out to be true in all cases done in our tests. The symmetry and semidefiniteness of $\tilde{\mathbf{Q}}$ indicates that the structure \eqref{eq:r13bc} is well maintained after the modification, and thus does not ruin the $L^2$ stability or the second law of thermodynamics.

To demonstrate the effect of this modification, we reconsider the example described in section \ref{sec:boundary_layer}, with the boundary conditions changed to the new ones introduced in this section by replacing $\hat{m}_{ij}$ with ${m}_{ij}$ in \eqref{eq:1dbcv}--\eqref{eq:1dbcsigma}. The results are plotted in figure \ref{fig:fixbc}. A comparison with figure \ref{fig:layers} clearly demonstrates that the non-physical boundary layers have been removed.

\begin{figure}
    \centering
    \includegraphics[width=0.4\textwidth]{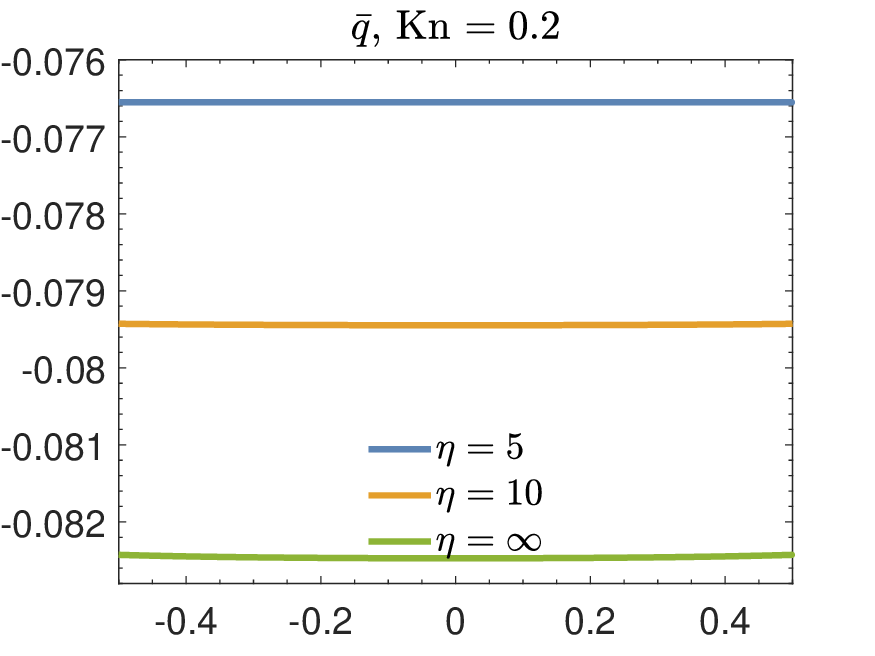}
    \hspace{0.05\textwidth}
    \includegraphics[width=0.4\textwidth]{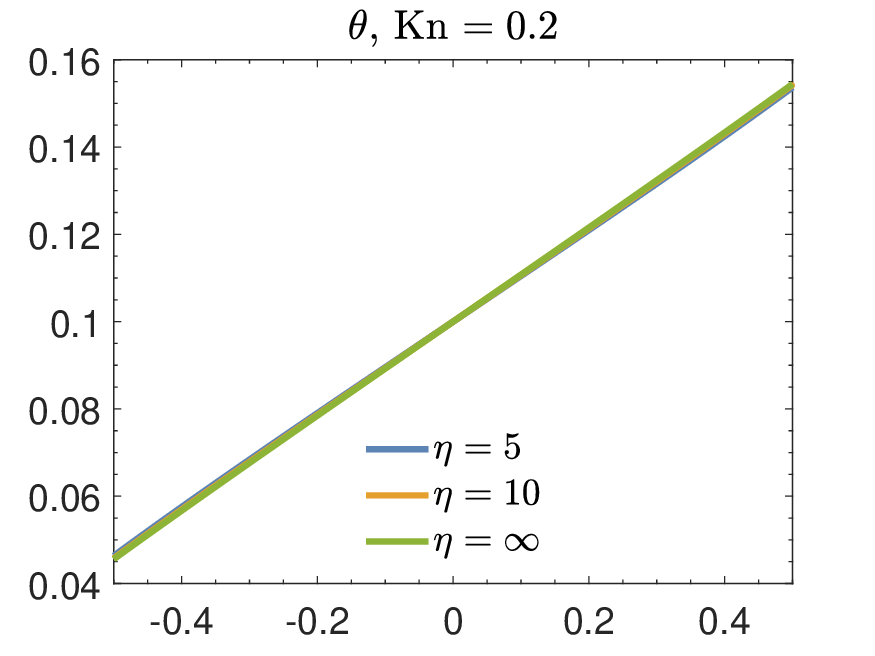}
    \caption{$\bar{q}$ and $\theta$ of the one-dimensional problem \eqref{eq:1deqrho}--\eqref{eq:1dbcsigma} with modified coefficient $\hat{m}_{ij} \to {m}_{ij}$. $\theta^W=0$ in the left boundary and $\theta^{W}=0.2$ in the right one.}
    \label{fig:fixbc}
\end{figure}

\section{One-dimensional simulations}
\label{sec:1Dexample}
In this section, we will apply our model to the example for one-dimensional channel flows in \cite{Hu2020,yang2024siap}. We will show the analytic solutions of the steady-state equations of our model and make a comparison with the results computed by direct simulation Monte Carlo (DSMC) from Bird's code (\cite{BirdBook1994}) in section \ref{sec:steadysol}.
After that, the numerical solutions of our time-dependent model will be demonstrated in section \ref{sec:timenumericalsol} using the finite element method.

The gas flow between two infinitely large parallel plates is considered in the one-dimensional channel flows. Both the two plates are perpendicular to the $x_2$-axis, and the two plates are assumed to move along $x_1$-axis with constant velocities. The distance of two plates is taken as $1$, and the coordinates are established such that the simulation domain is $x_2 \in (-0.5, 0.5)$. The temperature and velocity of the left plates ($x_2 = -0.5$) are denoted as $\theta^W_l$ and $v^W_l$, respectively, and similarly the temperature and velocity of the right plates ($x_2 = 0.5$) are denoted as $\theta^W_r$ and $v^W_r$, respectively. Under this setting, all the moments are functions only of $x_2$, and therefore the moment equations for the one-dimensional channel flows can be obtained via dropping all partial derivatives with respect to $x_1$ and $x_3$. Furthermore, as two plates move along $x_1$-axis, the distribution function enjoys symmetry $f(\bx, \xi_1,\xi_2,\xi_3,t)=f(\bx, \xi_1,\xi_2,-\xi_3, t)$ so that the moments that are odd in $\xi_3$ vanish. As a result, 
\begin{displaymath}
    v_3 = \bar{q}_3 = \bar{\sigma}_{13} = \bar{\sigma}_{23} =  0.
\end{displaymath}
Therefore in the one-dimensional channel flows, the $13$ moments are reduced to $9$ moments: $\rho$, $\theta$, $v_{1}$, $v_{2}$, $\bar{q}_{1}$, $\bar{q}_{2}$, $\bar{\sigma}_{11}$, $\bar{\sigma}_{12}$, $\bar{\sigma}_{22}$. In addition, both plates are assumed to be completely diffusive, where the accommodation coefficient is always chosen as $\chi = 1$.

\subsection{Steady-state examples}\label{sec:steadysol}
To compare with the example in sections \ref{sec:boundary_layer} \& \ref{sec:removelayer}, we consider a steady-state problem here. In fact, equations \eqref{eq:1deqrho}--\eqref{eq:1deqsigma} can be regarded as a portion of the considering one-dimensional steady-state channel flows by the following replacement: $x \to x_2$, $v \to v_2$, $\bar{q} \to \bar{q}_2$ and $\bar{\sigma} \to \bar{\sigma}_{22}$. The boundary conditions \eqref{eq:1dbcv}--\eqref{eq:1dbcsigma} also share the same form of the considering boundary conditions by the above replacement. We adopt the coefficients $\theta^W_l = 0$ and $\theta^W_r = 0.2$, consistent with the example presented in section \ref{sec:boundary_layer}.

To compare our solutions with the DSMC results, we show the results by using the Knudsen number $\overline{\kn}$ defined by the ratio of the mean free path to the characteristic length as in \cite{BirdBook1994}. The relationship between $\overline{\kn}$ and the parameter $\kn$ in this paper is
\begin{displaymath}
    \kn = \sqrt{\frac{\pi}{2}} \frac{15 \overline{\kn}}{(5 - 2 \omega)(7-2 \omega)}, \qquad \omega = \frac{1}{2} + \frac{2}{\eta - 1}.
\end{displaymath}
Moreover, we plot the actual heat flux $q_2$ instead of $\bar{q}_2$ (see \eqref{eq:qbartoq}) in this example, since $q_2$ is obtained from the DSMC results. To observe the effect of the Knudsen number, we choose two different $\overline{\kn}$ and plot the results for both original and fixed boundary conditions in figures \ref{fig:fixbckn01} \& \ref{fig:fixbckn005}. Our results are depicted as the blue, yellow, and green solid lines corresponding to $\eta = 5, 10$ and $\infty$. The relevant solutions obtained by DSMC are presented by dotted lines of the same colors.

\begin{figure}
    \centering
    \includegraphics[width=0.4\textwidth]{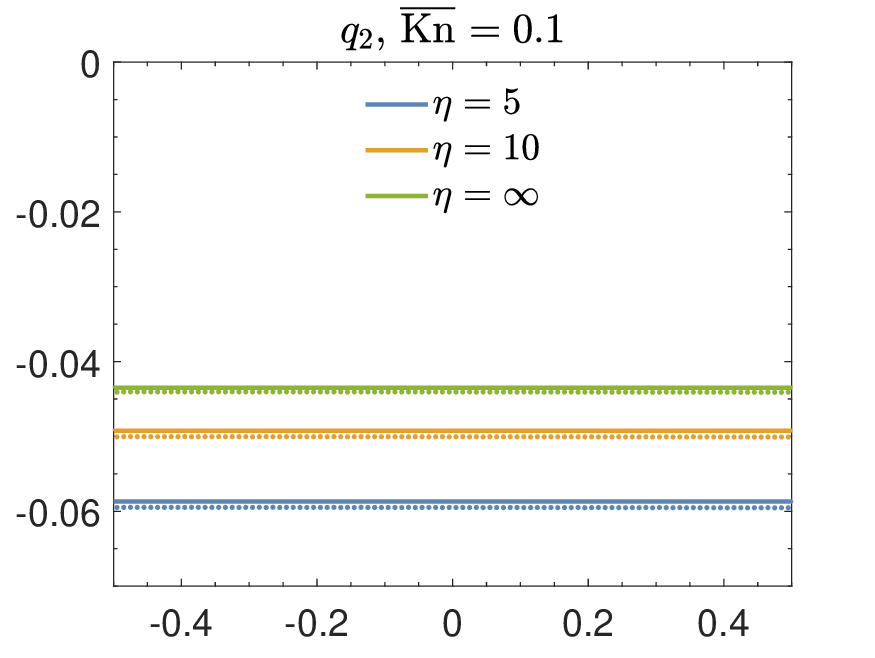}
    \hspace{0.05\textwidth}
    \includegraphics[width=0.4\textwidth]{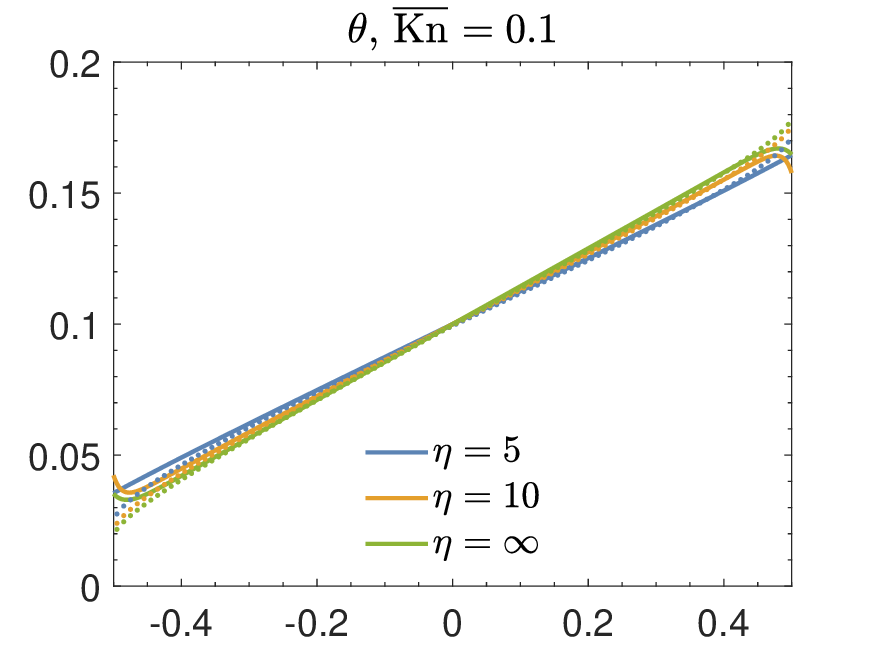} \\[10pt]
    \includegraphics[width=0.4\textwidth]{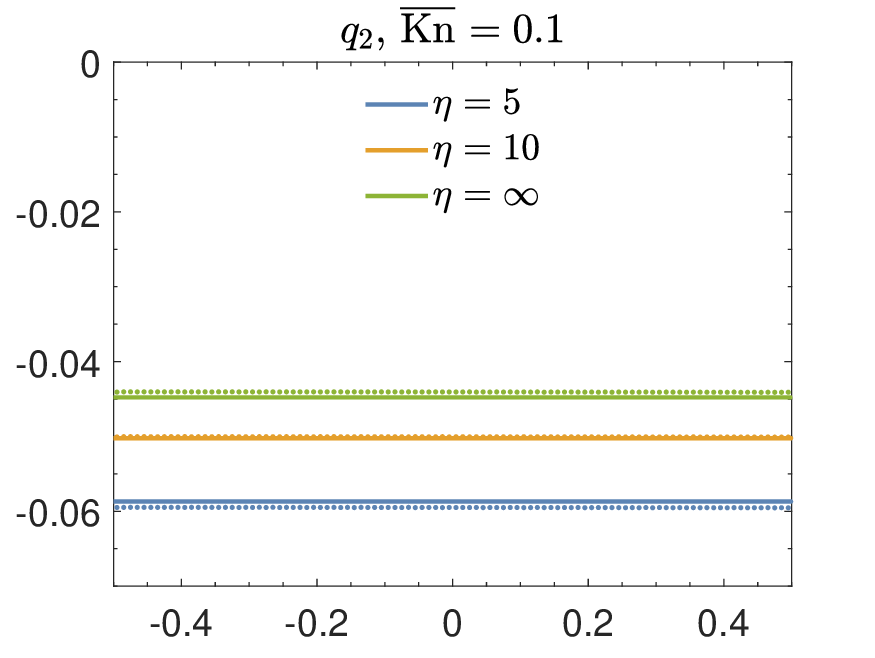}
    \hspace{0.05\textwidth}
    \includegraphics[width=0.4\textwidth]{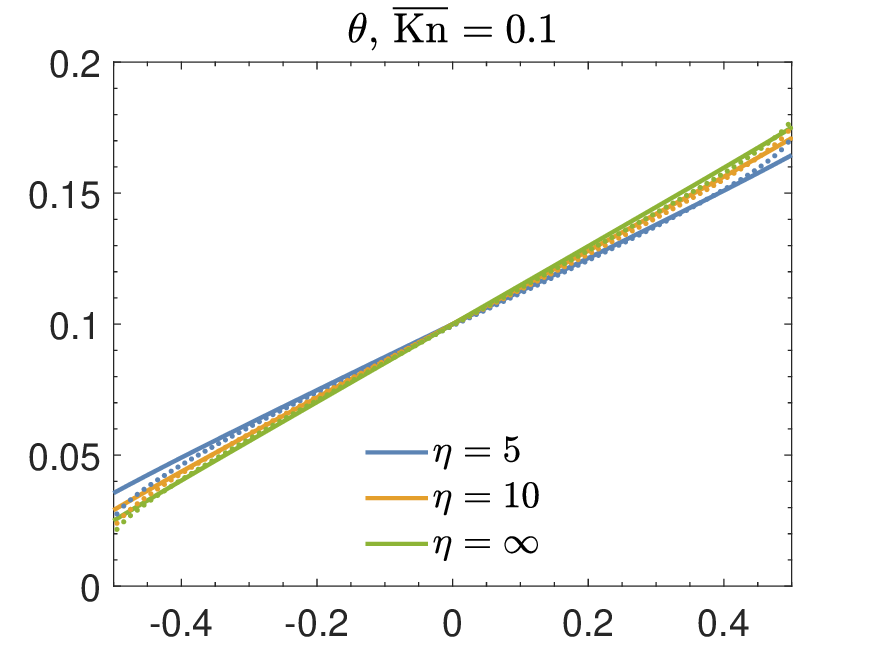}
    \caption{Results of steady-state example when $\overline{\kn}=0.1$. Top subfigures use coefficients $\hat{m}_{ij}$ and bottom subfigures use modified ones $m_{ij}$. DSMC solutions are given by dotted lines of the same colors.}
    \label{fig:fixbckn01}
\end{figure}

\begin{figure}
    \centering
    \includegraphics[width=0.4\textwidth]{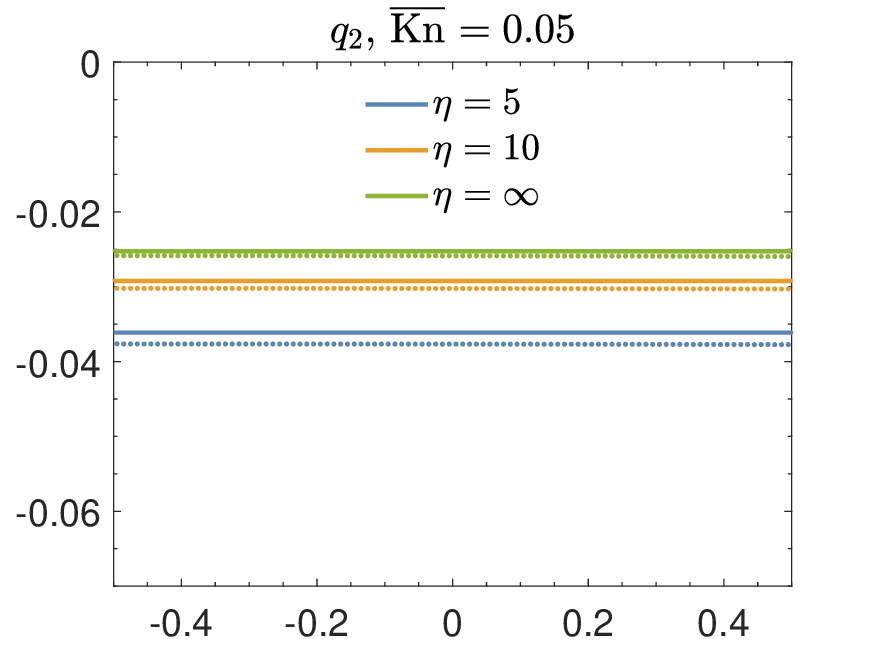}
    \hspace{0.05\textwidth}
    \includegraphics[width=0.4\textwidth]{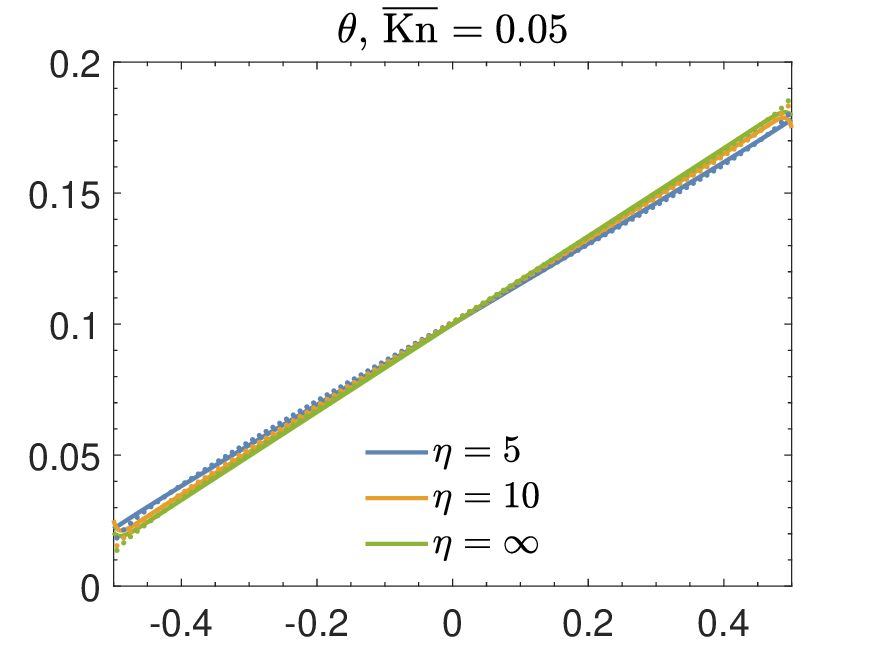} \\[10pt]
    \includegraphics[width=0.4\textwidth]{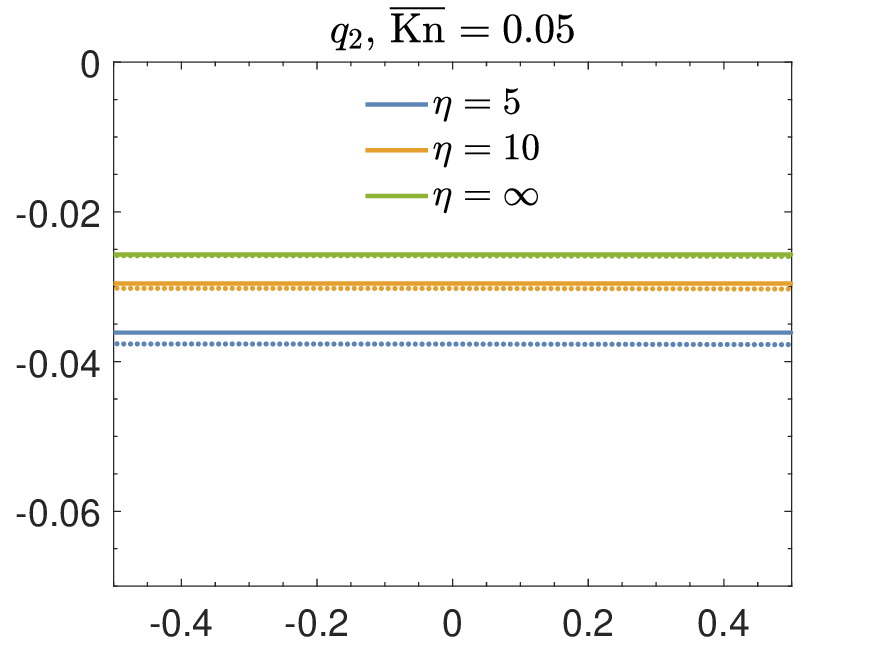}
    \hspace{0.05\textwidth}
    \includegraphics[width=0.4\textwidth]{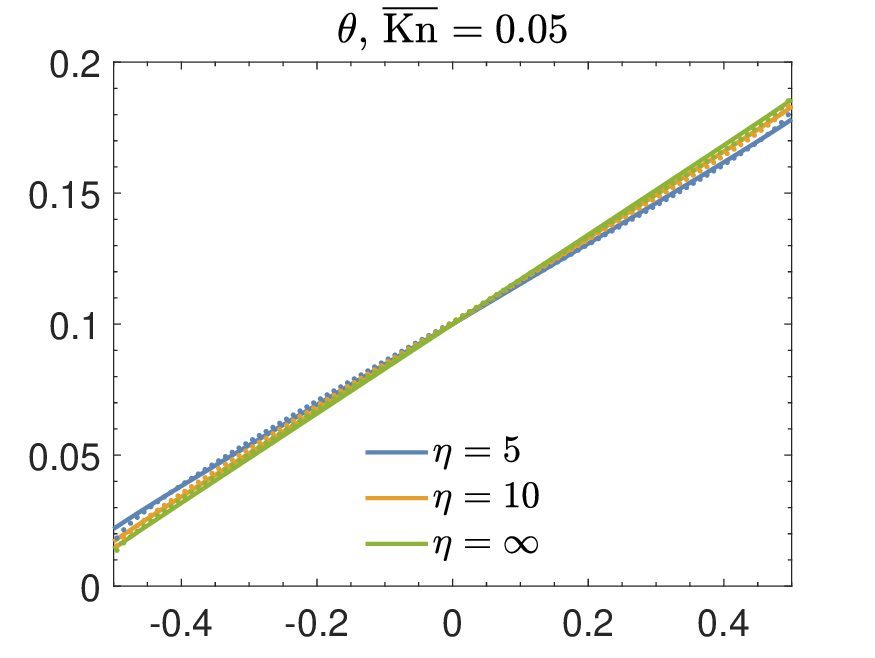}
    \caption{Results of steady-state example when $\overline{\kn}=0.05$. Top subfigures use coefficients $\hat{m}_{ij}$ and bottom subfigures use modified ones $m_{ij}$. DSMC solutions are given by dotted lines of the same colors.}
    \label{fig:fixbckn005}
\end{figure}

The solids lines of the R13 equations generally agree with the dotted lines of the DSMC results, and this agreement is more pronounced for smaller $\kn$, which validates our equations. Moreover, as the DSMC results display no boundary layer, the effectiveness of our modified Onsager boundary conditions is again validated by the successful removal of the undesired boundary layers. Furthermore, compared to the results in figure \ref{fig:layers}, the width of the undesired boundary layers decreases as the Knudsen number decreases, which is consistent with our analysis of the solution in \eqref{eq:1dsolform}. Also, it can be observed that the temperature jump gets more significant as the Knudsen number increases.

\subsection{Time-dependent examples}\label{sec:timenumericalsol}
After validating the steady-state equations, it would be more compelling to verify the derived model in the time-dependent case, as this is the primary focus of this work. However, it would be challenging to obtain the analytic time-dependent solution, therefore we aim to solve the time-dependent equations using the finite element method. 

The numerical setup is described as follows. We use a uniform mesh with a mesh size of $1/1000$ and apply the finite element method with piecewise linear functions. The Crank–Nicolson method is used for temporal discretization with a uniform time step of $1/4000$. The initial value for $\theta$ is set to $\frac{\theta^W_l+\theta^W_r}{2}$, while the initial values for all other moments are set to $0$. Boundary coefficients are taken from \cite{yang2024siap}, which will be specified in the following two cases.

\subsubsection{Couette flow}
In the planar Couette flow, the two plates move in opposite direction and maintain the same temperature. We choose $\theta^W_l = \theta^W_r = 0$ and $ v^W_l = -v^W_r = -0.2$. Results are depicted in figure \ref{fig:CouetteFlow}, where non-physical boundary layers can still be observed in the solution of time-dependent equations using the previous Onsager boundary conditions, although the width of these layers is very thin in this example. These thin layers are successfully removed at various time steps using our modified boundary conditions.  

\begin{figure}
    \centering
    \includegraphics[width=0.4\textwidth]{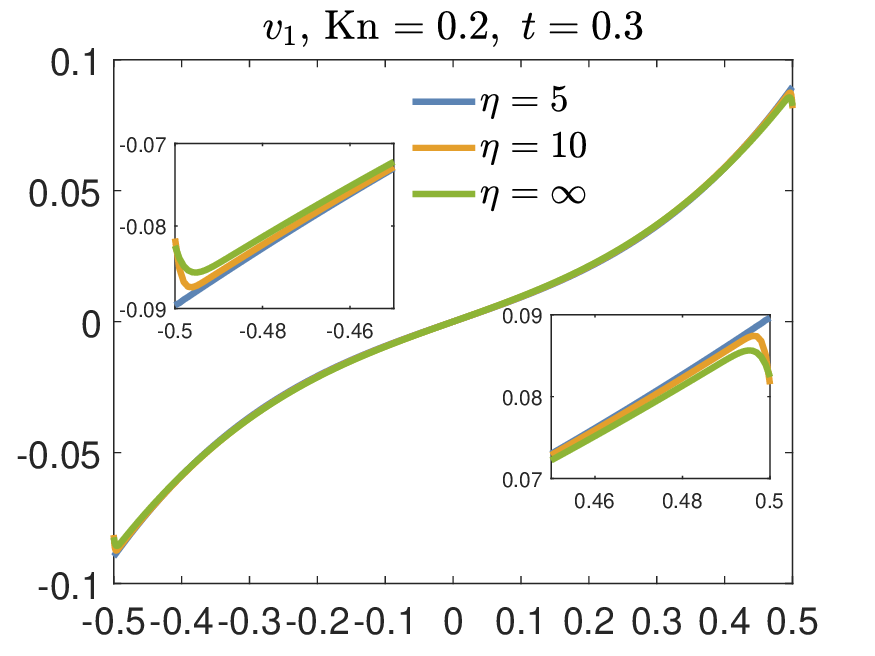}
    \hspace{0.05\textwidth}
    \includegraphics[width=0.4\textwidth]{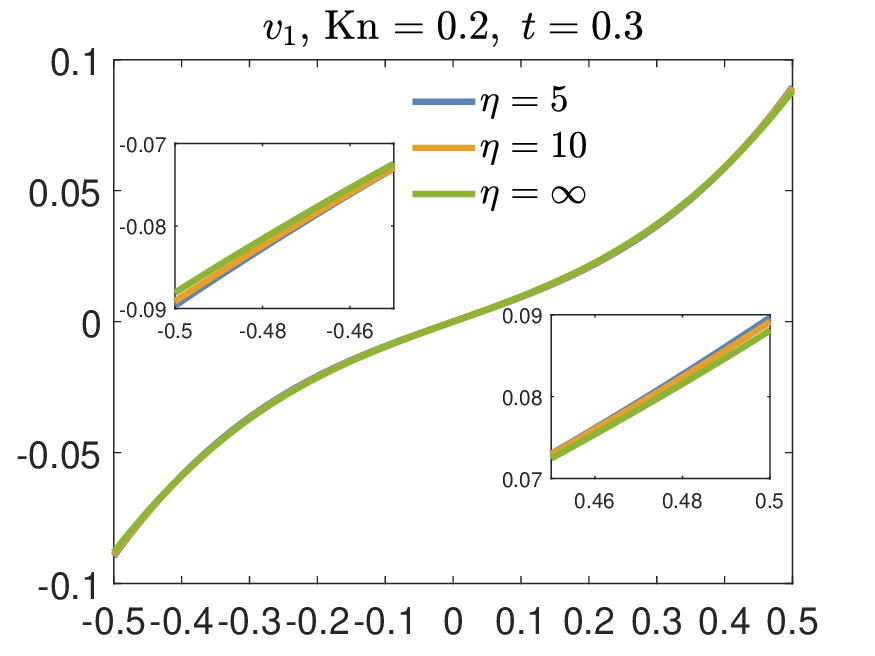} \\[10pt]
    \includegraphics[width=0.4\textwidth]{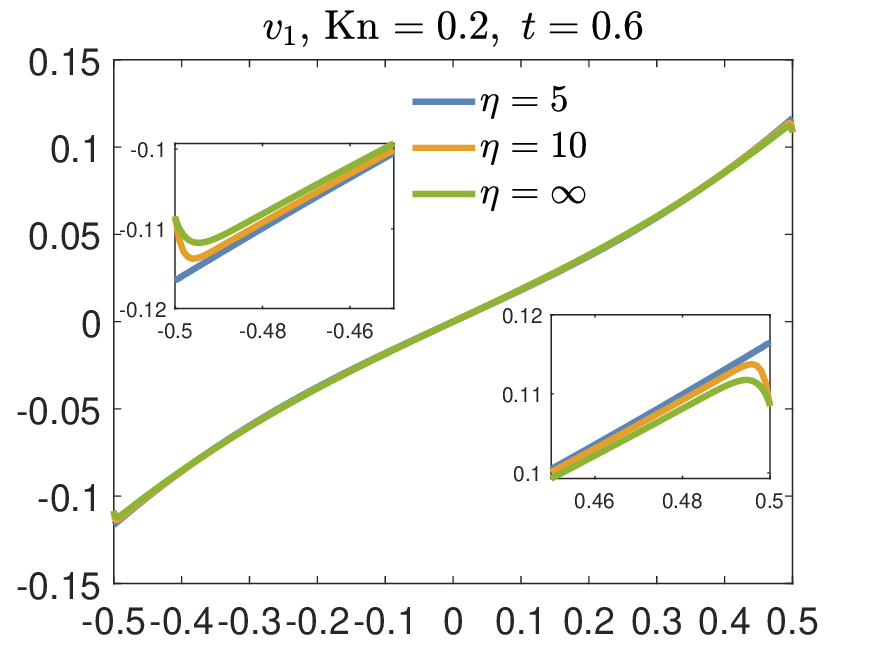}
    \hspace{0.05\textwidth}
    \includegraphics[width=0.4\textwidth]{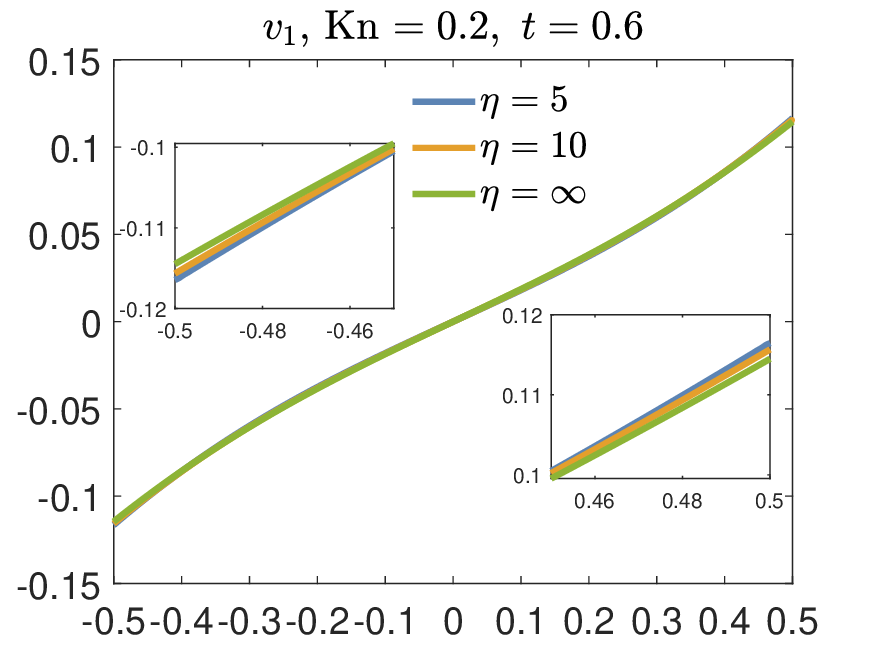}\\[10pt]
    \includegraphics[width=0.4\textwidth]{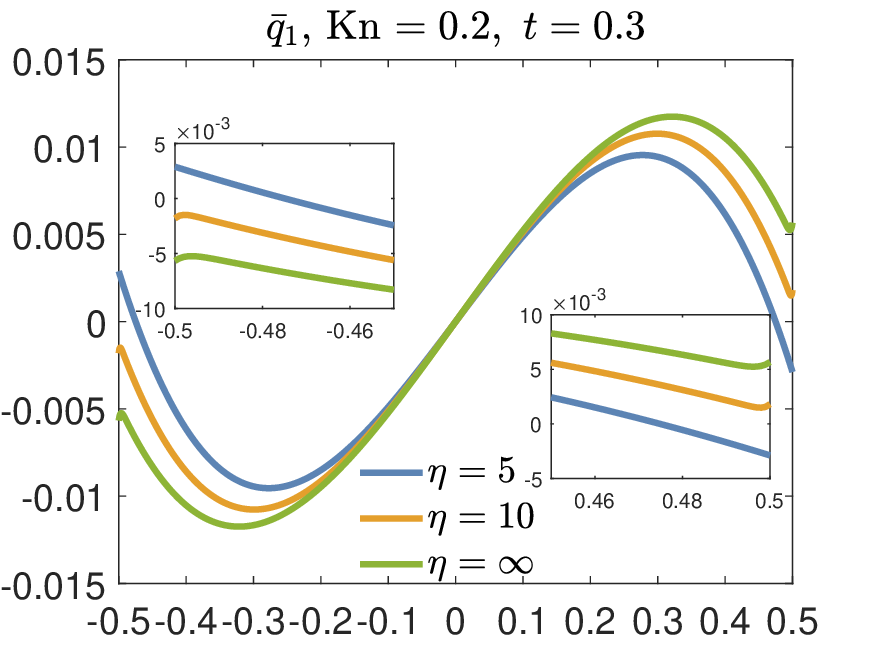}
    \hspace{0.05\textwidth}
    \includegraphics[width=0.4\textwidth]{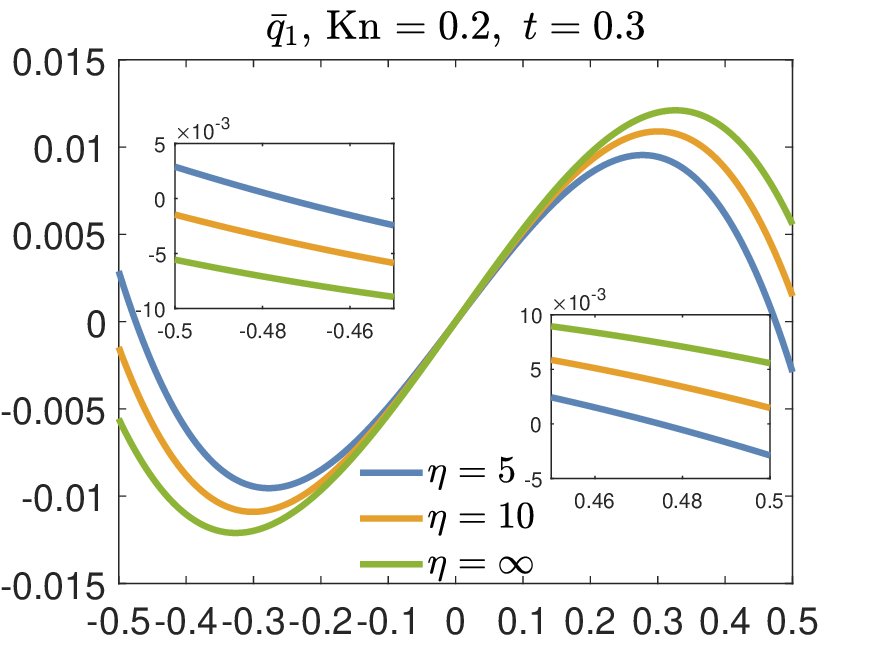}\\[10pt]
    \includegraphics[width=0.4\textwidth]{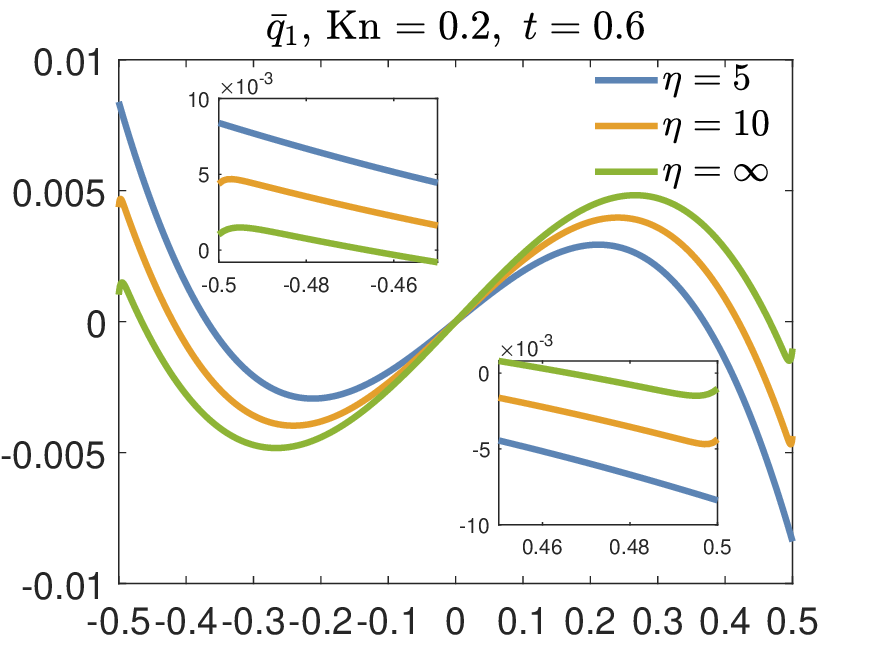}
    \hspace{0.05\textwidth}
    \includegraphics[width=0.4\textwidth]{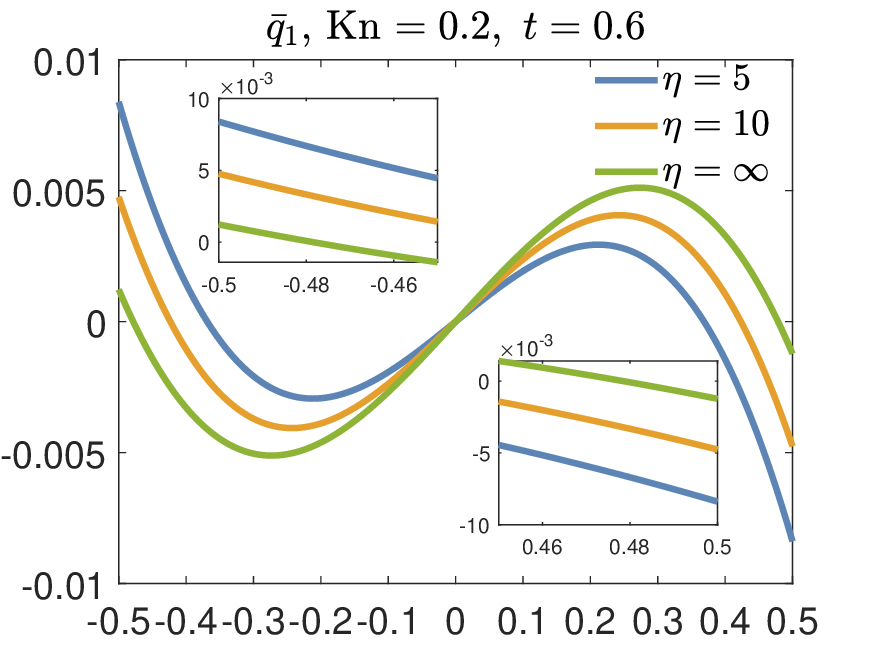}
    \caption{Results of the Couette flow. Left subfigures use previous Onsager boundary conditions and right subfigures use our modified boundary conditions.}
    \label{fig:CouetteFlow}
\end{figure}

\subsubsection{Fourier flow}
In the planar Fourier flow, the two plates are stationary, implying $v^W_l = v^W_r = 0$, and the gas dynamics between the plates are driven by the temperature difference. We set $\theta^W_l = 0 $ and $ \theta^W_r = 0.2$. The results in figure \ref{fig:FourierFlow} demonstrate that the previous Onsager boundary conditions produce non-physical boundary layers, which can be eliminated using our modified boundary conditions. Additionally, these layers are more pronounced for longer computational times, suggesting our modified boundary conditions would be more important in long-time simulations of the proposed R13 equations.

\begin{figure}
    \centering
    \includegraphics[width=0.4\textwidth]{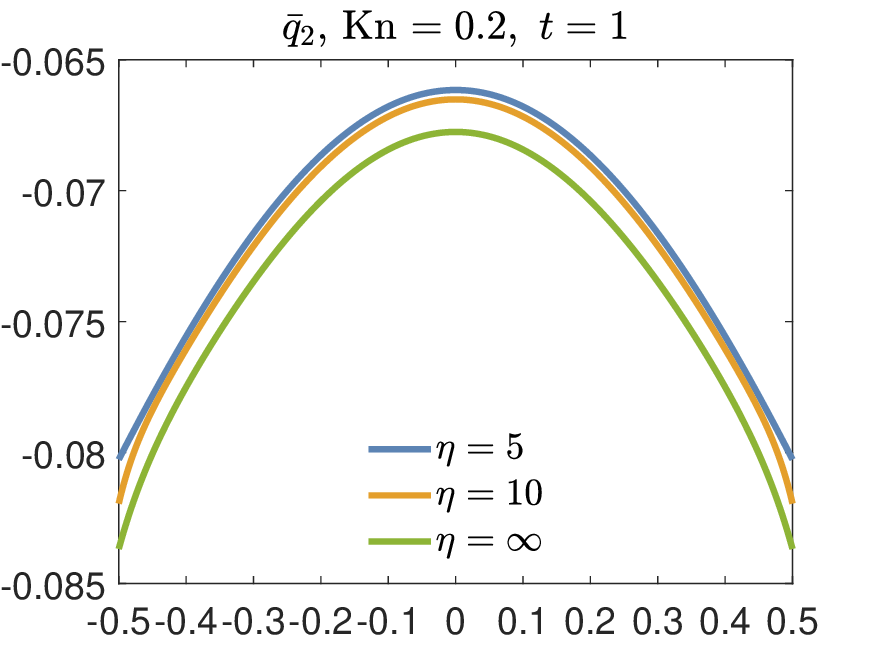}
    \hspace{0.05\textwidth}
    \includegraphics[width=0.4\textwidth]{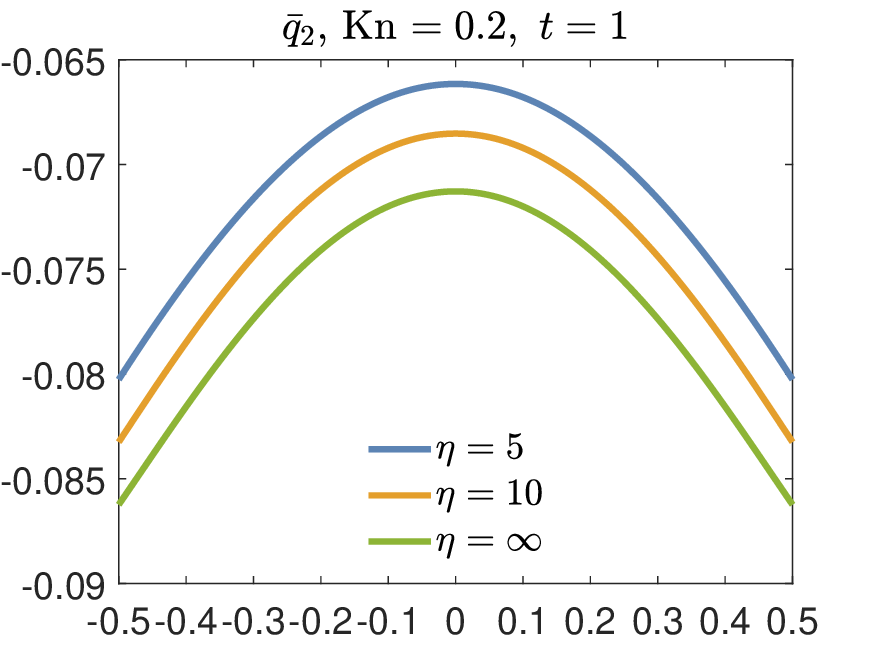} \\[10pt]
    \includegraphics[width=0.4\textwidth]{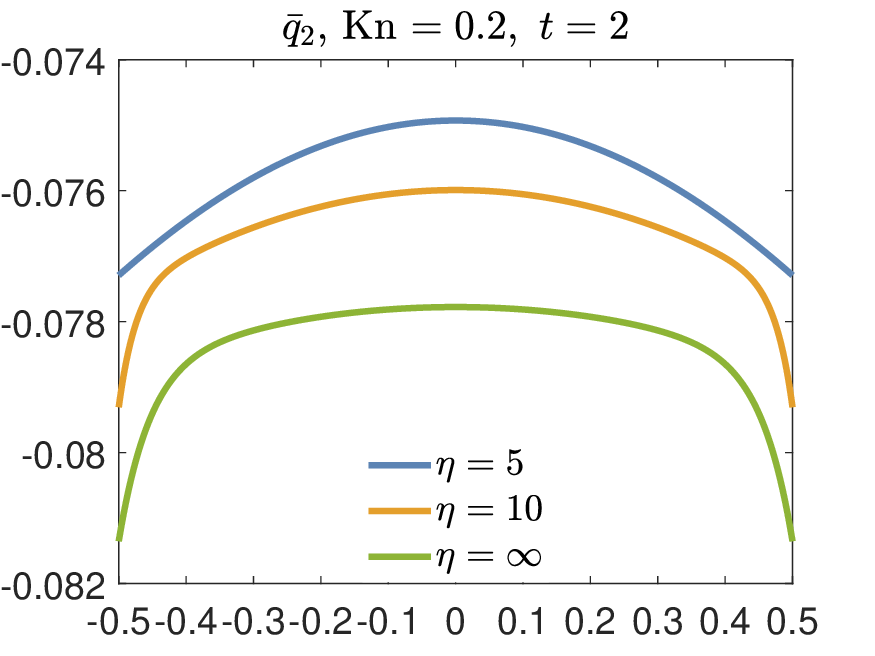}
    \hspace{0.05\textwidth}
    \includegraphics[width=0.4\textwidth]{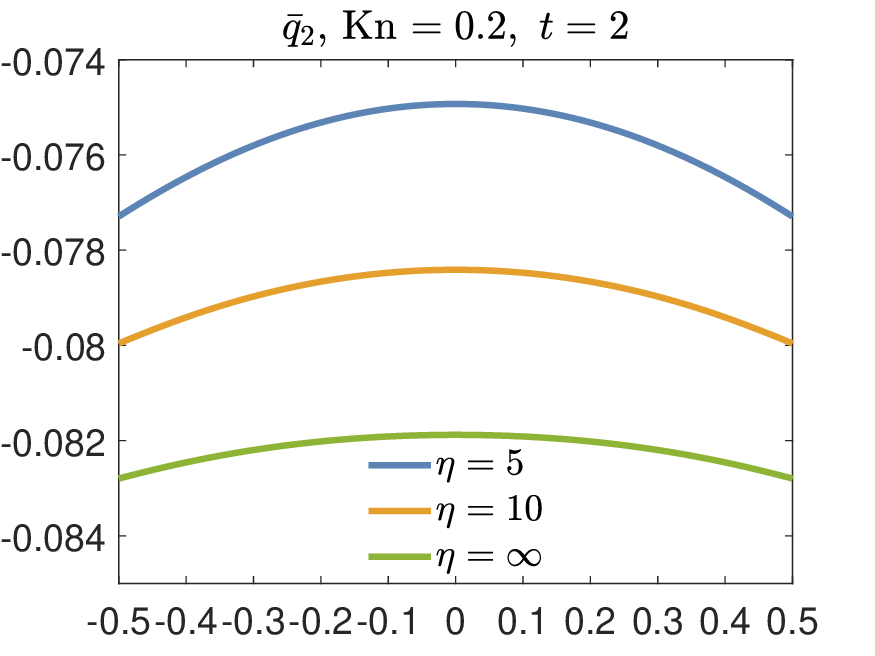} \\[10pt]
    \includegraphics[width=0.4\textwidth]{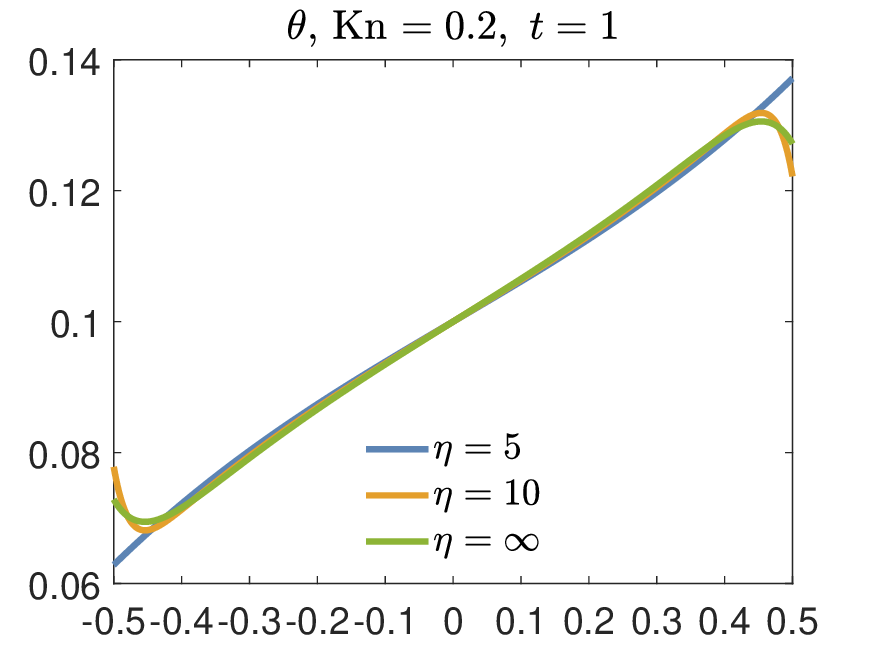}
    \hspace{0.05\textwidth}
    \includegraphics[width=0.4\textwidth]{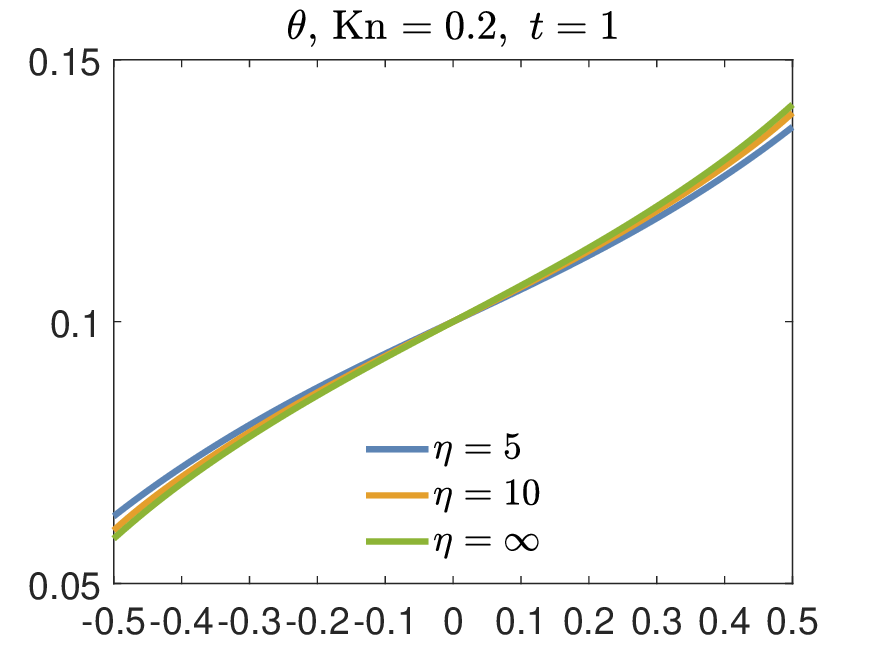} \\[10pt]
    \includegraphics[width=0.4\textwidth]{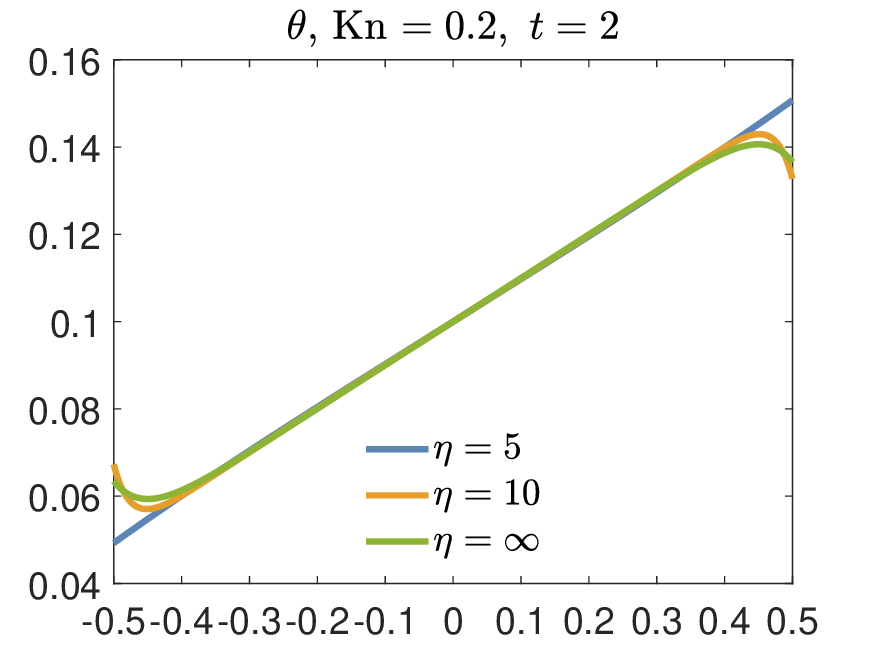}
    \hspace{0.05\textwidth}
    \includegraphics[width=0.4\textwidth]{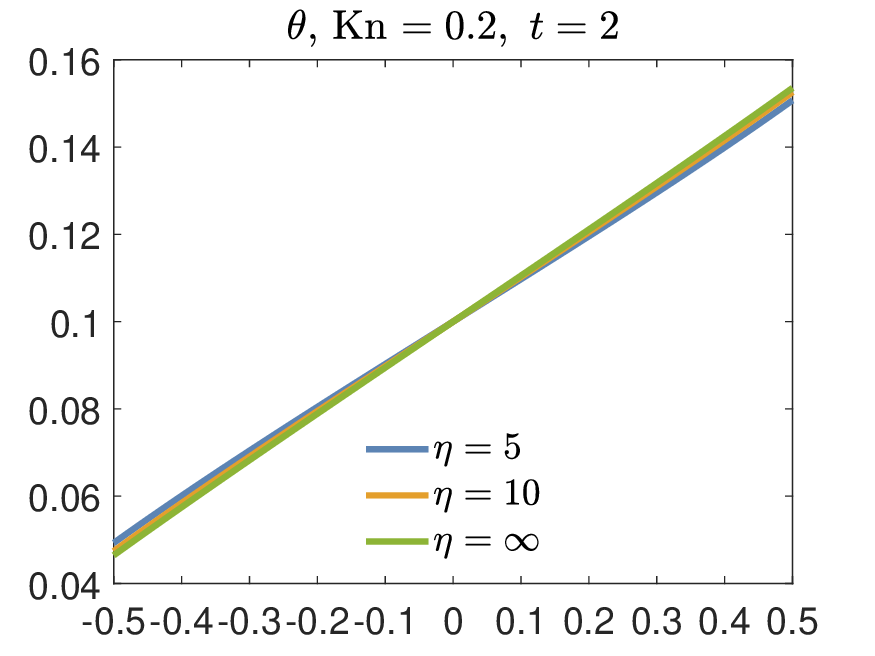}
    \caption{Results of the Fourier flow. Left subfigures use previous Onsager boundary conditions and right subfigures use our modified boundary conditions.}
    \label{fig:FourierFlow}
\end{figure}

\section{Conclusion}
\label{sec:conclusion}
This work provides a new perspective of the derivation of time-dependent regularized 13-moment equations from linear kinetic equation under general elastic collision models. 
Compared with previous works such as \cite{Hu2020} and \cite{yang2024siap}, our equations possess not only the super-Burnett order, but also a symmetric structure that is useful for deriving Onsager boundary conditions. {In the linear regime, our equations can well approximate the Boltzmann equation with moderate Knudsen numbers and preserve the conservation laws and the dissipation of entropy.} Another contribution of this work is the technique to remove undesired boundary layers, which needs to be applied since these layers are obviously unphysical.

Due to the symmetry of the equations and the Onsager boundary conditions, we expect that our equations also admit a symmetric weak form, which extends the weak form of R13 equations for Maxwell molecules (\cite{Theisen2021}). We are currently working on the finite element methods for our equations in the multidimensional case. Note that it is claimed in \cite{Rana2013} that R13 equations can provide satisfactory predictions up to Knudsen number $0.5$, and this needs to be tested for our equations with modified boundary conditions in multidimensional applications. Meanwhile, the work on applying this approach to nonlinear kinetic equations is also ongoing.
\\[10pt]
\paragraph{\textbf{Declaration of interests.}} The author reports no conflict of interest.

\appendix
\section{Galilean invariance of moment equations}
\label{app:galilean}
To demonstrate the Galilean invariance of our moment equations, we need to prove the following invariant properties:
\begin{itemize}
    \item Translational invariance: given any $\vb_{\text{ref}} \in \mathbb{R}^{3}$, under translation  
    \begin{gather*}
        \bx' = \bx + t \vb_{\text{ref}} , \  
 \rho'(\bx',t) = \rho(\bx,t) , \ \theta'(\bx',t) = \theta(\bx,t) , \\
 \vb'(\bx,t) = \vb(\bx',t) + \vb_{\text{ref}}, \ \qbarb'(\bx',t) = \qbarb(\bx,t) , \ \sigmabarb'(\bx',t) = \sigmabarb(\bx,t) ,
    \end{gather*}
    equations \eqref{newr13:eqa}--\eqref{newr13:eqe1} retain their form.
    \item Rotational invariance: given any orthogonal matrix $\Rb\in \mathbb{R}^{3\times 3}$, under rotation 
    \begin{gather*}
        \bx' = \Rb \bx , \ \rho'(\bx',t) = \rho(\bx,t) , \ \theta'(\bx',t) = \theta(\bx,t) ,\\
        \vb'(\bx',t) = \Rb\vb(\bx,t),  \ \qbarb'(\bx',t) = \Rb\qbarb(\bx,t) , \ \sigmabarb'(\bx',t) = \Rb\sigmabarb(\bx,t)\Rb^\intercal ,
    \end{gather*}
       equations \eqref{newr13:eqa}--\eqref{newr13:eqe1} retain their form. 
\end{itemize}

Since the equations rely solely on the derivatives of $\vb$ rather than on $\vb$ itself, the translational invariance can be readily verified in our linear equations. Consequently, we focus on the proof of the rotational invariance in this appendix. 

For \eqref{newr13:eqa}, we have 
\begin{equation}\label{inv_div_v}
    \frac{\partial v'_i}{\partial x'_i} =  \frac{\partial v'_i}{\partial x_k} \frac{\partial x_k}{\partial x'_i} =  R_{ij} \frac{\partial v_j}{\partial x_k} R_{ik} = \frac{\partial v_j}{\partial x_k} \delta_{jk} = \frac{\partial v_j}{\partial x_j}
\end{equation}
which implies that $\nabla_{\bx'}\cdot \vb' = \nabla_{\bx}\cdot \vb$. Therefore, \eqref{newr13:eqa} retains its form under the rotational transformation.

For \eqref{newr13:eqb}, we first calculate the derivative of $\sigmabar$ as
\begin{equation}\label{inv_div_sigma}
    \frac{\partial \sigmabar'_{ij}}{\partial x'_j} =  \frac{\partial \sigmabar'_{ij}}{\partial x_k}\frac{\partial x_k}{\partial x'_j} = R_{il}R_{jl'} \frac{\partial \sigmabar_{ll'}}{\partial x_k} R_{jk} =  R_{il}\delta_{kl'} \frac{\partial \sigmabar_{ll'}}{\partial x_k}  =R_{il}\frac{\partial \sigmabar_{lk}}{\partial x_k},
\end{equation}
and then we arrive at the relationship
\begin{displaymath}
    \frac{\partial^2 \sigmabar'_{ij}}{\partial x'_i\partial x'_j} = R_{il} \frac{\partial^2 \sigmabar_{lk}}{\partial x_k \partial x_{k'}} R_{ik'} = \frac{\partial^2 \sigmabar_{lk}}{\partial x_k \partial x_{k'}} \delta_{lk'} = \frac{\partial^2 \sigmabar_{k'k}}{\partial x_k \partial x_{k'}}
\end{displaymath}
which implies that $\nabla_{\bx'} \cdot (\nabla_{\bx'} \cdot \sigmabarb') = \nabla_{\bx} \cdot (\nabla_{\bx} \cdot \sigmabarb)$. Similar to \eqref{inv_div_v}, we can prove that $\nabla_{\bx'}\cdot \qbarb' = \nabla_{\bx}\cdot \qbarb$. In addition, $\Delta_{\bx'} \theta' = \Delta_{\bx} \theta$ according to the rotational invariance of Laplacian. As a result, \eqref{newr13:eqb} is also rotational invariant.   

For \eqref{newr13:eqc}, one can easily see that
\begin{equation}\label{inv_grad_v_rho_theta}
    \frac{\partial \vb'}{\partial \bx'} = \Rb \frac{\partial \vb}{\partial \bx}, \quad \nabla_{\bx'} \rho' = \Rb\nabla_{\bx} \rho, \quad \nabla_{\bx'} \theta' = \Rb\nabla_{\bx} \theta.
\end{equation}
In addition, we have $\nabla_{\bx'}\cdot \sigmabarb' = \Rb\nabla_{\bx}\cdot \sigmabarb$ by \eqref{inv_div_sigma}. For the term $\nabla \cdot (\nabla \vb)_\stf$, we use the fact 
\begin{equation}\label{inv_gradv}
\begin{split}
        (\nabla \vb)_\stf = &  \ \frac{1}{2}(\nabla \vb + (\nabla \vb)^\intercal) - \frac{1}{3}\text{tr}(\nabla \vb)I, \\ 
        \nabla_{\bx'}\vb' = & \  \Rb\nabla_{\bx}\vb \Rb^T,
        \end{split}
\end{equation}
and cyclic property of trace operator to show that $ (\nabla_{x'} \vb')_{\stf} = \Rb (\nabla_{x} \vb)_{\stf} \Rb^T$. Therefore, we have
\begin{displaymath}
    \frac{\partial }{\partial x'_j}[(\nabla_{x'} \vb')_{\stf}]_{ij} = R_{jk}   \frac{\partial }{\partial x_k} R_{il} [(\nabla_{x} \vb)_{\stf}]_{ll'} R_{jl'} = R_{il} \frac{\partial }{\partial x_k} [(\nabla_{x} \vb)_{\stf}]_{lk}
\end{displaymath}
which implies that $\nabla_{\bx'} \cdot (\nabla_{\bx'} \vb')_\stf = \Rb\nabla_{\bx} \cdot (\nabla_{\bx} \vb)_\stf$. Similarly, we can prove such relationship for $\qbarb$. Combining these two equalities of $\vb$ and $\qbarb$  in conjunction with \eqref{inv_div_sigma} and \eqref{inv_grad_v_rho_theta}, we have verified the rotational invariance of \eqref{newr13:eqc}. The rotational invariance of \eqref{newr13:eqd1} can be similarly proved as \eqref{newr13:eqc}.

For \eqref{newr13:eqe1}, similar as \eqref{inv_gradv}, one can prove that $\nabla_{\bx'}\qbarb' =   \Rb\nabla_{\bx}\qbarb \Rb^T$ and $(\nabla_{x'} \sigmabarb')_{\stf} = \Rb (\nabla_{x} \sigmabarb)_{\stf} \Rb^T$. For the term $\left(\nabla(\nabla\cdot \sigmabarb) \right)_\stf$, it holds that 
\begin{align*}
    [\nabla_{\bx'}(\nabla_{\bx'}\cdot \sigmabarb')]_{ij} = & \ \frac{\partial^2\sigmabar'_{ik}}{\partial x'_j\partial x'_k} \\
    = & \ R_{jk'}R_{il}R_{kl}R_{kl''} \frac{\partial^2 \sigmabar_{ll'}}{\partial x_{k'}\partial x_{l''}} = R_{jk'}R_{il}\delta_{ll''}\frac{\partial^2 \sigmabar_{ll'}}{\partial x_{k'}\partial x_{l''}} =  R_{il} \frac{\partial^2 \sigmabar_{ll'}}{\partial x_{k'}\partial x_{l}} R_{jk'} ,
\end{align*}
which implies that $(\nabla_{\bx'}(\nabla_{\bx'}\cdot \sigmabarb'))_{\stf} = \Rb (\nabla_{\bx}(\nabla_{\bx}\cdot \sigmabarb))_{\stf} \Rb^\intercal$ upon using the formula of $(\cdot)_{\stf}$ in \eqref{inv_gradv}. Furthermore, it is evident that $(\nabla^2_{\bx'}\theta')_\stf =\Rb (\nabla^2_{\bx}\theta)_\stf  \Rb^\intercal$. Therefore, the rotational invariance of \eqref{newr13:eqe1} has been verified. 

Overall, our moment equations \eqref{newr13:eqa}--\eqref{newr13:eqe1} satisfy the rotational invariance and, by encompassing the translational invariance, consequently satisfy the Galilean invariance.

\section{Tables of coefficients in the inverse-power-law model}
\label{app:coefficient}
Table \ref{tab:momentcoeff} lists the coefficients $k_i$ and $l_i$ of the moment equations in \eqref{newr13:eqa}--\eqref{newr13:eqe1} for some choices of parameter $\eta$ in the inverse-power-law model, and tables \ref{tab:bccoeff1} and \ref{tab:bccoeff2} list the coefficients $m_{ij}$ of the corresponding boundary conditions \eqref{eq:bcv}--\eqref{eq:bcm}.

\begin{table}
\centering \footnotesize
\begin{displaymath}
\begin{array}{| c | c | c | c | c | c |}
\hline
\hline
\eta &  k_{0}   &  k_{1}   &  k_{2}  &  k_{3} &  k_{4} \\
5 & 1 & 0 & 0 &  0 &  0 \\
7 & 9.9785\times 10^{-1} & 3.0773\times 10^{-3} & 1.2550\times 10^{-5} & 2.6072\times 10^{-3} & 4.8885\times 10^{-2} \\
10 & 9.9396\times 10^{-1} & 8.7436\times 10^{-3} & 4.5818\times 10^{-5} & 7.4080\times 10^{-3} & 8.1805\times 10^{-2}  \\ 
17 & 9.8883\times 10^{-1} & 1.6341\times 10^{-2} & 1.0021\times 10^{-4} & 1.3840\times 10^{-2} & 1.1124\times 10^{-1} \\
\infty & 9.7971\times 10^{-1} & 3.0261\times 10^{-2} & 2.0798\times 10^{-4} & 2.5607\times 10^{-2} & 1.5056\times 10^{-1} \\
\hline
\eta &  k_{5} &  k_{6} &  k_{7} & k_{8}  &  k_{9} \\
5 & 1 & 1 & 1 & 1 & 1 \\
7 & 9.9794\times 10^{-1} & 8.9911\times 10^{-1} & 9.7119\times 10^{-1} & 8.6773\times 10^{-1} & 9.6577\times 10^{-1} \\
10 & 9.9420\times 10^{-1} & 8.4173\times 10^{-1} & 9.5624\times 10^{-1} & 7.7962\times 10^{-1} & 9.4997\times 10^{-1} \\
17 & 9.8926\times 10^{-1} & 7.9687\times 10^{-1} & 9.4576\times 10^{-1} & 7.0221\times 10^{-1} & 9.4062\times 10^{-1} \\
\infty & 9.8041\times 10^{-1} & 7.4535\times 10^{-1} & 9.3584\times 10^{-1} & 6.0171\times 10^{-1} & 9.3491\times 10^{-1} \\
\hline
\eta &  k_{10} &  l_1 & l_2   \\
5 &  0 &  1  & 1 \\
7 & 2.8590\times 10^{-7} & 9.9420\times 10^{-1} & 9.9545\times 10^{-1} \\
10 & 1.1896\times 10^{-6} & 9.8385\times 10^{-1} & 9.8727\times 10^{-1} \\
17 & 2.8475\times 10^{-6} & 9.7049\times 10^{-1} & 9.7666\times 10^{-1} \\
\infty & 6.3621\times 10^{-6} & 9.4741\times 10^{-1} & 9.5812\times 10^{-1} \\
\hline
\hline
\end{array}
\end{displaymath}
\caption{Coefficients in moment equations \eqref{newr13:eqa}--\eqref{newr13:eqe1} for some power indices $\eta$ in the inverse-power-law model.\label{tab:momentcoeff}} 
\end{table}

\begin{table}
\centering \small
\begin{displaymath}
\begin{array}{| c | c | c | c | c | c |}
\hline
\hline
\eta &  m_{11}   &  m_{12}   &  m_{13}  &  m_{14} &  m_{15}  \\
5 & \frac{2}{\sqrt{2 \pi }} & \frac{1}{2 \sqrt{2 \pi }} & \frac{8}{5\sqrt{2 \pi }} & \frac{48}{25\sqrt{2 \pi }} & 0  \\
7 & 8.2699\times 10^{-1} & 1.7756\times 10^{-1} & 5.9524\times 10^{-1} & 7.7155\times 10^{-1} & 4.0454\times 10^{-2} \\
10 & 8.4706\times 10^{-1} & 1.6277\times 10^{-1} & 5.7153\times 10^{-1} & 7.7914\times 10^{-1} & 6.9432\times 10^{-2} \\ 
17 & 8.6507\times 10^{-1} & 1.4961\times 10^{-1} & 5.5359\times 10^{-1} & 7.8843\times 10^{-1} & 9.6598\times 10^{-2}  \\
\infty & 8.8890\times 10^{-1} & 1.3234\times 10^{-1} & 5.3390\times 10^{-1} & 8.0443\times 10^{-1} & 1.3481\times 10^{-1}  \\
\hline
\eta & m_{21} &  m_{22}   &  m_{23}  &  m_{24} &  m_{25} \\
5 & - & - & - & - & - \\
7 & 4.9821\times 10^{-3} & 1.9210\times 10^{-2} & 3.5585\times 10^{-3} & 4.6126\times 10^{-3} & 2.4185\times 10^{-4} \\
10 & 7.5057\times 10^{-3} & 3.1012\times 10^{-2} & 4.9556\times 10^{-3} & 6.7557\times 10^{-3} & 6.0203\times 10^{-4} \\
17 & 9.2216\times 10^{-3} & 4.0714\times 10^{-2} & 5.6679\times 10^{-3} & 8.0723\times 10^{-3} & 9.8900\times 10^{-4} \\
\infty & 1.0730\times 10^{-2} & 5.2344\times 10^{-2} & 5.9826\times 10^{-3} & 9.0140\times 10^{-3} & 1.5106\times 10^{-3} \\
\hline
\eta & m_{26} &  m_{27}   &  m_{28} \\
5 & - & - & - \\
7 & 2.9445\times 10^{-2} & 1.1082\times 10^{-1} & 4.5197\times 10^{-4} \\
10 & 5.0347\times 10^{-2} & 1.9074\times 10^{-1} & 9.9954\times 10^{-4} \\
17 & 7.0089\times 10^{-2} & 2.6780\times 10^{-1} & 1.6422\times 10^{-3} \\
\infty & 9.8823\times 10^{-2} & 3.8322\times 10^{-1} & 2.6338\times 10^{-3} \\
\hline
\eta  & m_{31} &  m_{32} &  m_{33} &  m_{34} & m_{35}\\
5 & \frac{1}{\sqrt{2 \pi }} & \frac{1}{5 \sqrt{2 \pi }} & 0 & \frac{2}{\sqrt{2 \pi }} & 0 \\
7 & 4.0358\times 10^{-1} & 7.2977\times 10^{-2} & 5.8325\times 10^{-8} & 7.8808\times 10^{-1} & 7.6807\times 10^{-6} \\
10 & 4.0655\times 10^{-1} & 6.8340\times 10^{-2} & 2.4645\times 10^{-7} & 7.8723\times 10^{-1} & 2.8477\times 10^{-5} \\
17 & 4.0908\times 10^{-1} & 6.4237\times 10^{-2} & 5.9816\times 10^{-7} & 7.9036\times 10^{-1} & 6.3150\times 10^{-5} \\
\infty & 4.1227\times 10^{-1} & 5.8927\times 10^{-2} & 1.3613\times 10^{-6} & 8.0020\times 10^{-1} & 1.3351\times 10^{-4} \\
\hline
\eta  & m_{41}  &  m_{42} &  m_{43} &  m_{44} &  m_{45} \\
5 & \frac{1}{\sqrt{2 \pi }} & \frac{11}{5\sqrt{2 \pi }} & 0 & \frac{2}{\sqrt{2 \pi }} & 0 \\
7 & 3.4342\times 10^{-1} & 8.5528\times 10^{-1} & 5.6647\times 10^{-8} & 7.6541\times 10^{-1} & 7.4598\times 10^{-6} \\
10 & 3.0441\times 10^{-1} & 8.4194\times 10^{-1} & 2.3349\times 10^{-7} & 7.4584\times 10^{-1} & 2.6980\times 10^{-5} \\
17 & 2.6826\times 10^{-1} & 8.3207\times 10^{-1} & 5.5187\times 10^{-7} & 7.2919\times 10^{-1} & 5.8263\times 10^{-5} \\
\infty & 2.1746\times 10^{-1} & 8.2380\times 10^{-1} & 1.2025\times 10^{-6} & 7.0683\times 10^{-1} & 1.1793\times 10^{-4} \\
\hline
\eta  & m_{46}   &  m_{47}  &  m_{48} \\
5 & 0 & 0 & \frac{24}{5} \\
7 & 1.2222\times 10^{-1} & 2.4505\times 10^{-1} & 4.7008 \\
10 & 2.0315\times 10^{-1} & 4.0914\times 10^{-1} & 4.6366 \\
17 & 2.7309\times 10^{-1} & 5.5322\times 10^{-1} & 4.5779 \\
\infty & 3.6056\times 10^{-1} & 7.3790\times 10^{-1} & 4.4916 \\
\hline
\hline
\end{array}
\end{displaymath}
\caption{Part $1$ of the coefficients in boundary conditions \eqref{eq:bcv}--\eqref{eq:bcm} for some power indices $\eta$ in the inverse-power-law model.\label{tab:bccoeff1}} 
\end{table}

\begin{table}
\centering \small
\begin{displaymath}
\begin{array}{| c | c | c | c | c | c |}
\hline
\hline
\eta &  m_{51} &  m_{52} &  m_{53} & m_{54}   &  m_{55} \\
5 & - & - & - & - & - \\
7 & 3.4282\times 10^{-2} & 8.5377\times 10^{-2} & 5.6547\times 10^{-9} & 7.6406\times 10^{-2} & 7.4466\times 10^{-7} \\
10 & 5.2206\times 10^{-2} & 1.4439\times 10^{-1} & 4.0044\times 10^{-8} & 1.2791\times 10^{-1} & 4.6270\times 10^{-6} \\
17 & 6.4303\times 10^{-2} & 1.9945\times 10^{-1} & 1.3228\times 10^{-7} & 1.7479\times 10^{-1} & 1.3966\times 10^{-5} \\
\infty & 7.3794\times 10^{-2} & 2.7955\times 10^{-1} & 4.0806\times 10^{-7} & 2.3986\times 10^{-1} & 4.0020\times 10^{-5} \\
\hline
\eta &  m_{56} &  m_{57} &  m_{58} \\
5 & - & - & - \\
7 & 2.7746\times 10^{-2} & 5.5632\times 10^{-2} & 4.7964\times 10^{-1} \\
10 & 6.0867\times 10^{-2} & 1.2259\times 10^{-1} & 8.1033\times 10^{-1} \\
17 & 1.0110\times 10^{-1} & 2.0482\times 10^{-1} & 1.1154 \\
\infty & 1.7187\times 10^{-1} & 3.5174\times 10^{-1} & 1.5445 \\
\hline
\eta &  m_{61} & m_{62}   &  m_{63}  &  m_{64} &  m_{65} \\
5 & \frac{2}{5\sqrt{2 \pi }} & \frac{7}{5\sqrt{2 \pi }} & \frac{8}{25\sqrt{2 \pi }} & \frac{48}{125\sqrt{2 \pi }} & 0 \\
7 & 1.5598\times 10^{-1} & 6.0144\times 10^{-1} & 1.1141\times 10^{-1} & 1.4441\times 10^{-1} & 7.5718\times 10^{-3} \\
10 & 1.5101\times 10^{-1} & 6.2395\times 10^{-1} & 9.9703\times 10^{-2} & 1.3592\times 10^{-1} & 1.2112\times 10^{-2} \\
17 & 1.4747\times 10^{-1} & 6.5107\times 10^{-1} & 9.0636\times 10^{-2} & 1.2909\times 10^{-1} & 1.5815\times 10^{-2} \\
\infty & 1.4418\times 10^{-1} & 7.0335\times 10^{-1} & 8.0389\times 10^{-2} & 1.2112\times 10^{-1} & 2.0298\times 10^{-2} \\
\hline
\eta &  m_{66} &  m_{67} & m_{68}  &  m_{69} \\
5 & 0 & 2 & 0 & 0 \\
7 & 2.1284\times 10^{-7} & 2.0460 & 9.0196\times 10^{-8} & 3.5825\times 10^{-7} \\
10 & 1.6606\times 10^{-6} & 2.0707 & 3.0801\times 10^{-7} & 2.8135\times 10^{-6} \\
17 & 5.1333\times 10^{-6} & 2.1281 & 6.3694\times 10^{-7} & 8.7714\times 10^{-6} \\
\infty & 1.5172\times 10^{-5} & 2.2756 & 1.2781\times 10^{-6} & 2.6311\times 10^{-5} \\
\hline
\eta &  m_{71} &  m_{81} \\
5 & \frac{1}{2 \sqrt{2 \pi }} & \frac{1}{2 \sqrt{2 \pi }} \\
7 & 2.0678\times 10^{-1} & 2.0678\times 10^{-1} \\
10 & 2.0996\times 10^{-1} & 2.0996\times 10^{-1} \\
17 & 2.1152\times 10^{-1} & 2.1152\times 10^{-1} \\
\infty & 2.1169\times 10^{-1} & 2.1169\times 10^{-1} \\
\hline 
\hline
\end{array}
\end{displaymath}
\caption{Part $2$ of the coefficients in boundary conditions \eqref{eq:bcv}--\eqref{eq:bcm} for some power indices $\eta$ in the inverse-power-law model.\label{tab:bccoeff2}} 
\end{table}

\section{Computation of the coefficients $\boldsymbol{\beta}$'s and $\boldsymbol{\gamma}$'s in the asymptotic analysis}
This section establishes the relationship between the coefficients $a_{lnn'}$ and $\beta$'s and $\gamma$'s in \eqref{eq:order0}--\eqref{eq:order3}. For any fixed $l$, we consider $a_{lnn'}$ as an infinite matrix. We can then define $b_{lnn'}^{(n_0)}$ as the inverse of a submatrix of $a_{lnn'}$:
\begin{displaymath}
\sum_{n' = n_0}^{+\infty} a_{lnn'} b_{ln'n''}^{(n_0)} = \delta_{nn''}, \qquad \forall n, n'' = n_0, n_0+1, \ldots.
\end{displaymath}
Due to the conservation of mass, momentum and energy, the quantities $b_{0nn'}^{(0)}$, $b_{0nn'}^{(1)}$ and $b_{1nn'}^{(0)}$ do not exist (see \eqref{eq:a_zero}). Since the collision operator $\mathcal{L}$ is negative semidefinite and has a kernel of five dimensions, all other coefficients $b_{lnn'}^{(n_0)}$ with nonnegative $l$, $n$, $n'$ and $n_0$ are well defined. These coefficients have been defined in \cite{yang2024siap} where R13 equations for steady-state flows are derived.

In practice, the infinite matrices $a_{lnn'}$ are truncated and $b_{lnn'}^{(n_0)}$ are obtained by taking the inverse directly. Then the coefficients $\beta$'s and $\gamma$'s in \eqref{eq:order0}--\eqref{eq:order3} are given by
\begin{gather*}
\beta_1^{(1),n} = b_{11n}^{(1)}, \qquad \beta_2^{(1),n} = b_{20n}^{(0)}, \\
\gamma^{(2),n}_0 =  \sum_{n'=2}^{+\infty} \frac{b_{0nn'}^{(2)} (\sqrt{2n'+3} b_{11n'}^{(1)} - \sqrt{2n'} b_{11,n'-1}^{(1)})}{b_{111}^{(1)}}, \\
\gamma_{t1}^{(1),n} = \sum_{n'=2}^{\infty} \frac{b^{(2)}_{1nn'}b^{(1)}_{11n'}}{b^{(1)}_{111}}, \qquad \gamma^{(1),n}_{s1} = \sum_{n'=2}^{\infty}
  \frac{b_{1nn'}^{(2)} (\sqrt{2n'+5} b_{20n'}^{(0)} - \sqrt{2n'} b_{20,n'-1}^{(0)})}{b_{200}^{(0)}}, \\
\gamma_{t2}^{(1),n} = \sum_{n'=2}^{\infty} \frac{b^{(1)}_{2nn'}b^{(0)}_{20n'}}{b^{(0)}_{200}}, \qquad   \gamma^{(1),n}_{s2} = \frac{2}{5} \sum_{n'=1}^{\infty}
  \frac{b_{2nn'}^{(1)} (\sqrt{2n'+5} b_{11n'}^{(1)} - \sqrt{2(n'+1)} b_{11,n'+1}^{(1)})}{b_{111}^{(1)}}, \\
\gamma^{(2),n}_3 =  \frac{3}{7} \sum_{n'=0}^{+\infty} \frac{b_{3nn'}^{(0)}}{b_{200}^{(0)}} \left(\sqrt{2n'+7} b_{20n'}^{(0)} - \sqrt{2(n'+1)} b_{20,n'+1}^{(0)}\right).
\end{gather*}
These results are obtained by asymptotic analysis. The derivation is similar to the steady-state case. Here we omit the details and refer interested readers to \cite{yang2024siap}.

\section{Computation of the coefficients $\boldsymbol{c_k^{\ell,n}}$} \label{sec:c}
The expressions of $c_1^{1,n}$ and $c_2^{0,n}$ have been introduced in \eqref{eq:c11n_c200} and \eqref{eq:c_normalization}. In this appendix, we will focus on the coefficients appearing in \eqref{eq:2nd_phi}.

The coefficients $c_0^{2,n}$ must be chosen such that $\langle \phi^2, \psi^n - d_{02}^n \psi^2 \rangle = 0$ for all $n \geqslant 3$. This leads to
\begin{displaymath}
c_0^{2,n} = d_{02}^n c_0^{2,2}, \qquad \forall n \geqslant 3.
\end{displaymath}
We choose the coefficient $c_0^{2,2}$ such that $c_0^{2,2} > 0$ and
\begin{displaymath}
\sum_{n=2}^{+\infty} |c_0^{2,n}|^2 = 1.
\end{displaymath}
In our implementation, we truncate this infinite series up to $n = 20$.

The coefficients $c_1^{2,n}$ and $c_1^{3,n}$ must be chosen such that for all $n \geqslant 4$, $\langle \phi_i^2, \varphi_1^n \rangle = \langle \phi_i^3, \varphi_1^n \rangle = 0$ for
\begin{displaymath}
\varphi_1^n = \left(\psi_i^n - \frac{\beta_1^{(1),n}}{\beta_1^{(1),1}} \psi_i^1 \right) - d_{12}^n\left( \psi_i^2 - \frac{\beta_1^{(1),2}}{\beta_1^{(1),1}} \psi_i^1 \right) - d_{13}^n \left( \psi_i^3 - \frac{\beta_1^{(1),3}}{\beta_1^{(1),1}} \psi_i^1\right).
\end{displaymath}
Therefore, the coefficients $c_1^{2,n}$ and $c_1^{3,n}$ must satisfy
\begin{equation} \label{eq:c_eq}
\begin{gathered}
c_1^{2,n} - d_{12}^n c_1^{2,2} - d_{13}^n c_1^{2,3} + \left( \frac{\beta_1^{(1),n}}{\beta_1^{(1),1}} - \frac{\beta_1^{(1),2}}{\beta_1^{(1),1}} d_{12}^n -\frac{\beta_1^{(1),3}}{\beta_1^{(1),1}} d_{13}^n \right) \sum_{n'=2}^{+\infty} \frac{\beta_1^{(1),n'}}{\beta_1^{(1),1}} c_1^{2,n'} = 0, \\
c_1^{3,n} - d_{12}^n c_1^{3,2} - d_{13}^n c_1^{3,3} + \left( \frac{\beta_1^{(1),n}}{\beta_1^{(1),1}} - \frac{\beta_1^{(1),2}}{\beta_1^{(1),1}} d_{12}^n -\frac{\beta_1^{(1),3}}{\beta_1^{(1),1}} d_{13}^n \right) \sum_{n'=2}^{+\infty} \frac{\beta_1^{(1),n'}}{\beta_1^{(1),1}} c_1^{3,n'} = 0
\end{gathered}
\end{equation}
for all $n \geqslant 4$. The linear systems for $c_1^{2,n}$ and $c_1^{3,n}$ are actually the same, and we just need to find linearly independent solutions for them. Notice that we need to find all coefficients $c_1^{2,n}$ and $c_1^{3,n}$ for $n \geqslant 2$, while the linear equations \eqref{eq:c_eq} are defined only for $n \geqslant 4$, the solutions are expected to form a two-dimensional linear space. Then $c_1^{2,n}$ and $c_1^{3,n}$ should form a basis of the space. In our implementation, we choose $c_1^{2,n}$ and $c_1^{3,n}$ such that
\begin{itemize}
\item The norms of $\phi_i^2$ and $\phi_i^3$ are equal to $1$;
\item $\langle \phi_i^2, \phi_i^3 \rangle = 0$;
\item $c_1^{2,2} > 0$, $c_1^{2,3} = 0$ and $c_1^{3,3} > 0$.
\end{itemize}
The last condition is to guarantee that when the collision model tends to Maxwell molecules, we have $\phi_i^2 \rightarrow \psi_i^2$ and $\phi_i^3 \rightarrow \psi_i^3$. To solve \eqref{eq:c_eq}, the system is again truncated up to $n = 20$.

The determination of $c_2^{1,n}$ and $c_2^{2,n}$ is similar to that of $c_1^{2,n}$ and $c_1^{3,n}$. The equations that $c_2^{1,n}$ and $c_2^{2,n}$ satisfy are
\begin{displaymath}
\begin{gathered}
c_2^{1,n} - d_{21}^n c_2^{1,1} - d_{22}^n c_2^{1,2} + \left( \frac{\beta_2^{(1),n}}{\beta_2^{(1),0}} - \frac{\beta_2^{(1),1}}{\beta_2^{(1),0}} d_{21}^n -\frac{\beta_2^{(1),2}}{\beta_2^{(1),0}} d_{22}^n \right) \sum_{n'=1}^{+\infty} \frac{\beta_2^{(1),n'}}{\beta_2^{(1),0}} c_2^{1,n'} = 0, \\
c_2^{2,n} - d_{21}^n c_2^{2,1} - d_{22}^n c_2^{2,2} + \left( \frac{\beta_2^{(1),n}}{\beta_2^{(1),0}} - \frac{\beta_2^{(1),1}}{\beta_2^{(1),0}} d_{21}^n -\frac{\beta_2^{(1),2}}{\beta_2^{(1),0}} d_{22}^n \right) \sum_{n'=1}^{+\infty} \frac{\beta_2^{(1),n'}}{\beta_2^{(1),0}} c_2^{2,n'} = 0.
\end{gathered}
\end{displaymath}
To determine these coefficients, we require that
\begin{itemize}
\item The norms of $\phi_{ij}^1$ and $\phi_{ij}^2$ are equal to $1$;
\item $\langle \phi_{ij}^1, \phi_{ij}^2 \rangle = 0$;
\item $c_2^{1,1} > 0$, $c_2^{1,2} = 0$ and $c_2^{2,2} > 0$.
\end{itemize}
When the collision model tends to Maxwell molecules, we have $\phi_{ij}^1 \rightarrow \psi_{ij}^1$ and $\phi_{ij}^2 \rightarrow \psi_{ij}^2$.

The determination of $c_3^{0,n}$ is similar to that of $c_0^{2,n}$. The solution is
\begin{displaymath}
c_3^{0,n} = d_{30} c_3^{0,0},
\end{displaymath}
where $c_3^{0,0}$ is selected to be positive and satisfies
\begin{displaymath}
\sum_{n=0}^{+\infty} |c_3^{0,n}|^2 = 1.
\end{displaymath}

\section{Expressions of $A_{ij}$}\label{sec:appendixAcoeff}
The expressions of $A_{ij}$ explicitly depends on the coefficients $c_k^{\ell,n}$, which are given as follows.
\begin{align*}
    & A_{45} = 3  \sum_{n=1}^{+\infty}c^{1,n}_1 \left(\sqrt{2n+5}c^{0,n}_2 - \sqrt{2n}c^{0,n-1}_2 \right),\\
    & A_{46} =  \sum_{n=1}^{+\infty}  c^{1,n}_1 \left( \sqrt{2n+3}c^{2,n}_0 - \sqrt{2(n+1)}c^{2,n+1}_0 \right),\\
    & A_{49} = 3  \sum_{n=1}^{+\infty}c^{1,n}_1 \left(\sqrt{2n+5}c^{1,n}_2 - \sqrt{2n}c^{1,n-1}_2 \right),\\
    & A_{4,10} = 3  \sum_{n=1}^{+\infty}c^{1,n}_1 \left(\sqrt{2n+5}c^{2,n}_2 - \sqrt{2n}c^{2,n-1}_2 \right),\\
    & A_{57} = 3   \sum_{n=0}^{+\infty} c^{0,n}_2 \left( \sqrt{2n+5}c^{2,n}_1 - \sqrt{2(n+1)} c^{2,n+1}_1 \right) ,\\
    & A_{58} = 3   \sum_{n=0}^{+\infty} c^{0,n}_2 \left( \sqrt{2n+5}c^{3,n}_1 - \sqrt{2(n+1)} c^{3,n+1}_1 \right) ,\\
    & A_{5,11} = \frac{15}{2}\sum_{n=0}^{+\infty} c^{0,n}_2 \left( \sqrt{2n+7} c^{0,n}_3 - \sqrt{2n}c^{0,n-1}_3 \right) ,\\
    & A_{67} = \sum_{n=2}^{+\infty} c^{2,n}_0 \left( \sqrt{2n+3} c^{2,n}_1 - \sqrt{2n}c^{2,n-1}_1 \right), \\
    & A_{68} = \sum_{n=2}^{+\infty} c^{2,n}_0 \left( \sqrt{2n+3} c^{3,n}_1 - \sqrt{2n}c^{3,n-1}_1 \right),\\
    & A_{79} = 3  \sum_{n=1}^{+\infty}c^{2,n}_1 \left(\sqrt{2n+5}c^{1,n}_2 - \sqrt{2n}c^{1,n-1}_2 \right),\\
    & A_{7,10} = 3  \sum_{n=1}^{+\infty}c^{2,n}_1 \left(\sqrt{2n+5}c^{2,n}_2 - \sqrt{2n}c^{2,n-1}_2 \right),\\
    & A_{89} = 3  \sum_{n=1}^{+\infty}c^{3,n}_1 \left(\sqrt{2n+5}c^{1,n}_2 - \sqrt{2n}c^{1,n-1}_2 \right),\\
    & A_{8,10} = 3  \sum_{n=1}^{+\infty}c^{3,n}_1 \left(\sqrt{2n+5}c^{2,n}_2 - \sqrt{2n}c^{2,n-1}_2 \right),\\
    & A_{9,11} = \frac{15}{2}  \sum_{n=0}^{+\infty} c^{1,n}_2 \left( \sqrt{2n+7} c^{0,n}_3 - \sqrt{2n}c^{0,n-1}_3 \right)  ,\\
    & A_{10,11} = \frac{15}{2}  \sum_{n=0}^{+\infty} c^{2,n}_2 \left( \sqrt{2n+7} c^{0,n}_3 - \sqrt{2n}c^{0,n-1}_3 \right) .
\end{align*}

\bibliographystyle{abbrv}
\bibliography{main}

\begin{thebibliography}{10}

\bibitem{Onsager_2019book}
A.~Agrawal, H.~M. Kushwaha, and R.~S. Jadhav.
\newblock {\em Microscale Flow and Heat Transfer}.
\newblock Springer Cham, 2019.

\bibitem{Alekseenko2022jcp_data}
A.~Alekseenko, R.~Martin, and A.~Wood.
\newblock Fast evaluation of the {Boltzmann} collision operator using data driven reduced order models.
\newblock {\em Journal of Computational Physics}, 470:111526, 2022.

\bibitem{Beckmann2018}
A.~F. Beckmann, A.~S. Rana, M.~Torrilhon, and H.~Struchtrup.
\newblock Evaporation boundary conditions for the linear r13 equations based on the onsager theory.
\newblock {\em Entropy}, 20(9):680, 2018.

\bibitem{BGK1954}
P.~L. Bhatnagar, E.~P. Gross, and M.~Krook.
\newblock A model for collision processes in gases. i. small amplitude processes in charged and neutral one-component systems.
\newblock {\em Phys. Rev.}, 94:511--525, May 1954.

\bibitem{BirdDSMC}
G.~A. {Bird}.
\newblock {Direct Simulation and the {Boltzmann} Equation}.
\newblock {\em Physics of Fluids}, 13(11):2676--2681, Nov. 1970.

\bibitem{BirdBook1994}
G.~A. {Bird}.
\newblock {\em Molecular Gas Dynamics And The Direct Simulation Of Gas Flows}.
\newblock Clarendon Press, Oxford, 05 1994.

\bibitem{bobylev2006instabilities}
A.~V. Bobylev.
\newblock Instabilities in the {Chapman-Enskog} expansion and hyperbolic {Burnett} equations.
\newblock {\em Journal of Statistical Physics}, 124:371--399, 2006.

\bibitem{bunger2023KRM}
J.~B\"unger, E.~Christhuraj, A.~Hanke, and M.~Torrilhon.
\newblock Structured derivation of moment equations and stable boundary conditions with an introduction to symmetric, trace-free tensors.
\newblock {\em Kinetic and Related Models}, 16(3):458--494, 2023.

\bibitem{lin2024sinum}
Z.~Cai, B.~Lin, and M.~Lin.
\newblock A positive and moment-preserving {Fourier} spectral method.
\newblock {\em SIAM Journal on Numerical Analysis}, 62(1):273--294, 2024.

\bibitem{Cai2015}
Z.~Cai and M.~Torrilhon.
\newblock Approximation of the linearized {B}oltzmann collision operator for hard-sphere and inverse-power-law models.
\newblock {\em J. Comput. Phys.}, 295:617--643, 2015.

\bibitem{yang2024siap}
Z.~Cai, M.~Torrilhon, and S.~Yang.
\newblock Linear regularized 13-moment equations with {Onsager} boundary conditions for general gas molecules.
\newblock {\em SIAM Journal on Applied Mathematics}, 84(1):215--245, 2024.

\bibitem{Cai2020regularized}
Z.~Cai and Y.~Wang.
\newblock Regularized 13-moment equations for inverse power law models.
\newblock {\em J. Fluid Mech.}, 894:A12, 2020.

\bibitem{chapman1970mathematical}
S.~Chapman and T.~G. Cowling.
\newblock {\em The Mathematical Theory of Non-Uniform Gases}.
\newblock Cambridge University Press, 1970.

\bibitem{Claydon2017}
R.~Claydon, A.~Shrestha, A.~S. Rana, J.~E. Sprittles, and D.~A. Lockerby.
\newblock Fundamental solutions to the regularised 13-moment equations: efficient computation of three-dimensional kinetic effects.
\newblock {\em Journal of Fluid Mechanics}, 833:R4, 2017.

\bibitem{Dimarco2018jcp}
G.~Dimarco, R.~Loubère, J.~Narski, and T.~Rey.
\newblock An efficient numerical method for solving the {Boltzmann} equation in multidimensions.
\newblock {\em Journal of Computational Physics}, 353:46--81, 2018.

\bibitem{Dimarco_Pareschi_2014_acta}
G.~Dimarco and L.~Pareschi.
\newblock Numerical methods for kinetic equations.
\newblock {\em Acta Numerica}, 23:369–--520, 2014.

\bibitem{Gamba2017fastspectral}
I.~M. Gamba, J.~R. Haack, C.~D. Hauck, and J.~Hu.
\newblock A fast spectral method for the {Boltzmann} collision operator with general collision kernels.
\newblock {\em SIAM Journal on Scientific Computing}, 39(4):B658--B674, 2017.

\bibitem{Grad13_1949}
H.~Grad.
\newblock On the kinetic theory of rarefied gases.
\newblock {\em Communications on Pure and Applied Mathematics}, 2(4):331--407, 1949.

\bibitem{Han2019learningclosure}
J.~Han, C.~Ma, Z.~Ma, and W.~E.
\newblock Uniformly accurate machine learning-based hydrodynamic models for kinetic equations.
\newblock {\em Proceedings of the National Academy of Sciences}, 116(44):21983--21991, 2019.

\bibitem{Hu2020}
Z.~Hu, S.~Yang, and Z.~Cai.
\newblock Flows between parallel plates: Analytical solutions of regularized 13-moment equations for inverse-power-law models.
\newblock {\em Phys. Fluids}, 32(12):122007, 2020.

\bibitem{jadhav2023oburnett}
R.~S. Jadhav, U.~Yadav, and A.~Agrawal.
\newblock {OBurnett} equations: thermodynamically consistent continuum theory beyond the {Navier-Stokes} regime.
\newblock {\em ASME Journal of Heat and Mass Transfer}, 145(6):060801, 2023.

\bibitem{Jiang2013}
K.~Jiang, D.~Sun, and K.~Toh.
\newblock {\em Discrete Geometry and Optimization}, chapter Solving Nuclear Norm Regularized and Semidefinite Matrix Least Squares Problems with Linear Equality Constraints, pages 133--162.
\newblock Springer, 2013.

\bibitem{jin2001regularization}
S.~Jin and M.~Slemrod.
\newblock Regularization of the {Burnett} equations via relaxation.
\newblock {\em Journal of Statistical Physics}, 103:1009--1033, 2001.

\bibitem{Maxwell1879bc}
J.~C. Maxwell.
\newblock On stresses in rarified gases arising from inequalities of temperature.
\newblock {\em Philosophical Transactions of the Royal Society of London}, 170:231--256, 1879.

\bibitem{Mieussens2000jcp}
L.~Mieussens.
\newblock Discrete-velocity models and numerical schemes for the {Boltzmann-BGK} equation in plane and axisymmetric geometries.
\newblock {\em Journal of Computational Physics}, 162(2):429--466, 2000.

\bibitem{muller2003extended}
I.~M{\"u}ller, D.~Reitebuch, and W.~Weiss.
\newblock Extended thermodynamics--consistent in order of magnitude.
\newblock {\em Continuum Mechanics and Thermodynamics}, 15:113--146, 2003.

\bibitem{Myong1999nccr}
R.~S. Myong.
\newblock {Thermodynamically consistent hydrodynamic computational models for high-{Knudsen}-number gas flows}.
\newblock {\em Physics of Fluids}, 11(9):2788--2802, 09 1999.

\bibitem{Ottinger2010thermo}
H.~C. \"Ottinger.
\newblock Thermodynamically admissible 13 moment equations from the {Boltzmann} equation.
\newblock {\em Phys. Rev. Lett.}, 104:120601, Mar 2010.

\bibitem{Rana2013}
A.~Rana, M.~Torrilhon, and H.~Struchtrup.
\newblock A robust numerical method for the {R}13 equations of rarefied gas dynamics: Application to lid driven cavity.
\newblock {\em J. Comput. Phys.}, 236:169--186, 2013.

\bibitem{rana2015numerical}
A.~S. Rana, A.~Mohammadzadeh, and H.~Struchtrup.
\newblock A numerical study of the heat transfer through a rarefied gas confined in a microcavity.
\newblock {\em Continuum Mechanics and Thermodynamics}, 27:433--446, 2015.

\bibitem{sarna2018stable}
N.~Sarna and M.~Torrilhon.
\newblock On stable wall boundary conditions for the hermite discretization of the linearised {Boltzmann} equation.
\newblock {\em Journal of Statistical Physics}, 170:101--126, 2018.

\bibitem{Agrawal2016pre}
N.~Singh and A.~Agrawal.
\newblock Onsager's-principle-consistent 13-moment transport equations.
\newblock {\em Phys. Rev. E}, 93:063111, Jun 2016.

\bibitem{Struchtrup20052ndorder}
H.~Struchtrup.
\newblock Derivation of 13 moment equations for rarefied gas flow to second order accuracy for arbitrary interaction potentials.
\newblock {\em Multiscale Modeling \& Simulation}, 3(1):221--243, 2005.

\bibitem{Struchtrup2005macroscopic}
H.~Struchtrup.
\newblock {\em Macroscopic transport equations for rarefied gas flows}.
\newblock Springer Berlin, Heidelberg, 2005.

\bibitem{Struchtrup2017}
H.~Struchtrup, A.~Beckmann, A.~S. Rana, and A.~Frezzotti.
\newblock Evaporation boundary conditions for the {R}13 equations of rarefied gas dynamics.
\newblock {\em Physics of Fluids}, 29(9):092004, 2017.

\bibitem{StruchtrupR13_2003}
H.~Struchtrup and M.~Torrilhon.
\newblock {Regularization of {Grad}’s 13 moment equations: Derivation and linear analysis}.
\newblock {\em Physics of Fluids}, 15(9):2668--2680, 09 2003.

\bibitem{Struchtrup2007}
H.~Struchtrup and M.~Torrilhon.
\newblock $h$ theorem, regularization, and boundary conditions for linearized 13 moment equations.
\newblock {\em Phys. Rev. Lett.}, 99:014502, 2007.

\bibitem{Struchtrup2013pof}
H.~Struchtrup and M.~Torrilhon.
\newblock {Regularized 13 moment equations for hard sphere molecules: Linear bulk equations}.
\newblock {\em Physics of Fluids}, 25(5):052001, 05 2013.

\bibitem{Taheri2012pre_r13_app}
P.~Taheri and M.~Bahrami.
\newblock Macroscopic description of nonequilibrium effects in thermal transpiration flows in annular microchannels.
\newblock {\em Phys. Rev. E}, 86:036311, Sep 2012.

\bibitem{Theisen2021}
L.~Theisen and M.~Torrilhon.
\newblock fenicsr13: A tensorial mixed finite element solver for the linear r13 equations using the fenics computing platform.
\newblock {\em ACM Trans. Math. Softw.}, 47(2), 2021.

\bibitem{r13_compare_2017}
M.~Y. Timokhin, H.~Struchtrup, A.~A. Kokhanchik, and Y.~A. Bondar.
\newblock {Different variants of {R13} moment equations applied to the shock-wave structure}.
\newblock {\em Physics of Fluids}, 29(3):037105, 03 2017.

\bibitem{Torrilhon2017}
M.~Torrilhon and N.~Sarna.
\newblock Hierarchical {Boltzmann} simulations and model error estimation.
\newblock {\em Journal of Computational Physics}, 342:66--84, 2017.

\bibitem{Torrilhon2008}
M.~Torrilhon and H.~Struchtrup.
\newblock Boundary conditions for regularized 13-moment-equations for micro-channel-flows.
\newblock {\em J. Comput. Phys.}, 227(3):1982--2011, 2008.

\bibitem{Wu2014}
L.~Wu, J.~M. Reese, and Y.~Zhang.
\newblock Solving the boltzmann equation deterministically by the fast spectral method: application to gas microflows.
\newblock {\em Journal of Fluid Mechanics}, 746:53–--84, 2014.

\bibitem{Agrawal2023aip}
U.~Yadav, A.~Jonnalagadda, and A.~Agrawal.
\newblock {Third-order accurate 13-moment equations for non-continuum transport phenomenon}.
\newblock {\em AIP Advances}, 13(4):045311, 04 2023.

\bibitem{yong2015cdf}
Y.~Zhu, L.~Hong, Z.~Yang, and W.~Yong.
\newblock Conservation-dissipation formalism of irreversible thermodynamics.
\newblock {\em Journal of Non-Equilibrium Thermodynamics}, 40(2):67--74, 2015.

\end{thebibliography}

\end{document}